\documentclass{article}
\usepackage[preprint]{neurips_2026}

\usepackage[utf8]{inputenc} 
\usepackage[T1]{fontenc}    
\usepackage{hyperref}       
\usepackage{url}            
\usepackage{booktabs}       
\usepackage{amsfonts}       
\usepackage{nicefrac}       
\usepackage{microtype}      
\usepackage{xcolor}         

\usepackage{amsmath}
\usepackage{amssymb}
\usepackage{mathtools}
\usepackage{amsthm}
\usepackage{enumitem}

\theoremstyle{plain}
\newtheorem{theorem}{Theorem}[section]
\newtheorem{proposition}[theorem]{Proposition}
\newtheorem{lemma}[theorem]{Lemma}
\newtheorem{corollary}[theorem]{Corollary}
\newtheorem{definition}[theorem]{Definition}

\theoremstyle{definition}

\theoremstyle{remark}
\newtheorem{remark}[theorem]{Remark}
\newtheorem{assumption}[theorem]{Assumption}

\newcommand{\mbE}{\mathbb{E}}
\newcommand{\mbS}{\mathbb{S}}
\newcommand{\mbC}{\mathbb{C}}

\newcommand{\mbP}{\mathbb{P}}

\newcommand{\mcF}{\mathcal{F}}

\newcommand{\mcL}{\mathcal{L}}
\newcommand{\mcM}{\mathcal{M}}
\newcommand{\mcN}{\mathcal{N}}
\newcommand{\mcP}{\mathcal{P}}

\newcommand{\mcR}{\mathcal{R}}
\newcommand{\mbR}{\mathbb{R}}
\newcommand{\mcS}{\mathcal{S}}

\newcommand{\U}{\mathcal{U}} 

\newcommand{\diff}{\mathrm{d}}

\DeclareMathOperator{\Cov}{Cov}
\DeclareMathOperator{\Var}{Var}
\DeclareMathOperator{\tr}{tr}
\DeclareMathOperator{\inj}{inj}
\DeclareMathOperator{\argmin}{argmin}
\DeclareMathOperator{\argmax}{argmax}

\newcommand{\Normal}{\mathcal{N}}
\newcommand{\IF}{\mathrm{IF}}

\title{Intrinsic Riemannian Cross-covariance for Manifold-valued Random Objects}

\author{%
    Carlos J. Soto\\
    Department of Mathematics and Statistics\\
    University of Massachusetts Amherst\\
    Amherst, MA 01002 \\
    \texttt{carlossoto@umass.edu} \\
    \And  
    Cheng Wang \\  
    Department of Mathematics and Statistics\\
    University of Massachusetts Amherst\\
    Amherst, MA 01002 \\
    \texttt{chengwang@umass.edu} \\
    \And
    Yujing Huang \\
    Department of Industrial and Systems Engineering\\
     State University of New York at Buffalo\\ 
    Buffalo, NY 14260 \\
    \texttt{yujinghu@buffalo.edu} \\
    \And
    Xiaoyu Chen \\
    Department of Industrial and Systems Engineering\\
    State University of New York at Buffalo\\ 
    Buffalo, NY 14260 \\
    \texttt{xchen325@buffalo.edu} \\
}

\begin{document}

\maketitle

\begin{abstract}
  Covariance estimation yields a fundamental second-order statistic underlying representation learning, dimension reduction, and dependence modeling. While covariance has been well understood in Euclidean spaces, it is ill-defined for random objects residing on nonlinear Riemannian manifolds, which increasingly arise in modern machine learning applications involving shapes, symmetric positive definite (SPD) matrices, etc. 
    This paper introduces an intrinsic Riemannian cross-covariance for manifold-valued random objects. Our approach defines covariance and correlation by transporting local variations to a common tangent space via parallel transport, yielding a second-order descriptor that is independent of arbitrary coordinate choices. We establish that the proposed covariance inherits desirable properties of its Euclidean counterparts and characterize its asymptotic behavior.
    Numerical studies on spheres and SPD manifolds, together with real-data experiments on heart valve shapes in Kendall's shape space, demonstrate the effectiveness of our estimators and verify the stated properties. Our results position the Riemannian covariance as a fundamental tool for second-order learning and analysis in non-Euclidean representation spaces. 
\end{abstract}

\section{Introduction}
Understanding the dependence between two random objects is a natural and important task underlying representation learning. This dependence can typically be quantified via the cross-covariance matrix along with subsequent analysis including a scalar covariance and correlation. The foundations for the correlation analysis, however, are grounded in the existence of a vector space structure. For non-Euclidean manifold-valued data, such vector space structures are not well-defined, preventing the natural definition of cross-covariance and correlation.

A Riemannian manifold is a generalization of Euclidean space which need not be flat. Data that naturally lives on Riemannian manifolds is often encountered in statistical analysis. Extending linear, statistical methods to non-linear spaces is a ubiquitous endeavor such as the Fr\'echet mean, the Riemannian center of mass, which extends the notion of the Euclidean mean to metric spaces \cite{frechet1948elements}. Data such as directional data \cite{mardia1975statistics}, meteorological data \cite{fu2025adaptive}, protein data \cite{mardia2018directional}, and discrete densities \cite{rao1945information}, 
can all be thought of as spherical data, a Riemannian manifold of constant, positive curvature. Other common examples include covariance matrices, for instance in brain imaging, as elements of the cone of symmetric positive definite matrices \cite{dai2019analyzing}, point cloud shape data as elements of complex projective space \cite{kendall1984shape}, and tree structured data \cite{lu2015clustering}. Analyzing such data requires extensions of the standard Euclidean statistical tools.

Due to the prevalence of non-Euclidean data in machine learning analyses and the need to quantifiably analyze such data, generalizations of the notion of covariance have been proposed. Some such generalizations include, to Riemannian manifolds for a single random variable \cite{pennec2006intrinsic}, to Wasserstein space \citep{petersen2019frechet}, to Riemannian functional data
\cite{LinIntrinsic,shao2022intrinsic}, and to Riemannian manifolds for more than one random variables sharing the same support \cite{abuqrais2024riemannian}.

The contributions of this paper are as follows:
\begin{enumerate}
    \item we introduce the \textit{intrinsic footpoint-invariant Riemannian cross-covariance}, a notion of covariance between two random variables that live on a Riemannian manifold, and subsequently define the \textit{Riemannian correlation};
    \item we elucidate the properties of each introduced estimate as well as establish their asymptotic behavior, and;
    \item we demonstrate the effectiveness of our estimates on simulated data on the sphere and the SPD manifold, as well as on experimental shape data via Kendall's shape space.
\end{enumerate}

\section{Background}
In the following, we introduce the necessary background on differential geometry and Riemannian manifolds. For additional details, we refer to \citet{do1992riemannian} and \citet{lee2018introduction}. After introducing the foundational geometric tools, we give background on the state-of-the-art covariance estimation for Riemannian manifolds.
\subsection{Riemannian Geometry}\label{app:Riem}
Let $\mcM$ be a $d-$dimensional smooth topological manifold. At every point $p\in\mcM$ there exists a neighborhood which is diffeomorphic to $\mbR^d$. Suppose $\alpha:[-\delta,\delta]\rightarrow\mcM$ is a smooth curve on $\mcM$ such that $\alpha(0)=p$ and denote by $\alpha'$ its derivative. The tangent space of $\mcM$ at $p$, denoted $T_p\mcM$, is the collection of all tangent vectors $\alpha'(0)$, that is $T_p\mcM=\{\alpha'(0)|\alpha(0)=p, \alpha:[-\delta,\delta]\rightarrow\mcM\}$. Each tangent space is a vector space with $\dim T_p\mcM=\dim\mcM$, and the collection of all tangent spaces is referred to as the tangent bundle, $T\mcM=\{(p,v)|v\in T_p\mcM,p\in\mcM\}$. We equip each tangent space with an inner product $g_p(\cdot,\cdot)=\langle\cdot,\cdot\rangle_{p}:T_p\mcM\times T_p\mcM\rightarrow \mbR$ which varies smoothly on $\mcM$. A Riemannian manifold is a smooth manifold equipped with a Riemannian metric, $(\mcM,g)$.

With respect to the metric, we have a unique torsion-free, affine connection $\nabla$, the Levi--Civita connection. This connection allows us to define the notion of parallel vector fields. Given a smooth path $\alpha:[0,1]\rightarrow\mcM$, the parallel transport of $v\in T_{\alpha(0)}\mcM$ is denoted as $\Gamma_{\alpha(0)}^{\alpha(1)}v$ and satisfies $\nabla_{\alpha'(t)}V(t)=0$ with $V$ being a smooth vector field with $V(\alpha(0))=v$. Equivalently the vector field $V$ must satisfy $\frac{D}{dt}V=0$ with respect to the covariant derivative. Since parallel transport moves tangent vectors in a parallel fashion, it preserves inner products; that is, parallel transport is an isometry, so for $v,w\in T_p\mcM$ we have $\langle v,w\rangle_{p}=\langle\Gamma_p^q v,\Gamma_p^q  w\rangle_{q}$. As a consequence, parallel transport preserves vector norms. Given a smooth path $\alpha:[0,1]\rightarrow\mcM$, with $v\in T_{\alpha(t)}\mcM$, one can show that there exists a unique parallel vector field $V$ such that $V(t)=v$ \citep[Theorem 4.32]{lee2018introduction}. This, however, is still dependent on the chosen path.

The length $\mcL(\cdot)$ of a smooth path $\alpha$ is the cumulative instantaneous rate of change, i.e., $\mcL(\alpha)=\int_0^1 \langle \alpha'(t),\alpha'(t)\rangle^{1/2}_{\alpha(t)}$. A geodesic is a path that minimizes the length between two points. That is, a geodesic is the solution to $$\gamma=\argmin_{\alpha:\alpha(0)=p,\alpha(1)=q}\mcL(\alpha).$$ This notion of minimal length results in a notion of distance; the distance between $p,q\in\mcM$ is defined as the length of the geodesic between said points. So, $d(p,q)=\mcL(\alpha)$ such that $\alpha$ is a geodesic with $\alpha(0)=p$ and $\alpha(1)=q$.


The exponential map maps elements from a tangent space to the manifold, that is,  $\exp_p:T_p\mcM\rightarrow\mcM$ such that $\exp_pv:=q$ such that there exists a geodesic $\alpha:[0,1]\rightarrow\mcM$ such that $\alpha(0)=p$ and $\alpha(1)=q$ and $\alpha'(0)=v$. The exponential map is invertible within a neighborhood of its footpoint $p$, this inverse map is referred to as the logarithm map. The logarithm map $\log_p:\mcM\rightarrow T_p\mcM$ produces a vector in the tangent space with direction and length of the input, $d(p,q)=\|\log_pq\|_p=\|\log_qp\|_q$.

As we require the exponential map to be invertible, we need to consider on which domain it is as such. Riemannian manifolds are locally diffeomorphic to a Euclidean space, and the injectivity radius formally quantifies this locality. 
The injectivity radius at $p\in\mcM$, denoted as $\inj(p)$, is the largest (supremum) radius of an open ball in $T_p\mcM$ centered at the origin such that $\exp_p$ is defined and $B_r\subset T_p\mcM$ and $\exp_p(B_r)$ are diffeomorphic. The injectivity radius of the manifold, $\inj{\mcM}$, is the smallest injectivity radius on the manifold, $\inj{\mcM}=\inf\{\inj(p)|p\in\mcM\}$.

The Riemannian volume form $\diff\omega$ is defined as $\diff\omega=\sqrt{|g|}\diff x^1\times \cdots\times \diff x^d$ where $|g|$ is the determinant of the matrix representation of the metric in local coordinates. The Riemannian volume form allows us to define densities with respect to this base measure.
 
The term ``footpoint" is not formally defined but generally refers to the point on the manifold at which the object of interest is defined. For instance the footpoint of $T_p\mcM$, $\log_p$, and $\exp_p$ is $p$.
For the remainder of this paper, unless otherwise stated, we assume $\mcM$ to be a complete, compact, finite-dimensional Riemannian manifold. For each manifold we require an exponential map, the logarithm map, and the parallel transport. These are standard tools, and we include them in \ref{app:ManNotes} for completeness.

\subsection{Related Literature - Covariance}\label{RL_Cov}

The Fr\'echet mean \citep{frechet1948elements} generalizes the notion of the arithmetic mean to metric spaces and is defined with respect to the Fr\'echet variance function.
The Fr\'echet variance of a random variable $X$ with probability measure $\lambda$ at $x\in\mcM$ is defined  as
\begin{align*}
    \sigma^2_X(x)&=\mbE( d^2(x,y))\\
    &=\int_{\mcM} d^2(x,y)\diff\lambda(y).
\end{align*}

The Fr\'echet mean is the point on the manifold that minimizes the variance, i.e., $$\mbE(X)=\mu\coloneq \argmin_{x\in\mcM}\sigma^2_X(x).$$ While the Fr\'echet variance is defined at any point $x$, if no point is designated, then the variance is assumed to be at the mean, i.e., $\sigma^2_X\coloneq \sigma^2_X(\mu)$. The Fr\'echet variance may not have a unique minimizer on general manifolds, namely on positively curved manifolds \cite{karcher1977riemannian,afsari2011riemannian}. To ensure the existence of a unique mean, one is required to bound the density in a geodesic ball centered at $p$ with radius $r\leq\frac{1}{2}\min\{\inj\mcM,\pi/\sqrt{\kappa}\}$, $B_r(p)$, \citep{karcher1977riemannian,kendall1990probability}, where $\inj\mcM$ denotes the injectivity radius and $\kappa$ is an upper bound on the sectional curvature.

\citet{pennec2006intrinsic} first extended the definition of covariance to Riemannian manifolds focusing on a single random variable. Suppose $X$ is a random variable as above with a unique mean $\mbE(X)=\mu$ then the covariance of $X$ is defined as 
\begin{align*}
    \Sigma_{X,X}^{Pennec}=\Cov_X(\mu):&=\mbE\left[(\log_{\mu}X)(\log_\mu X)^T\right] \\
    &= \int_\mcM (\log_{\mu}x)(\log_\mu x)^T d\lambda(x).
\end{align*}
Similar to the Euclidean setting, we have that the trace of the covariance is the Fr\'echet variance, that is, $\tr(\Sigma^{Pennec}_{X,X})=\sigma_X^2$. Given a dataset $D=\{x_1,\dots,x_n\}$, the empirical estimate of the covariance is $\hat{\Sigma}^{Pennec}_{X,X}=\frac{1}{n}\sum_i\left[(\log_{\hat{\mu}}x_i)(\log_{\hat{\mu}} x_i)^T\right]$ where $\hat{\mu}$ is the empirical estimate of the Fr\'echet mean. 

The previous definition, however, is limited to measuring the covariance of a single random variable. Extending the idea of covariance to more than one variable for complex spaces has been of recent interest.  \citet{abuqrais2024riemannian} propose an extension of the covariance definition to two random variables on Riemannian manifolds. To this end, the authors require that both densities have the same support. That is, given random variables $X$ and $Y$ defined on $\mcM$, their respective measures are defined, almost surely, in the same geodesic ball $B_r(p)$. They define the covariance as follows.

\begin{definition}\label{def:AbuCov}[\citet{abuqrais2024riemannian}]
   The covariance between random variances $X$ and $Y$ at point $p\in\mcM$ is 
   
       $$\Sigma^{A\&P}_{X,Y}(p)= \mbE\left[(\log_pX)(\log_pY)^T\right]-\mbE\left[\log_pX\right]\mbE\left[\log_pY\right]^T.$$
  
   
   The Riemannian covariance is defined as $$\text{Rcov}(p)(X,Y)\coloneq \tr(\Sigma^{A\&P}_{X,Y}(p)).$$ 
\end{definition}

Note that if $Y=X$ and $p=\mu$ then we have that $\Sigma^{A\&P}_{X,Y}(p)=\Sigma^{Pennec}_{X,X}$ and so $\text{Rcov}(\mu)(X,X)=\sigma^2_X.$ The empirical estimate in this case is $$\hat{\Sigma}^{A\&P}_{X,Y}(p)=\frac{1}{n}\sum \left[(\log_px_i)(\log_py_i)^T\right]-\left(\frac{1}{n}\sum\log_px_i\right)\left(\frac{1}{n}\sum\log_py_i\right)^T.$$
Definition \ref{def:AbuCov}, by construction, requires both measures to exist in the same geodesic ball; this can be a restrictive limitation. Further, the covariance is dependent on the choice of footpoint $p$.  \citet{abuqrais2024riemannian} give some guidance on how to choose $p$ such as selecting the midpoint of the geodesic connecting the two sample Fr\'echet means. We aim to address both of these limitations. The following proposition shows that such definition is not footpoint-invariant; the proof can be found in \ref{ss:AdProof}.

\begin{proposition}\label{prop:ap-variant}
    On a Riemannian manifold with non-zero curvature, the Abuqrais-Pigoli covariance $\text{Rcov}(p)(X,Y)$ is not invariant with respect to the choice of footpoint $p$.
\end{proposition}

\citet{petersen2019wasserstein} introduced a notion of covariance for univariate probability density functions endowed with the Wasserstein metric. 
Let $\mcF_1, \mcF_2$ denote two random densities with distribution functions $F_1, F_2$, and let $\bar f_i$ (with distribution function $\bar F_i$) be the Fr\'echet (Wasserstein) mean of $\mcF_i$. The optimal transport map from $\bar f_i$ to a realization $f_i$ is $T_i=F_i^{-1}\circ\bar F_i$, and its tangent representation $T_i-\mathrm{id}$ lives in the tangent space at $\bar f_i$ (formally $L^2(\bar f_i)$). The squared Wasserstein distance is induced from this inner product, $d^2(\bar f_1,\bar f_2)=\int(T(u)-u)^2\bar f_1(u)\,du$, where $T=\bar F_2^{-1}\circ \bar F_1$ is the optimal transport from $\bar f_1$ to $\bar f_2$.

\citet{petersen2019wasserstein} define the Wasserstein covariance as
\begin{align}\label{eq:wassersteinCov}
    \Sigma_{\mcF_1,\mcF_2}&=\mbE(\langle\widetilde{T}_1,T_2\rangle_{\bar{f}_2}) \nonumber \\
    &=\mbE\left(\int_{\mbR}(\widetilde{T}_1(u)-u)(T_2(u)-u)\bar{f}_2(u)du\right),
\end{align}
where $T_i=F_i^{-1}\circ{\bar{F}_i}$ is the optimal transport map and $\widetilde{T}_1=T_1\circ (\bar{F}^{-1}_1\circ \bar{F}_2)-(\bar{F}^{-1}_1\circ \bar{F}_2)+\text{id}$ 
is the parallel transport of the optimal transport map to the tangent space of $\bar{f}_2$.
\subsection{Related Literature - Correlation}
\citet{kim2014canonical} introduced a way to define the correlation as an extension of canonical correlation analysis to manifold data. We summarize their method here with more details in Appendix~\ref{app:cca}. Traditional CCA finds linear subspaces at the means of two multivariate random variables and finds projections that maximize the correlation between them. \citet{kim2014canonical} extend this idea by finding linear subspaces of tangent spaces. Let $\pi_S$ denote the projection operator onto geodesic submanifold (see Appendix~\ref{app:cca}) $S\subset \mcM$ defined as $\pi_S(x)=\argmin_{\tilde{x}\in S} d^2(x,\tilde{x})$. 

Let the data be $\{x_i,y_i\}_{i=1}^n$, $\mathbf{a}=\{a_i\}$. $\mathbf{b}=\{b_i\}$, with $\bar{a},\bar{b}$ the Euclidean means of $\mathbf{a}$ and $\mathbf{b}$, respectively, and $\mu,\nu$ the Fr\'echet means of $X$ and $Y$, respectively. Here $W_x,W_y$ are the basis tangent vectors and $\pi_{W_x}$ refers to the projection onto the space spanned by $W_x$. 
Following this idea, the generalized CCA is expressed as an optimization problem 
\begin{align*}
    \mcR^{CCA}_{X,Y}=&\text{corr}(\pi_{W_x}(x),\pi_{W_y}(y))\\
    =&\max_{W_x,W_y,\mathbf{a,b}}\frac{\sum_i(a_i-\bar{a})(b_i-\bar{b})}{\sqrt{\sum_i(a_i-\bar{a})^2}\sqrt{\sum_i(b_i-\bar{b})^2}}\\
    s.t.\quad &a_i=\argmin_{a_i\in(-\epsilon,\epsilon)}\|\log_{\exp_{\mu}(a_iW_x)}(x_i)\|^2\\
    &b_i=\argmin_{b_i\in(-\epsilon,\epsilon)}\|\log_{\exp_{\nu}(b_iW_y)}(y_i)\|^2\\
    &\forall i\in\{1,\dots,n\}.
\end{align*}

This formulation finds the tangent vectors which determine the geodesic submanifolds such that the correlation of projection coefficients is maximized. Once an appropriate $\mcR^{CCA}_{X,Y}$ is obtained, one can determine the covariance by multiplying by the standard deviation of both set of coefficients. 

\citet{abuqrais2024riemannian} also propose a generalization for a Riemannian correlation defined in terms of their covariance. They define the correlation matrix as 
$$\varrho^{A\&P}_{X,Y}(p)=\frac{\Sigma^{A\&P}_{X,Y}(p)}{\sqrt{\tr(\Sigma^{A\&P}_{X,X}(p))\tr(\Sigma^{A\&P}_{Y,Y}(p))}},$$ where $\Sigma^{A\&P}_{X,Y}(p)$ is as in Definition~\ref{def:AbuCov} and the correlation is $\mcR^{A\&P}_{X,Y}(p)=\tr(\varrho^{A\&P}_{X,Y}(p))$. The authors show that this satisfies the crucial property that $\mcR^{A\&P}_{X,Y}(p)\in[-1,1]$ \citep[Proposition 3.5]{abuqrais2024riemannian}.
\section{Intrinsic Footpoint-invariant Riemannian Cross-covariance}
Proofs for Section~\ref{ss:RCov} and Section~\ref{ss:RCorr} can be found in Appendix~\ref{ss:AdProof} and proofs for Section~\ref{ss:TheoGuar} can be found in Appendix~\ref{app:consproofs}.
\subsection{Riemannian Covariance}\label{ss:RCov}
Suppose we have two random variables $X,Y$ with unique, respective Fr\'echet means $\mu,\nu$. To ensure we have unique means we require the following assumption. Proofs for this section can be found in \ref{ss:AdProof}.

\begin{assumption}\label{A1}
    To ensure we have unique Fr\'echet means, each dataset $\{Z_1,Z_2,\dots,Z_n\}$ is bounded in an open ball $B_r(p)$ with $r\leq\frac{1}{2}\min\{\inj\mcM,\pi/\sqrt{\kappa}\}$ \citep{karcher1977riemannian,kendall1990probability}, where $\inj\mcM$ denotes the injectivity radius and $\kappa$ is an upper bound on the sectional curvature. If $\kappa\leq 0$, then $r$ must simply be finite.
\end{assumption}

Further, our definition \ref{def:OurCov} requires the use of parallel transport. This raises an additional issue of uniqueness. We elaborate on this point in \ref{app:Riem}, but note that parallel transport is generally not unique unless it is restricted to geodesics. Thus, we have, 

\begin{assumption}\label{A2}
    To ensure a unique parallel transport, we require a unique geodesic between the Fr\'echet means. Thus, we assume $d(\mu,\nu)<\inj\mcM$.
\end{assumption}

\begin{definition}[Riemannian covariance]\label{def:OurCov}
    Given random variables $X$ and $Y$ with respective Fr\'echet means $\mu$ and $\nu$, let the Riemannian cross-covariance at $p$ be defined as $$\Sigma_{X,Y}(p)=\mbE\left[(\Gamma_\mu^p\log_\mu X)(\Gamma_\nu^p\log_\nu Y)^T\right]$$ where the parallel transport is with respect to the geodesic $\alpha$, $p\in\alpha$, $\alpha:[0,1]\rightarrow\mcM$, $\alpha(0)=\mu$, and $\alpha(1)=\nu$.
    The Riemannian covariance between $X$ and $Y$ at point $p$ is thus $$\Cov_{X,Y}(p)=\tr(\Sigma_{X,Y}(p)).$$  
\end{definition}

To motivate our definition, we have the following theorem which shows that our definition is a direct generalization of the Euclidean covariance.
\begin{theorem}\label{thm:ReduceToEuclid}
    When $\mcM=\mbR^n$, the Riemannian covariance reduces to the standard Euclidean covariance. 
\end{theorem}
\begin{remark}
In Theorem~\ref{thm:ftinvar} we show that our definition is invariant to choice of $p$. When there is no canonical or reference point choice for $p$, we define the covariance as
$$\Sigma_{X,Y}=\mbE\left[(\Gamma_\mu^\nu\log_\mu X)(\log_\nu Y)^T\right],\qquad \Cov_{X,Y}=\tr(\Sigma_{X,Y}).$$

In addition, our Riemannian covariance connects to the Wasserstein covariance (Eq.~\ref{eq:wassersteinCov}) by defining an operator-valued covariance in a common tangent space, whose trace yields a scalar dependence measure. In the one-dimensional Wasserstein setting, the covariance is defined directly as an 
$L^2$ inner product, which coincides with the trace of this operator under the flat geometry.
\end{remark}

\begin{theorem}[Footpoint Invariance]\label{thm:ftinvar}
    We have $\Cov_{X,Y}(p)=\Cov_{X,Y}(q)$ for all $p,q\in\alpha$ where $\alpha$ is the geodesic joining $\mu$ and $\nu$.
\end{theorem}

Our definition is similar to that proposed by \citet{abuqrais2024riemannian}, however, we utilize a parallel transport to achieve a common tangent space rather than utilize the tangent space at some midpoint. The use of parallel transport is necessary as $\log_{\mu*}$ and $\log_{\nu*}$ push the densities forward onto different tangent spaces, so the outer (and inner) product is not well-defined. This issue is not present in the Euclidean setting, as the tangent space at every point is the space itself, $T_p\mbR^d\cong \mbR^d$. Further, note that the vectors $\Gamma_\mu^p\log_\mu X$ and $\Gamma_\nu^p\log_\nu Y$ have length equal to the distance between the data and their respective mean: $\|\Gamma_\mu^p\log_\mu X\|_p=\|\log_\mu X\|_\mu=d(\mu,X)$. So, our definition is a \textit{centered} second moment. Definition \ref{def:AbuCov} centers the second moment by subtracting the ``mean" e.g., $\mbE\left[\log_pX\right]$. However, $\mbE\left[\log_pX\right]\ne\log_p\mbE\left[X\right]$, generally, i.e., the mean of the pushforward density onto the tangent space at an arbitrary point is not necessarily the $\log$ of the Fr\'echet mean of the density. 

\begin{proposition} \label{prop:prop1}
Let $\mathcal{H}$ denote the Hilbert space of Hilbert-Schmidt operators on $T_\nu M$, equipped with the inner product $\langle A,B\rangle_\mathcal{H}:=\mathrm{tr}(AB^{\ast})$. Then the Riemannian covariance between $X$ and $Y$ equals the Hilbert-Schmidt
inner product between the transported intrinsic covariance operator $\widetilde{C}_{X,Y} =\mathbb{E}\bigl[(\Gamma^{\nu}_{\mu}\log_\mu X)\otimes\log_\nu Y\bigr] \in\mathcal{H}$ and the identity operator $\mathrm{id}_{T_\nu M}$.
\end{proposition}

We characterize the behavior of our estimator in Proposition~\ref{prop:properties}. The proposed Definition~\ref{def:OurCov} satisfies some standard results comparable to those of the Euclidean covariance such as generalization of the variance, symmetry, and equaling $0$ for independent random variables.

\begin{proposition}\label{prop:properties}
The Riemannian Covariance as in \ref{def:OurCov} has the following properties:
    \begin{enumerate}
        \item 
        $\Cov_{X,X}(\mu)=\sigma^2_X$ (Fr\'echet variance special case);
        \item $\Cov_{X,Y}(p)=\Cov_{Y,X}(p)$ (Symmetry);
        \item if $\Cov_{X,X}(\mu)=0$, then $X$ is degenerate (Degeneracy);
        \item  if $X$ and $Y$ are independent, then $\Cov_{X,Y}(p)=0$ (Uncorrelatedness);
        \item for a fixed point $q\in\mcM$, $\Cov_{X,q}(q)=0$ (Variance of a constant);
        \item $\Cov_{(a,X),Y}(p)=a\Cov_{X,Y}(p)$ where $(a,X)$ is $\exp_\mu(a\log_\mu(X))$ and $a\in\mbR$ (Scaling).
    \end{enumerate}
\end{proposition}
\begin{remark}
    As multiplication is not defined on Riemannian manifolds, item~6 of Proposition~\ref{prop:properties} requires us to define what it means to scale a random variable. We propose a definition and mention the implications with regard to our assumptions in the proof. Further, items~6 and~2 in conjunction with~1 yield $\Cov_{(a,X),(a,X)}(\mu)=a^2\Cov_{X,X}(\mu)=a^2\sigma^2_X$.
\end{remark}

\subsection{Riemannian Correlation}\label{ss:RCorr}

Given the definition of Riemannian covariance, we introduce the notion of Riemannian correlation. This notion generalizes the Pearson correlation from Euclidean spaces to the manifold setting. 

\begin{definition}[Riemannian correlation]\label{def:corr}
    Let $X$ and $Y$ be random variables with respective means $\mu$ and $\nu$ satisfying Assumption \ref{A2}. Let $\alpha$ be the geodesic with $\alpha(0)=\mu$ and $\alpha(1)=\nu$. For some $p\in\alpha$, the Riemannian cross-correlation at point $p$ is defined as $$\varrho_{X,Y}(p)=\frac{\Sigma_{X,Y}(p)}{\sqrt{\Cov_{X,X}(\mu)}\sqrt{\Cov_{Y,Y}(\nu)}},$$
    and the corresponding Riemannian correlation coefficient is given by $$\mcR_{X,Y}(p)=\tr(\varrho_{X,Y}(p))=\frac{\Cov_{X,Y}(p)}{\sqrt{\Cov_{X,X}(\mu)}\sqrt{\Cov_{Y,Y}(\nu)}}.$$
\end{definition}

Here, $\Cov_{X,X}(\mu)$ and $\Cov_{Y,Y}(\nu)$ are computed at $\mu$ and $\nu$, respectively. However, benefiting from invariance under parallel transport, these quantities can be evaluated at any point. The properties of the Riemannian correlation, Proposition~\ref{prop:corrproperties}, parallel those of the Riemannian covariance (i.e., Proposition~\ref{prop:properties}). A fundamental property of Riemannian correlation is its boundedness. Therefore, we have the following result, which is analogous to its Euclidean counterpart.
\begin{proposition}\label{prop:RieCor}
    We have that $\mcR_{X,Y}(p)\in [-1,1]$.
\end{proposition}

The properties of the Riemannian correlation, Proposition~\ref{prop:corrproperties}, parallel those of the Riemannian covariance (i.e., Proposition~\ref{prop:properties}).

\begin{proposition}\label{prop:corrproperties}
    The Riemannian correlation coefficient as in \ref{def:corr} has the following properties:
    \begin{enumerate}
        \item $\mcR_{X,Y}(p)=\mcR_{X,Y}(q)$ (Invariance);
        \item $\mcR_{X,Y}(p)=\mcR_{Y,X}(p)$ (Symmetry);
        \item If $X$ and $Y$ are independent, then $\mcR_{X,Y}(p)=0$ (Uncorrelatedness);
        \item $\mcR_{(a,X),Y}(p)=\text{sign}(a)\mcR_{X,Y}(p)$ (Scale invariance);
        \item If X is nondegenerate, $\mcR_{X,X}(p)=1$ (Self-correlation).
    \end{enumerate}
\end{proposition}

\subsection{Consistency and Asymptotic Normality}\label{ss:TheoGuar}
Given i.i.d.\ samples $\{(X_i,Y_i)\}_{i=1}^n$ and sample Fr\'echet means $\hat\mu,\hat\nu$, we first define the plug-in estimator
\begin{align*}
\widehat{\Sigma}_{X,Y}(p)\coloneq 
\frac1n\sum_{i=1}^n \hat U_i \hat V_i^\top,
\end{align*}
where $\hat U_i\coloneq \Gamma_{\hat\mu}^p \log_{\hat\mu}(X_i)$ and $\hat V_i\coloneq \Gamma_{\hat\nu}^p \log_{\hat\nu}(Y_i)$ at point $p$.
We also define the scalar covariance
\(
\Cov_{X,Y}(p)\coloneq \tr(\Sigma_{X,Y}(p))
\)
and its estimator $\widehat\Cov_{X,Y}(p)\coloneq \tr(\widehat\Sigma_{X,Y}(p))$.

\begin{assumption}[Regularity conditions]
    We impose the following regularity conditions:
        \begin{enumerate}[label=(A\arabic*)]
            \item 
            There exists a geodesically convex open set $\U\subset\mcM$ containing $\mu,\nu,p$ such that:
            (i) $\U$ is contained in a normal neighborhood of each point in $\U$,
            (ii) sectional curvature on $\U$ is bounded, and
            (iii) the injectivity radius on $\U$ is bounded away from $0$.
            
            \item 
            $\mbE\|\log_\mu(X)\|_\mu^2<\infty$ and $\mbE\|\log_\nu(Y)\|_\nu^2<\infty$.
            
            \item 
            $\mathbb{P}(X \in \U,\, Y \in \U) \geq 1 - \delta$ for some $\delta \geq 0$ (so the argument applies after truncation).
            
            \item
            $\mu$ and $\nu$ are unique Fr\'echet means and the population Fr\'echet variance functions are twice continuously differentiable on $\U$.
            Let
            \[
            \begin{aligned}
            H_X &\coloneq  \nabla^2 \Big( x\mapsto \mbE[d^2(X,x)] \Big)\Big|_{x=\mu},\\
            H_Y &\coloneq  \nabla^2 \Big( y\mapsto \mbE[d^2(Y,y)] \Big)\Big|_{y=\nu}.
            \end{aligned}
            \]
            Assume $H_X$ and $H_Y$ are positive definite.
            
            \item
            $\mbE\|\log_\mu(X)\|_\mu^4<\infty$ and $\mbE\|\log_\nu(Y)\|_\nu^4<\infty$.
            
            \item
            Define transport maps into $T_p\mcM$:
            \[
            \phi(x)\coloneq \Gamma_\mu^p\log_\mu(x),\qquad
            \psi(y)\coloneq \Gamma_\nu^p\log_\nu(y),
            \]
            and their plug-in counterparts
            \[
            \hat\phi_{\hat\mu}(x)\coloneq \Gamma_{\hat\mu}^p\log_{\hat\mu}(x),\qquad
            \hat\psi_{\hat\nu}(y)\coloneq \Gamma_{\hat\nu}^p\log_{\hat\nu}(y).
            \]
            Assume $\phi,\psi$ are $C^2$ on $\U$ and the maps
            $m\mapsto \Gamma_m^p\log_m(z)$ are $C^2$ jointly in $(m,z)\in\U\times\U$ with uniformly bounded first derivatives on $\U\times\U$.
            (Under (A1) this holds for the Levi--Civita connection on compact subsets of $\U$.)
            \end{enumerate}
\end{assumption}

\begin{remark}
Assumptions (A1)--(A4) are the standard regularity conditions under which sample Fr\'echet means satisfy strong consistency and $\sqrt{n}$-asymptotics.
Assumptions (A5)--(A6) ensure that the plug-in covariance admits a first-order expansion.
\end{remark}

Under these conditions, the following two lemmas will be used to establish the consistency and asymptotic properties of the estimators.

\begin{lemma}[Strong consistency of sample Fr\'echet means (\citealp{bhattacharya2003large}, Thm.~2.3)]\label{lem:mean_consistency}
Under (A1)--(A4), $\hat\mu\to \mu$ almost surely and $\hat\nu\to\nu$ almost surely.
\end{lemma}

\begin{lemma}[$\sqrt{n}$ expansion of sample Fr\'echet means (\citealp{bhattacharya2005large}, Thm.~2.1)]\label{lem:mean_clt}
Under (A1)--(A5), in $T_\mu\mcM$ and $T_\nu\mcM$ respectively,
\begin{align}
\sqrt{n}\,\log_\mu(\hat\mu) &= H_X^{-1}\,\frac{1}{\sqrt{n}}\sum_{i=1}^n \xi_i + o_p(1),\label{eq:mu_clt}\\
\sqrt{n}\,\log_\nu(\hat\nu) &= H_Y^{-1}\,\frac{1}{\sqrt{n}}\sum_{i=1}^n \zeta_i + o_p(1),\label{eq:nu_clt}
\end{align}
where $\xi_i\in T_\mu\mcM$ and $\zeta_i\in T_\nu\mcM$ are i.i.d.\ mean-zero random vectors with finite second moments.
Moreover, $\frac1{\sqrt{n}}\sum\xi_i \overset{d}{\to} \Normal(0,\Omega_X)$ and $\frac1{\sqrt{n}}\sum\zeta_i \overset{d}{\to} \Normal(0,\Omega_Y)$.
\end{lemma}

We then prove the strong consistency of the plug-in covariance estimator $\widehat{\Sigma}_{X,Y}(p)$ for a fixed footpoint $p\in\U$ and, hence, of its scalar trace $\widehat{\Cov}_{X,Y}(p)$. This almost-sure convergence justifies treating $\widehat{\Sigma}_{X,Y}(p)$ as a consistent plug-in quantity in the subsequent $\sqrt{n}$-asymptotic expansion.

\begin{theorem}[Strong consistency]\label{thm:Sigma_consistency}
Under (A1)--(A4) and (A6), for any fixed $p\in\U$,
\[
\widehat{\Sigma}_{X,Y}(p) \xrightarrow{a.s.} \Sigma_{X,Y}(p).
\]
Consequently, $\widehat\Cov_{X,Y}(p)\to \Cov_{X,Y}(p)$ almost surely.
\end{theorem}

\begin{lemma}[Asymptotic linear representation of $\widehat\Sigma_{X,Y}(p)$]\label{lem:Sigma_if}
Assume (A1)--(A6) and fix $p\in\U$.
Let $\mathrm{vec}(\cdot)$ stack columns of a $d\times d$ matrix into a $d^2$-vector and define
\[
\Delta_n\coloneq \mathrm{vec}(\widehat\Sigma_{X,Y}(p))-\mathrm{vec}(\Sigma_{X,Y}(p)).
\]
Then
\begin{equation}\label{eq:final_if}
\sqrt{n}\,\Delta_n
=
\frac1{\sqrt{n}}\sum_{i=1}^n \IF_i + o_p(1),
\end{equation}
where
\begin{equation}\label{eq:IF}
\IF_i
\coloneq 
\Big(\mathrm{vec}(Z_i)-\mbE\mathrm{vec}(Z_i)\Big)
+
A_\mu H_X^{-1}\xi_i
+
A_\nu H_Y^{-1}\zeta_i.
\end{equation}
Here, $Z_i\coloneq U_iV_i^\top$ with $U_i=\Gamma_\mu^p\log_\mu(X_i)$ and $V_i=\Gamma_\nu^p\log_\nu(Y_i)$,
and $A_\mu,A_\nu$ and $\xi_i,\zeta_i$ are as defined in \eqref{eq:defA_mu}--\eqref{eq:defA_nu} and Lemma~\ref{lem:mean_clt}.
\end{lemma}

\begin{theorem}[$\sqrt{n}$-asymptotic normality]\label{thm:Sigma_clt}
Under (A1)--(A6), for any fixed $p\in\U$,
\[
\sqrt{n}\Big(\mathrm{vec}(\widehat\Sigma_{X,Y}(p))-\mathrm{vec}(\Sigma_{X,Y}(p))\Big)
\overset{d}{\to}
\Normal(0,\Xi),
\]
where
\begin{equation}\label{eq:Xi}
\Xi
=
\Var(\IF_1)
=
\Var\Big(
\mathrm{vec}(Z_1)
+ A_\mu H_X^{-1}\xi_1
+ A_\nu H_Y^{-1}\zeta_1
\Big).
\end{equation}
\end{theorem}

Since $\widehat\Cov_{X,Y}(p)=\tr(\widehat\Sigma_{X,Y}(p))$ is a linear functional of $\widehat\Sigma_{X,Y}(p)$, the delta method is immediate. We hence have the corollary below stating the CLT for the scalar covariance.

\begin{corollary}[CLT for scalar covariance]\label{cor:trace_clt}
Under the conditions of Theorem~\ref{thm:Sigma_clt},
\[
\sqrt{n}\Big(\widehat\Cov_{X,Y}(p)-\Cov_{X,Y}(p)\Big)
\overset{d}{\to}
\Normal\big(0,\ \sigma^2_{\tr}\big),
\]
where $\sigma^2_{\tr}=\Var\big(\langle \mathrm{vec}(I_d), \IF_1\rangle\big)$ and $I_d$ is the $d\times d$ identity matrix.
\end{corollary}

\begin{remark}[Plug-in variance estimation and Wald inference]
One can estimate $\Xi$ by the sample covariance of estimated influence terms:
\[
\widehat\Xi
=
\frac1n\sum_{i=1}^n \hat\IF_i\hat\IF_i^\top,
\]
where $\hat\IF_i$ replaces $(\mu,\nu,H_X,H_Y,A_\mu,A_\nu,\xi_i,\zeta_i)$ by consistent estimators
(e.g.\ empirical Hessians and empirical gradients for the Fr\'echet loss, and empirical averages for $A_\mu,A_\nu$).
Then asymptotic Wald-type confidence regions follow:
\[
\mathrm{vec}(\widehat\Sigma_{X,Y}(p)) \pm z_{\alpha/2}\, \frac{1}{\sqrt{n}}\sqrt{\mathrm{diag}(\widehat\Xi)}.
\]
For the trace, use $\hat\sigma^2_{\tr}=\mathrm{vec}(I_d)^\top \widehat\Xi\,\mathrm{vec}(I_d)$.
\end{remark}

\begin{remark}[Correlation]
    The consistency and asymptotic normality properties of Riemannian cross-correlation $\hat\varrho_{X,Y}(p)$ and Riemannian correlation coefficient $\hat\mcR_{X,Y}(p)$ follow directly. This is since they are simply the compositions of continuous functions of $\widehat\Sigma_{X,Y}(p)$ and $\widehat\Cov_{X,Y}(p)$.
\end{remark}
\section{Examples}\label{ss:examp}
For the following simulations, we utilize code from ManOpt \cite{JMLR:v17:16-177} and Geomstats \cite{miolane2020geomstats}.

\subsection{Sphere}
We first present some simulation studies on unit sphere $\mbS^2$ to illustrate the behavior of Riemannian covariance and correlation, see \ref{app:spheredetail} for more details. The first dataset is sampled from a truncated uniform distribution. A dependent dataset is then generated by rotating the first dataset, adding noise and transporting it to another point on the sphere, as detailed in \ref{sphere:simu}.

We compare our proposed Riemannian correlation with estimates from \citet{abuqrais2024riemannian} in different settings. Unless otherwise specified, we evaluate Riemannian correlation in the tangent space of the Fr\'echet mean of the dependent dataset and evaluate $A\&P$ correlation in the tangent space of the joint Fr\'echet mean, i.e. the midpoint of the geodesic connecting the two Fr\'echet means.

\begin{figure}[h]
    \centering
    {\includegraphics[height=3.2cm]{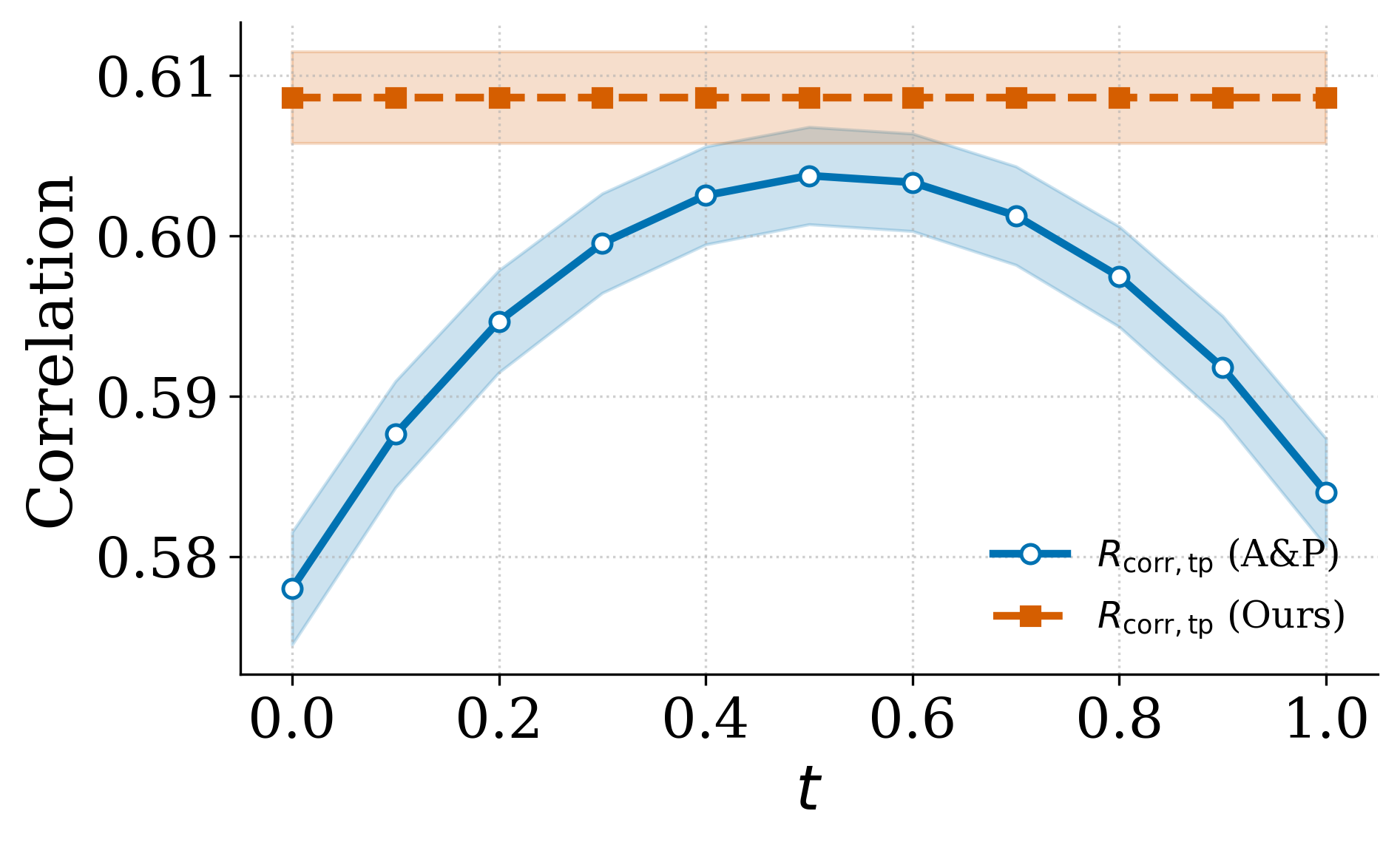}}
     \caption{Geodesic Evaluation on Sphere. The shaded area indicates 95\% confidence intervals (mean $\pm$ 2SE). The x-axis represents the location where the evaluation is taken along the geodesic connecting sample Fr\'echet means. Parameters: $n=200$, $\tau=\pi/6$, $\eta=\pi/4$, $q=(0,-1,0)$, $\sigma_\epsilon=0.15$, MC $= 200$.}
    \label{fig:eval_geo}
\end{figure}

For Figure \ref{fig:eval_geo}, we transport the dependent set to point $(0,-1,0)$ and evaluate the correlation estimate along the geodesic connecting Fr\'echet means of two datasets. The shaded area denotes the $95\%$ confidence interval based on two standard error. As shown in the plot, our estimate is numerically consistent along the geodesic, whereas the $A\&P$ correlation has an inverted U-shaped pattern with the largest estimate at the joint Fr\'echet mean. This demonstrates that the $A\&P$ correlation is subject to the choice of footpoint, as discussed in Section~\ref{RL_Cov}, while our method is footpoint-invariant.

\begin{figure}[h]
    \centering
    \begin{tabular}{@{}c@{} @{}c@{} @{}c@{}}
         \includegraphics[width=0.3\linewidth]{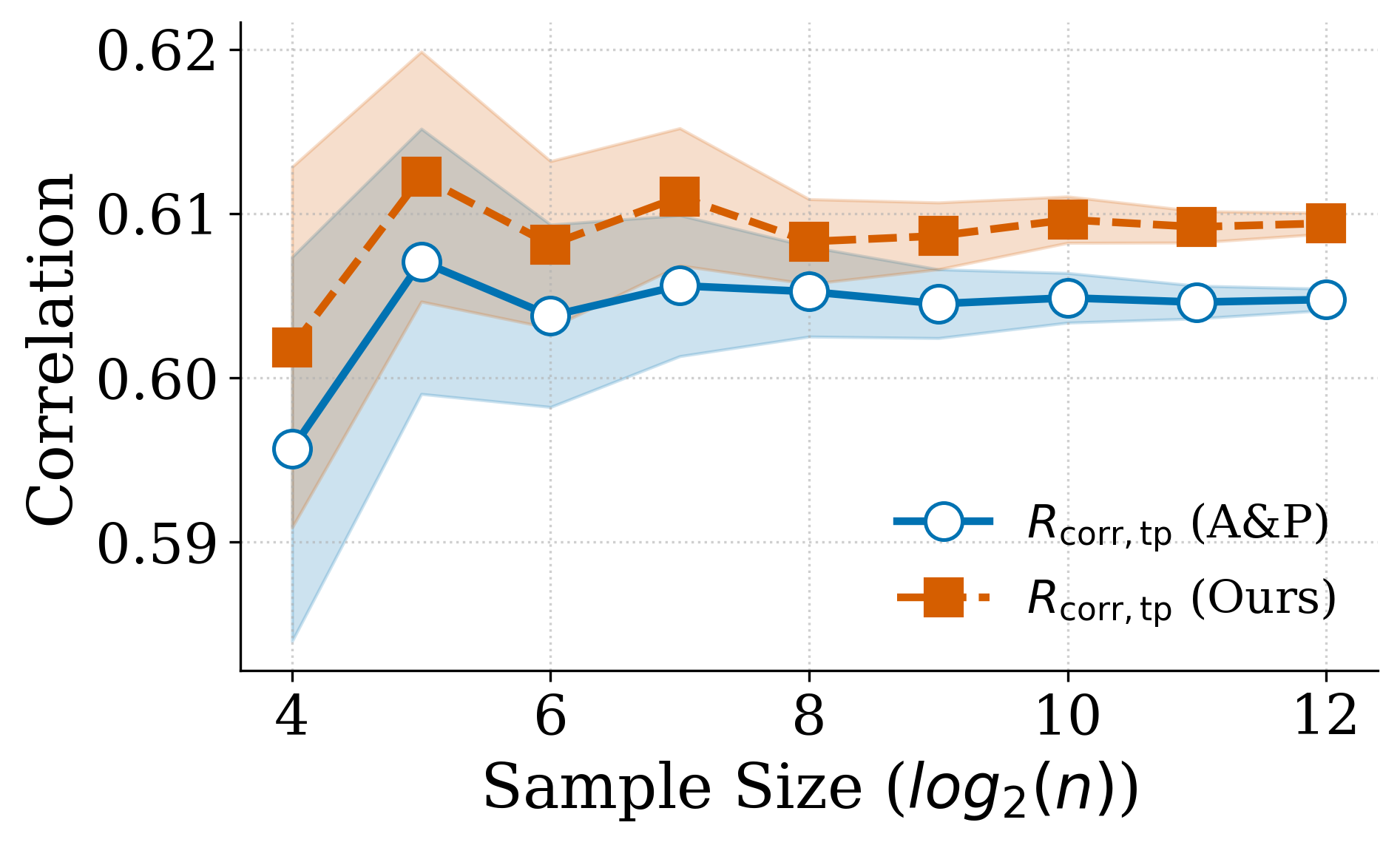}&
         \includegraphics[width=0.3\linewidth]{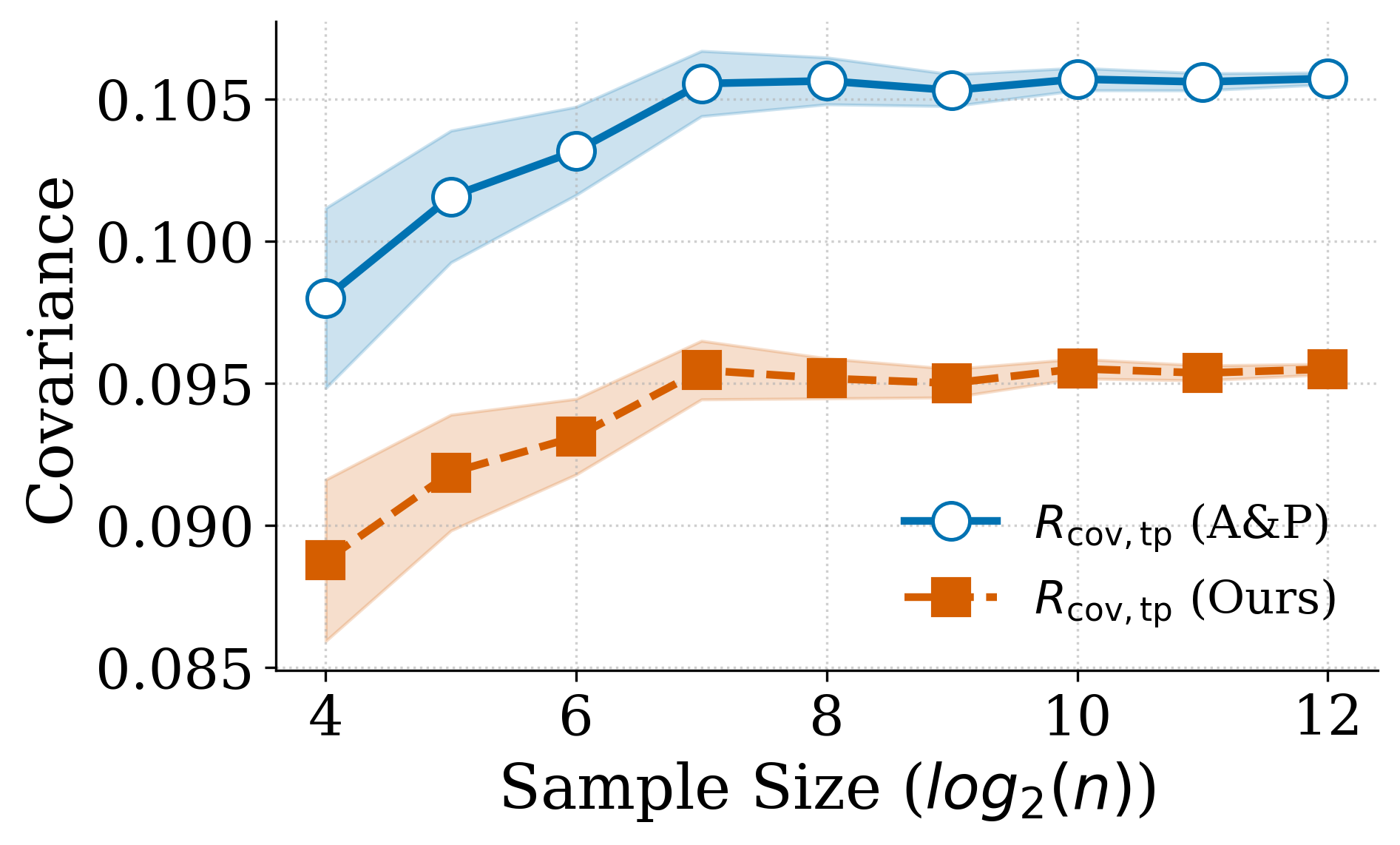}&
         \includegraphics[width=0.3\linewidth]{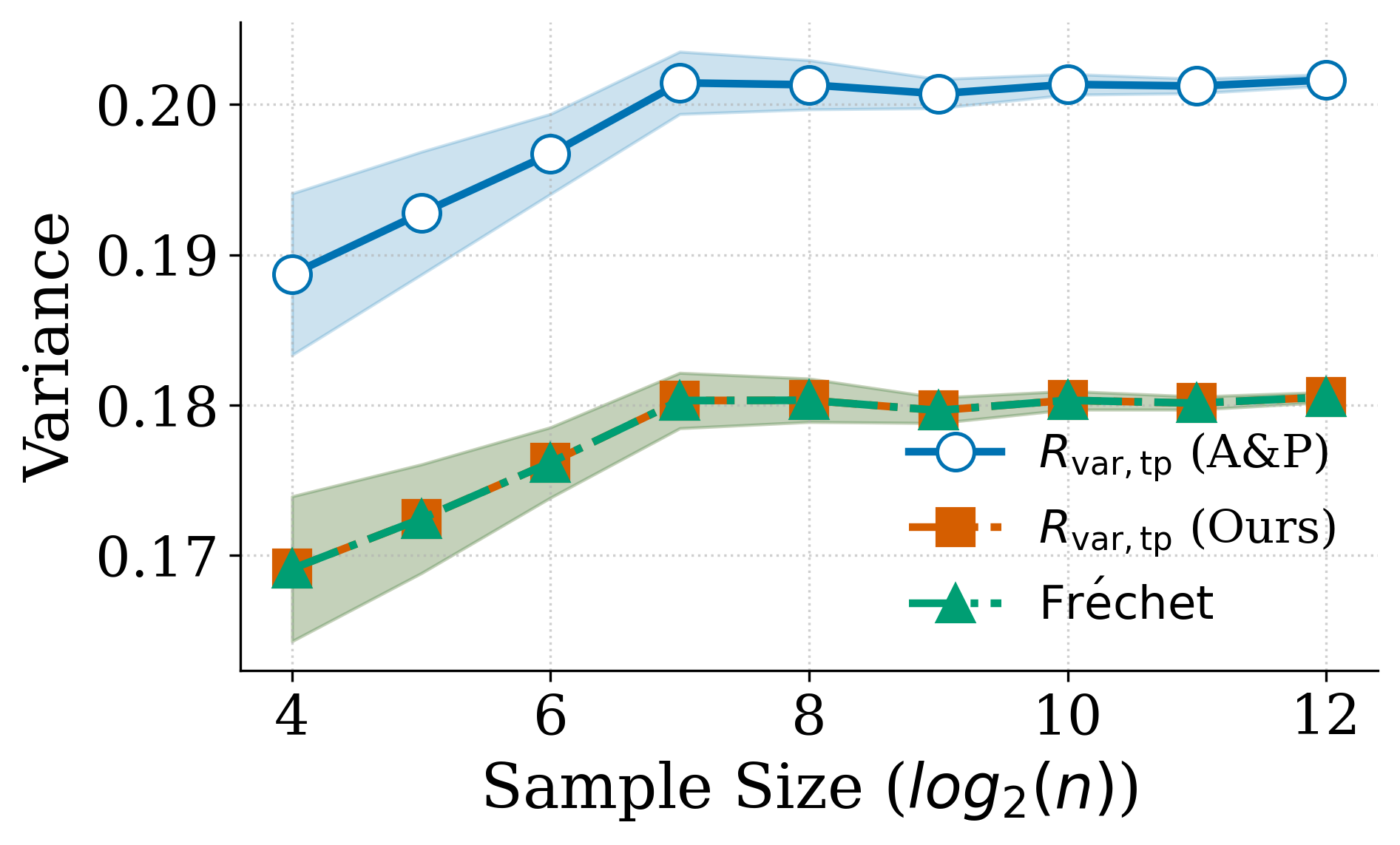}
    \end{tabular}
    \caption{Comprehensive Comparison on Sphere with parallel-transported dataset. Correlation, covariance, and variance estimates versus sample size. Shaded regions indicate 95\% confidence intervals. Parameters: $\tau=\pi/6$, $\eta=\pi/4$, $q=(0,-1,0)$, $\sigma_\epsilon=0.15$, MC $=200$.}
    \label{fig:comparison}
\end{figure}

Figure \ref{fig:comparison} compares the correlation and covariance estimates from our method and those from $A\&P$ correlation. We transport the dependent set to the point $(0,-1,0)$ and evaluate both methods at the joint Fr\'echet mean, with different sample sizes. As before, the shaded region represents the confidence interval corresponding to two standard errors. The left plot shows the estimates of correlation, where both methods exhibit convergence as the sample size increases. Compared to our approach, $A\&P$ correlation yields a slightly smaller result. The two panels in the middle and right are the corresponding covariance and variance estimates used in computing the correlation. For the variance, we also include the estimate of Fr\'echet variance, computed using Fr\'echet mean of transported set instead of the joint Fr\'echet mean. It can be seen that our method matches the Fr\'echet variance exactly, which is expected since parallel transport allows evaluation at the mean at all times. In contrast, the $A\&P$ method tends to produce a higher estimate, and could potentially lead to overestimation and bias.

\subsection{Symmetric Positive Definite Matrices}

We extend the simulation study to the manifold of Symmetric Positive Definite (SPD) matrices, $\mcP_d$, endowed with the Affine Invariant Riemannian Metric (AIRM), see Appendix~\ref{app:spdmdetail} for more details. This setting, characterized by non-positive (and non-constant) sectional curvature, provides a rigorous testbed for evaluating the proposed estimators.
The data generation process, detailed in Appendix \ref{app:spd_sim}, constructs correlated datasets by parallel transporting tangent vectors from a source distribution to a target location, thereby preserving the intrinsic geometric structure while varying the ambient configuration.

The fundamental property of footpoint invariance is first examined. Figure \ref{fig:spd_geodesic_eval} illustrates the trajectory of correlation estimates as the evaluation point $p$ traverses the geodesic connecting the sample means, for both the original dataset (non-transport) and the parallel-transported dataset (transport). We include this figure in Appendix~\ref{app:spd_sim} as the results are similar to the spherical case in Figure~\ref{fig:eval_geo}. The proposed Riemannian correlation yields estimates with low variation (standard error $SE \approx 0.008$) across all 21 evaluation points, which is consistent with footpoint invariance (Theorem \ref{thm:ftinvar}). In contrast, the A\&P estimator exhibits substantial variation ($SE \approx 0.15$), revealing a systematic bias that depends on the evaluation point. This dependence arises because the logarithmic map $\log_p(\cdot)$ at varying footpoints distorts the metric structure differently in the presence of negative curvature. Our transport-based formulation mitigates this flaw by aligning tangent spaces via parallel transport, which preserves inner products under isometry. The comparison between the two panels demonstrates that our method produces nearly identical estimates for both original and parallel-transported datasets, which is consistent with the estimator capturing the intrinsic geometry invariant under isometry.

\begin{figure}[h]
    \centering
    \begin{tabular}{@{}c@{} @{}c@{} @{}c@{}}
        \includegraphics[width=0.3\linewidth]{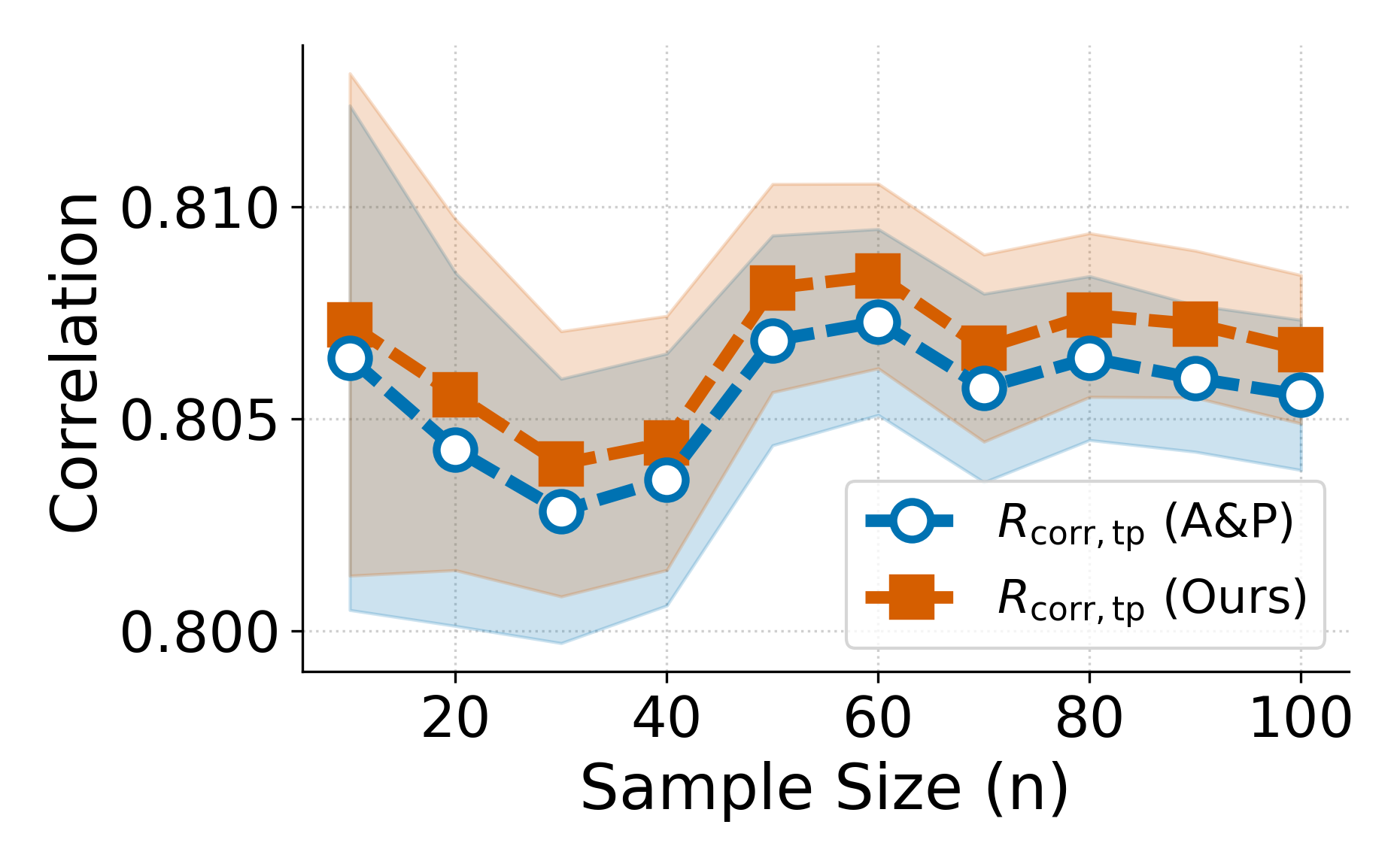}&
        \includegraphics[width=0.3\linewidth]{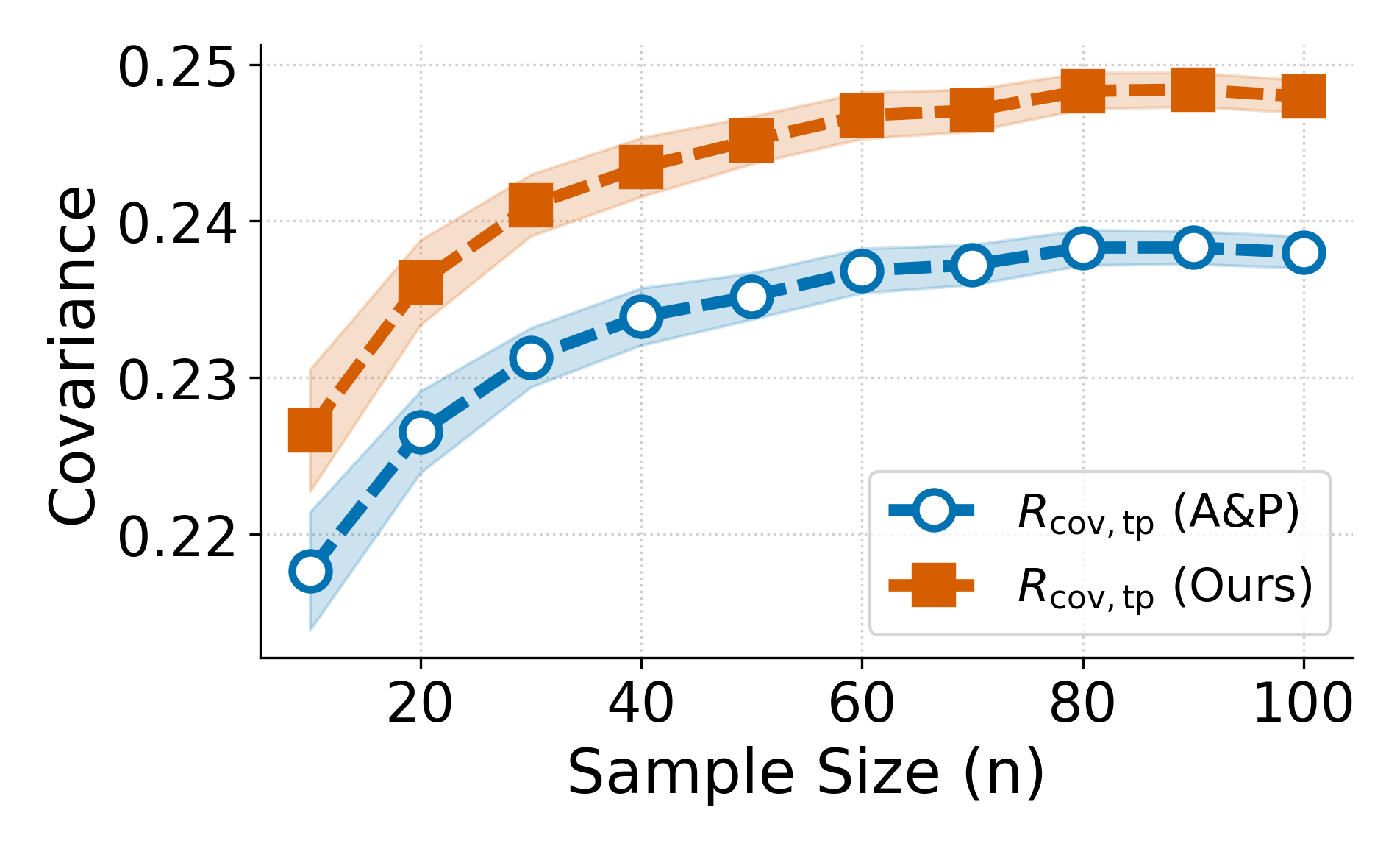}&
        \includegraphics[width=0.3\linewidth]{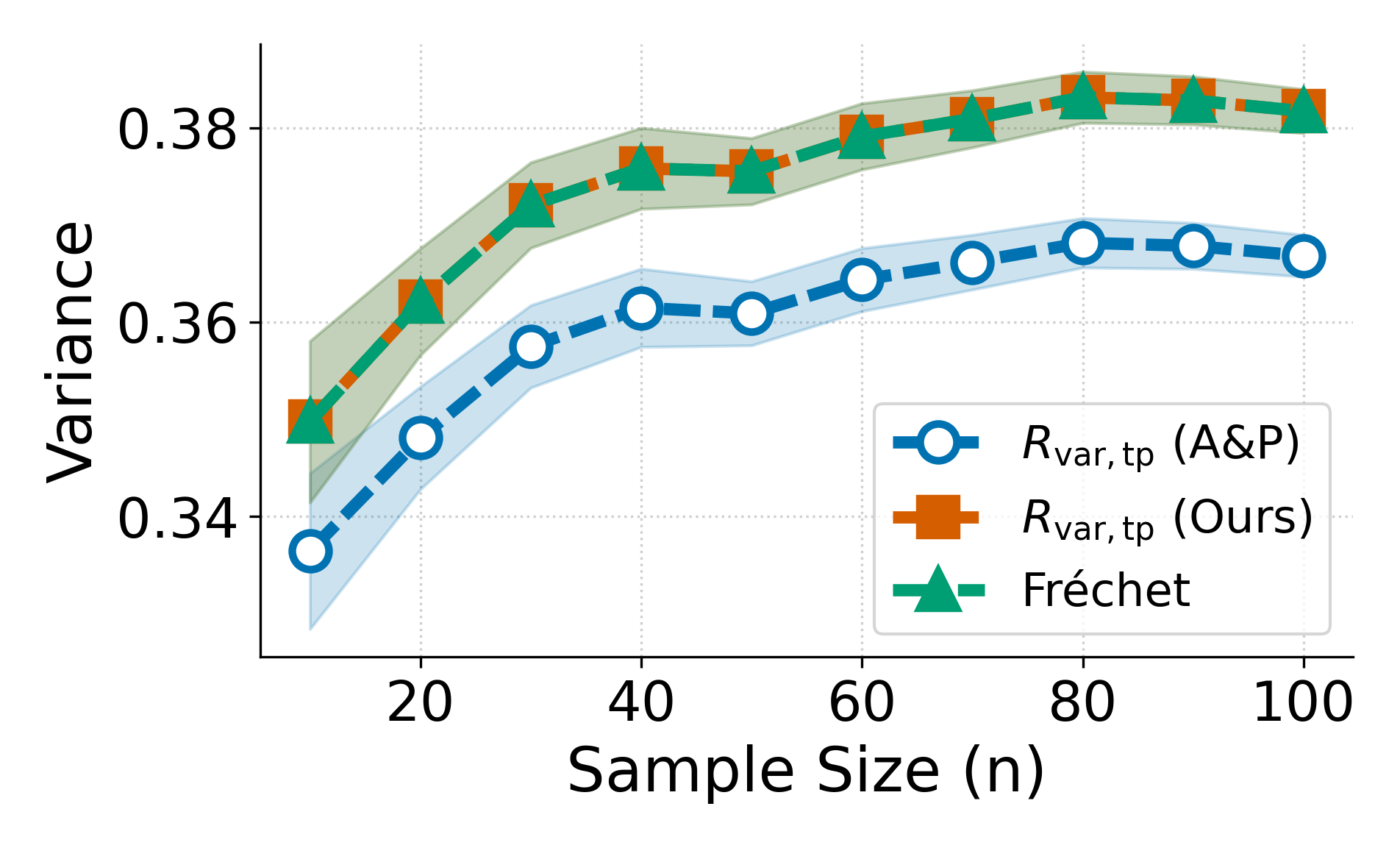}
    \end{tabular}
    \caption{Comprehensive Comparison on the SPD Manifold with parallel-transported dataset. Correlation, covariance, and variance estimates versus sample size. Shaded regions indicate 95\% confidence intervals. Parameters: $d=3, \tau=0.5, \sigma_\epsilon=0.15$, MC $=200$.}
    \label{fig:spd_comprehensive2}
\end{figure}

\subsection{Kendall's Shape Space}

Kendall's shape space \citep{kendall1984shape} is a geometric framework for shape analysis in a way that treat objects as having the same shape regardless of their location, scale, and orientation. For more details, we refer to \ref{app:kendalldetail}. Kendall's shape space is identified with the quotient space $$\mcL_{m,k}/\left(\mbR^m\rtimes(\mbR^+\times SO(m))\right),$$ with the translation group $\mbR^m$, scaling group $\mbR^+$, and rotation group $SO(m)$. Here $\rtimes$ is the semi-product of groups. In the special case of planar shapes $(m=2)$, the shape space is a complex projective space $\mbC P^{k-2}$, which has a well-defined Riemannian geometry structure. 

We evaluate our approach on the Sunnybrook Cardiac MRI dataset \cite{radau2009evaluation}, which consists of two-dimensional cine magnetic resonance images (MRI) acquired in the short-axis view. The dataset includes four clinically distinct groups: Healthy controls (Normal), Hypertrophic cardiomyopathy (Hypertrophy), Heart failure without infarction (HF-Infarct), and Heart failure with infarction (HF+Infarct). Landmarks are extracted from contour coordinate of each cardiac structure. We compute the correlation between the inner wall (endocardium) and outer wall (epicardium) of left ventricular (LV) anatomy within each group. Our goal is to investigate whether the cardiac condition could affect the LV structure and symmetry.

\begin{table}[h]
    \centering
    \caption{Correlation estimates for the four groups using our proposed method, the $A\&P$ correlation, and canonical correlation analysis (CCA).}
    \small
    \begin{tabular}{cccc}
    \toprule
        & $R_{corr}$(Ours) & $R_{corr}$($A\&P)$ & CCA \\
    \midrule
    Normal($n=9$)  &  0.587 & 0.637 & 1\\
    Hypertrophy($n=12$) & 0.352 & 0.394 & 1\\
    HF-infarct($n=12$) &  0.840 & 0.846 & 1\\
    HF+infarct($n=12$) & 0.628  & 0.659 & 1\\
    \bottomrule
    \end{tabular}
    
    \label{tab:shape_corr}
\end{table}

Table \ref{tab:shape_corr} compares the correlation estimates from our proposed approach, $A\&P$ correlation, and CCA. As in the previous analysis, the $A\&P$ correlation is computed at the joint Fr\'echet mean. The proposed estimates are close to those obtained using the $A\&P$ method, whereas CCA yields correlations equal to 1 across all four groups. This reflects an inherent limitation of CCA: since it directly optimizes correlation, it is prone to substantial overfitting in small-sample and high-dimensional settings.

Similar shapes in shape space are expected to exhibit a high degree of association. The estimated Riemannian correlation results suggest that the HF-infarct group has the highest correlation, while the hypertrophy group shows a relatively weak association. As illustrated in Figure \ref{fig:shape_visual}, the inner and outer walls of the LV in the HF-infarct group are expected to remain geometrically symmetric. In contrast, MRI images of patient with hypertrophy typically display imbalanced wall thickness, leading to a disruption of symmetry. These observations are consistent with existing finds in the literature. For instance, \citet{maron2014hypertrophic} reports that in hypertrophic cardiomyopathy, the inner wall becomes irregular and chaotic, whereas the outer wall remains relatively smooth, thereby breaking the geometric coupling structure. Moreover, \citet{eshaghian2007cardiac} states that amyloid infiltration of the heart can lead to a uniformly thickened wall with a firm, rubbery consistency, resulting in heart failure without infarction. Overall, our proposed estimator provides reasonable results and presents a quantitative measure of structural coupling in shape geometry.
 
\section{Discussion}
We have introduced two novel definitions, the intrinsic footpoint-invariant Riemannian cross-covariance and the subsequent Riemannian correlation. These tools are fundamental for understanding the second-order information present in manifold-valued random objects. Our estimators enjoy many desirable properties (e.g., Proposition~\ref{prop:properties}), which makes them a robust and natural choice for manifold statistical analysis. Lastly, their asymptotic properties are vital for quantifying their behavior.
The experimental results on the sphere, SPD matrices, and Kendall's shape space show the applicability of our method and validate its properties. Given its effectiveness, the method can further the area of geometric learning on Riemannian manifolds, including regression and principal component analysis.
We explored a few manifolds; however, many others can be considered, such as the Grassmannian, the Stiefel manifold, the spaces of rotation and reflection matrices, hyperbolic space, and product manifolds such as the torus, to name a few.
Further, our method can be extended to Lie groups, where the natural choice for a tangent space is the Lie algebra.

\newpage
\bibliography{ref.bib}

@article{petersen2019frechet,
  title={Fr{\'e}chet regression for random objects with Euclidean predictors},
  author={Petersen, Alexander and M{\"u}ller, Hans-Georg},
  journal={The Annals of Statistics},
  volume={47},
  number={2},
  pages={691--719},
  year={2019},
  publisher={JSTOR}
}

@article{pennec2006intrinsic,
  title={Intrinsic statistics on Riemannian manifolds: Basic tools for geometric measurements},
  author={Pennec, Xavier},
  journal={Journal of Mathematical Imaging and Vision},
  volume={25},
  number={1},
  pages={127--154},
  year={2006},
  publisher={Springer}
}

@article{abuqrais2024riemannian,
  title={A Riemannian covariance for manifold-valued data},
  author={Abuqrais, Meshal and Pigoli, Davide},
  journal={arXiv preprint arXiv:2410.06164},
  year={2024}
}

@article{karcher1977riemannian,
  title={Riemannian center of mass and mollifier smoothing},
  author={Karcher, Hermann},
  journal={Communications on pure and applied mathematics},
  volume={30},
  number={5},
  pages={509--541},
  year={1977},
  publisher={Wiley Online Library}
}

@inproceedings{frechet1948elements,
  title={Les {\'e}l{\'e}ments al{\'e}atoires de nature quelconque dans un espace distanci{\'e}},
  author={Fr{\'e}chet, Maurice},
  booktitle={Annales de l'institut Henri Poincar{\'e}},
  volume={10},
  pages={215--310},
  year={1948}
}

@inproceedings{kim2014canonical,
  title={Canonical correlation analysis on riemannian manifolds and its applications},
  author={Kim, Hyunwoo J and Adluru, Nagesh and Bendlin, Barbara B and Johnson, Sterling C and Vemuri, Baba C and Singh, Vikas},
  booktitle={European conference on computer vision},
  pages={251--267},
  year={2014},
  organization={Springer}
}

@article{petersen2019wasserstein,
  title={Wasserstein covariance for multiple random densities},
  author={Petersen, Alexander and M{\"u}ller, Hans-Georg},
  journal={Biometrika},
  volume={106},
  number={2},
  pages={339--351},
  year={2019},
  publisher={Oxford University Press}
}

@book{lee2018introduction,
  title={Introduction to Riemannian manifolds},
  author={Lee, John M},
  volume={2},
  year={2018},
  publisher={Springer}
}

@book{do1992riemannian,
  title={Riemannian geometry},
  author={do Carmo, Manfredo Perdigao},
  volume={2},
  year={1992},
  publisher={Springer}
}

@article{kendall1990probability,
  title={Probability, convexity, and harmonic maps with small image I: uniqueness and fine existence},
  author={Kendall, Wilfrid S},
  journal={Proceedings of the London Mathematical Society},
  volume={3},
  number={2},
  pages={371--406},
  year={1990},
  publisher={Wiley Online Library}
}

@article{LinIntrinsic,
author = {Zhenhua Lin and Fang Yao},
title = {{Intrinsic Riemannian functional data analysis}},
volume = {47},
journal = {The Annals of Statistics},
number = {6},
publisher = {Institute of Mathematical Statistics},
pages = {3533 -- 3577},
year = {2019},
doi = {10.1214/18-AOS1787},
URL = {https://doi.org/10.1214/18-AOS1787}
}

@article{shao2022intrinsic,
  title={Intrinsic Riemannian functional data analysis for sparse longitudinal observations},
  author={Shao, Lingxuan and Lin, Zhenhua and Yao, Fang},
  journal={The Annals of Statistics},
  volume={50},
  number={3},
  pages={1696--1721},
  year={2022},
  publisher={Institute of Mathematical Statistics}
}

@article{mardia1975statistics,
  title={Statistics of directional data},
  author={Mardia, Kantilal Varichand},
  journal={Journal of the Royal Statistical Society Series B: Statistical Methodology},
  volume={37},
  pages={349--371},
  year={1975},
  publisher={Oxford University Press}
}

@article{fu2025adaptive,
  title={Adaptive Electromagnetic Analysis via Non-Euclidean Manifold Learning for Atmospheric Precipitation Understanding},
  author={Fu, Tian and Yao, Tianliang and Wang, Haoyu and Chen, Bin},
  journal={IEEE Geoscience and Remote Sensing Letters},
  year={2025},
  publisher={IEEE}
}

@incollection{mardia2018directional,
  title={Directional statistics in protein bioinformatics},
  author={Mardia, Kanti V and Foldager, Jesper Illemann and Frellsen, Jes},
  booktitle={Applied Directional Statistics},
  pages={17--40},
  year={2018},
  publisher={Chapman and Hall/CRC}
}

@article{rao1945information,
  title={Information and the accuracy attainable in the estimation of statistical parameters},
  author={Rao, C Radhakrishna and others},
  journal={Bull. Calcutta Math. Soc},
  volume={37},
  number={3},
  pages={81--91},
  year={1945}
}

@article{dai2019analyzing,
  title={Analyzing dynamical brain functional connectivity as trajectories on space of covariance matrices},
  author={Dai, Mengyu and Zhang, Zhengwu and Srivastava, Anuj},
  journal={IEEE transactions on medical imaging},
  volume={39},
  number={3},
  pages={611--620},
  year={2019},
  publisher={IEEE}
}

@article{kendall1984shape,
  title={Shape manifolds, procrustean metrics, and complex projective spaces},
  author={Kendall, David G},
  journal={Bulletin of the London mathematical society},
  volume={16},
  number={2},
  pages={81--121},
  year={1984},
  publisher={Wiley Online Library}
}

@article{lu2015clustering,
  title={Clustering tree-structured data on manifold},
  author={Lu, Na and Miao, Hongyu},
  journal={IEEE transactions on pattern analysis and machine intelligence},
  volume={38},
  number={10},
  pages={1956--1968},
  year={2015},
  publisher={IEEE}
}

@article{cheng1989historical,
  title={An historical note on finite rotations},
  author={Cheng, Hui and Gupta, Krishna C},
  journal={Journal of Applied Mechanics},
  volume={56},
  number={1},
  pages={139--145},
  year={1989},
  publisher={American Society of Mechanical Engineers Digital Collection}
}

@article{afsari2011riemannian,
  title={Riemannian L\^{}$\{$p$\}$ center of mass: existence, uniqueness, and convexity},
  author={Afsari, Bijan},
  journal={Proceedings of the American Mathematical Society},
  volume={139},
  number={2},
  pages={655--673},
  year={2011}
}

@article{JMLR:v17:16-177,
    author = {James Townsend and Niklas Koep and Sebastian Weichwald},
    journal = {Journal of Machine Learning Research},
    number = {137},
    pages = {1–5},
    title = {Pymanopt: A Python Toolbox for Optimization on Manifolds using Automatic Differentiation},
    url = {http://jmlr.org/papers/v17/16-177.html},
    volume = {17},
    year = {2016}
}

@article{fletcher2004principal,
  title={Principal geodesic analysis for the study of nonlinear statistics of shape},
  author={Fletcher, P Thomas and Lu, Conglin and Pizer, Stephen M and Joshi, Sarang},
  journal={IEEE transactions on medical imaging},
  volume={23},
  number={8},
  pages={995--1005},
  year={2004},
  publisher={IEEE}
}

@article{bhattacharya2005large,
  title={Large sample theory of intrinsic and extrinsic sample means on manifolds---II},
  author={Bhattacharya, Rabi and Patrangenaru, Vic},
  journal={The Annals of Statistics},
  volume={33},
  number={3},
  pages={1225--1259},
  year={2005},
  doi={10.1214/009053605000000093}
}

@article{bhattacharya2003large,
  title={Large sample theory of intrinsic and extrinsic sample means on manifolds},
  author={Bhattacharya, Rabi and Patrangenaru, Vic},
  journal={The Annals of Statistics},
  volume={31},
  number={1},
  pages={1--29},
  year={2003},
  doi={10.1214/aos/1046294456}
}

@article{miolane2020geomstats,
  title={Geomstats: A python package for riemannian geometry in machine learning},
  author={Miolane, Nina and Guigui, Nicolas and Le Brigant, Alice and Mathe, Johan and Hou, Benjamin and Thanwerdas, Yann and Heyder, Stefan and Peltre, Olivier and Koep, Niklas and Zaatiti, Hadi and others},
  journal={Journal of Machine Learning Research},
  volume={21},
  number={223},
  pages={1--9},
  year={2020}
}

@inproceedings{guigui2021parallel,
  title={Parallel transport on kendall shape spaces},
  author={Guigui, Nicolas and Maignant, Elodie and Trouv{\'e}, Alain and Pennec, Xavier},
  booktitle={International Conference on Geometric Science of Information},
  pages={103--110},
  year={2021},
  organization={Springer}
}

@article{radau2009evaluation,
  title={Evaluation framework for algorithms segmenting short axis cardiac MRI.},
  author={Radau, Perry and Lu, Yingli and Connelly, Kim and Paul, Gideon and Dick, Alexander J and Wright, Graham A},
  journal={The MIDAS Journal},
  year={2009},
  publisher={NumFOCUS-Insight Software Consortium (ITK)}
}

@article{maron2014hypertrophic,
  title={Hypertrophic cardiomyopathy: present and future, with translation into contemporary cardiovascular medicine},
  author={Maron, Barry J and Ommen, Steve R and Semsarian, Christopher and Spirito, Paolo and Olivotto, Iacopo and Maron, Martin S},
  journal={Journal of the American College of Cardiology},
  volume={64},
  number={1},
  pages={83--99},
  year={2014},
  publisher={American College of Cardiology Foundation Washington, DC}
}

@article{eshaghian2007cardiac,
  title={Cardiac amyloidosis: new insights into diagnosis and management},
  author={Eshaghian, Shervin and Kaul, Sanjay and Shah, Prediman K},
  journal={Reviews in Cardiovascular Medicine},
  volume={8},
  number={4},
  pages={189--199},
  year={2007},
  publisher={IMR Press}
}

@misc{mlsp-2014-mri,
    author = {Eduardo Castro and joycenv and Mustafa Cetin and Navin Cota and Rogers F Silva and Vince},
    title = {MLSP 2014 Schizophrenia Classification Challenge},
    year = {2014},
    howpublished = {\url{https://kaggle.com/competitions/mlsp-2014-mri}},
    note = {Kaggle}
}
\bibliographystyle{plainnat}


\appendix

\section{Riemannian Geometry}\label{app:ManNotes}
\subsection{Sphere}\label{app:spheredetail}
Let $\mcM=\mbS^d$ be the $d$-dimensional unit sphere in $\mbR^{d+1}$ with the metric induced from the ambient space. The exponential map is given by
$$ \exp_p(v) = \cos(\|v\|)\, p + \sin(\|v\|)\, \frac{v}{\|v\|}.$$ The logarithm map is given by $$\log_p(q) = \frac{\theta}{\sin(\theta)}\, (q-\cos(\theta) p),$$ with $q \notin \{p,-p\}$ and $\theta=\arccos\langle p,q\rangle$. Lastly, the parallel transport of a vector $v$ along a geodesic from $p$ to $q$ is given by $$\Gamma_p^q(v)=v-\frac{\langle \log_pq,v\rangle_p}{\theta^2}\, (\exp^{-1}_pq+\exp^{-1}_qp),$$ with $\theta$ as above.
\subsection{SPDM}\label{app:spdmdetail}
Let $\mcM$ be the space of SPD matrices with the affine invariant metric (the Rao-Fisher metric). The exponential map is given by $$\exp_p(v) = p^{1/2}\,\text{Exp}\!\left(p^{-1/2}vp^{-1/2}\right) p^{1/2},$$ and the log map is given by $$\log_q(p) = q^{1/2}\,\text{Log}\!\left(q^{-1/2}p\, q^{-1/2}\right) q^{1/2},$$ where $\text{Exp}$ and $\text{Log}$ are the matrix exponential and logarithm, respectively. The parallel transport of tangent vector $v\in T_p\mcM$ along a geodesic connecting $p$ and $q$ is given by $$\Gamma_p^qv=(qp^{-1})^{1/2}v(p^{-1}q)^{1/2}.$$
\subsection{Kendall's shape space}\label{app:kendalldetail}
Define the ordered k-tuple landmark space $$\mcL_{m,k}=\{X\in\mbR^{m\times k}|\dim(\text{span}(X))=m, k>m\}.$$
For an object $X\in\mcL_{m,k}$, we first remove the effect of translation and scaling $$\hat X=\frac{X-\bar x\,\mathbf{1}_{k}^T}{\|X-\bar x\,\mathbf{1}_{k}^T\|_F},$$
where $\bar x\in\mbR^m$ is the mean across landmarks (i.e., the row-wise mean over the $k$ columns) and $\|\cdot\|_F$ is the Frobenius norm. This defines the preshape space $$\mcS_{m}^k=\Big\{X\in\mcL_{m,k}\,\Big|\,X\mathbf{1}_k=0,\ \|X\|_F=1\Big\},$$ which lies on a hypersphere $\mcS^{m(k-1)-1}$. The Kendall's shape space is further defined by taking quotient space of preshape space under the rotation group $$\Sigma_m^k = \mcS_m^k/SO(m).$$

For each element $\tilde X\in\Sigma_m^k$, there corresponds an orbit in preshape space $\tilde X=[\hat X]=\{O\hat X|O\in SO(m)\}$. Therefore, the exponential map in $\Sigma_{m}^k$ is defined (on a chosen representative $\hat X_1$ of $\tilde X_1$) by $$\exp_{\tilde X_1}(v)=\cos(\|v\|_F)\hat X_1+\sin(\|v\|_F)\frac{v}{\|v\|_F}.$$
The logarithm map is given by $$\log_{\tilde X_1}(\tilde X_2)=\frac{\theta}{\sin(\theta)}(O^*\hat X_2-\cos(\theta)\hat X_1),$$
where $\theta=\arccos(\langle\hat X_1, O^*\hat X_2\rangle_F)$ and $O^*=\argmax_{O\in SO(m)}\langle\hat X_1,O\hat X_2\rangle_F.$ The spherical parallel transport cannot be applied on Kendall's shape space directly, but transported vector can be approximated using a pole ladder algorithm \citep{guigui2021parallel}.
\section{Canonical Correlation}\label{app:cca}
Here we include more details from \citet{kim2014canonical}. The geodesic submanifold \cite{fletcher2004principal} $S\subset\mcM$ is defined as $S=\exp_p(\mathrm{span}(\{b_i\})\cap U)$ with $U\subset T_p\mcM$ and $b_i\in T_p\mcM$. The geodesic submanifold is the subspace of $\mcM$ generated by pushing the span of some vectors $\{b_i\}$ onto the manifold. Note that when $|\{b_i\}|=1$ then $S$ is simply a geodesic. Note that once an appropriate $\mcR^{CCA}_{X,Y}$ is obtained, one can determine the covariance by multiplying by the standard deviations of both sets of coefficients.

Notice that projection onto the geodesic submanifold can be expressed in terms of the inner product, thus using Euclidean metric, the Log-Euclidean framework of CCA is actually equivalent to the manifold-valued CCA:
\begin{align*}
    \mcR^{CCA}_{X,Y}&=\max_{W_x,W_y}\frac{\sum_i\langle W_x,\log_{\mu_x}(x_i)\rangle\langle W_y, \log_{\mu_y}(y_i)\rangle}{\sqrt{\sum_i\langle W_x,\log_{\mu_x}(x_i)\rangle^2}\sqrt{\sum_i\langle W_y, \log_{\mu_y}(y_i)\rangle^2}}\\
    &=\max_{W_x,W_y}\frac{W^T_x \Cov(X,Y)\,W_y}{\sqrt{W^T_x \Cov(X,X)\,W_x}\sqrt{W^T_y \Cov(Y,Y)\,W_y}},
\end{align*}
where $\Cov(X,Y)=\frac{1}{n}\sum_i[(\log_{\mu_x}x_i)(\log_{\mu_y}y_i)^T]$ and similar for $\Cov(X,X)$ and $\Cov(Y,Y)$. Unlike the general cross-covariance, which accounts for all possible linear combination between datasets, the CCA computes a normalized maximal cross-correlation, thus it tends to overestimate the correlation when datasets share the same coordinates.
\section{Euclidean Covariance and Correlation}
Let $(\Omega, \mcF, \mbP)$  be a probability space and let $X$ and $Y$ be two integrable random variables such that $X, Y\in L^1(\Omega, \mcF, \mbP)$. The expectation of $X$ is defined as $$\mbE(X)=\int_\Omega X(\omega)d\mbP(\omega),$$ and similary for $\mbE(Y)$. The covariance of $X$ and $Y$ is given by $$\text{Cov}(X,Y)=\mbE((X-\mbE(X))(Y-\mbE(Y))).$$
Equivalently, $$\text{Cov}(X,Y)=\mbE(XY)-\mbE(X)\mbE(Y).$$
The covariance has the following properties:
\begin{enumerate}
    \item Variance: $\Cov(X,X)=\text{Var}(X)$;
    \item Symmetry: $\Cov(X,Y)=\Cov(Y,X)$;
    \item Scaling: $\Cov(aX,Y)=a\Cov(X,Y)$;
    \item Uncorrelatedness: If $X$ and $Y$ are independent, then $\Cov(X,Y)=0$;
    \item Invariance to constant: $\Cov(X+c,Y)=\Cov(X,Y)$;
    \item Variance of a constant: $\Cov(X,c)=0$;
    \item Additivity: $\Cov(X_1+X_2,Y)=\Cov(X_1,Y)+\Cov(X_2,Y)$.
    More generally, $$\Cov\left(\sum_ia_iX_i,\sum_jb_jY_j\right)=\sum_i\sum_ja_ib_j\Cov(X_i,Y_j).$$
\end{enumerate}
With the variance given by $$\text{Var}(X)=\Cov(X,X),$$ the Pearson correlation is defined as $$\rho(X,Y)=\frac{\Cov(X,Y)}{\sqrt{\text{Var}(X)}\sqrt{\text{Var}(Y)}}$$
Thus we have the following properties for correlation:
\begin{enumerate}
    \item Bounded: $\rho(X,Y)\in[-1,1]$;
    \item Symmetry: $\rho(X,Y)=\rho(Y,X)$;
    \item Scale invariance: $\rho(aX+b, cY+d)=\text{sign}(ac)\rho(X,Y)$;
    \item Self-correlation: if X is nondegenerate, $\rho(X,X)=1$.
\end{enumerate}
\section{Covariance Under Basis Rotation}

The parallel transport can be interpreted as a rotation of basis when the manifold is viewed as embedded in some ambient space, i.e. $\mbR^n$. Suppose that we want to move a vector $w$ from the tangent space of $p\in\mcM$ to another tangent space of $q\in\mcM$ along the geodesic connecting $p$ and $q$. The vector $w\in T_p\mcM$ can be written as $$w=\sum_na_ie_i^p,$$ where $\{e_i^p\}$ is a basis of $T_p\mcM$, then under the parallel transport, the corresponding basis on $T_q\mcM$ is mapped by $e_i^q=De_i^p$, where $D\in SO(n)$ and the transported vector $w'\in T_q\mcM$ is $$w'=\sum_na_ie_i^q=Dw.$$ Therefore, we identify $D$ with parallel transport matrix, denoted $D\simeq\Gamma_\mu^\nu$.

Under this setting, for two random variables $X,Y\in\mcM$ with respective mean $\mu,\nu$, we have random vectors $u=\log_\mu X\in T_\mu\mcM$ and $v=\log_\nu Y\in T_\nu\mcM$. Let $\gamma:[0,1]\to\mcM$ be a geodesic with $\gamma(0)=\mu$, $\gamma(1)=\nu$, and for some $t\in[0,1]$, the corresponding matrices to parallel transport is given by \begin{align*}
    A\simeq \Gamma_{\gamma(0)}^{\gamma(1)},\quad B(t)\simeq \Gamma_{\gamma(0)}^{\gamma(t)}, \quad C(t)\simeq \Gamma_{\gamma(t)}^{\gamma(1)},
\end{align*}
where $A,B,C\in SO(n)$, and hence we also have $A=C(t)B(t)$ and \begin{align*}
    A^T\simeq \Gamma^{\gamma(0)}_{\gamma(1)},\quad B^T(t)\simeq \Gamma^{\gamma(0)}_{\gamma(t)}, \quad C^T(t)\simeq \Gamma^{\gamma(t)}_{\gamma(1)},
\end{align*}

The Riemannian cross-covariance at $p=\alpha(t)$ can be defined alternatively as \begin{align*}
    \Sigma_{X,Y}(p)&=\mbE((\Gamma_\mu^p\log_\mu X)(\Gamma_\nu^p\log_\nu Y)^T)\\
    &=\mbE((\Gamma_{\gamma(0)}^{\gamma(t)}u)(\Gamma_{\gamma(1)}^{\gamma(t)}v)^T)\\
    &=\mbE((B(t)u)(C(t)^Tv)^T)\\
    &=\mbE(B(t)uv^TC(t)).
\end{align*}

It is obvious that the covariance matrix is subject to the choice of the p, however for the Riemannian covariance, we have \begin{align*}
    \Cov_{X,Y}(p)&=\tr(\Sigma_{X,Y}(p))\\
    &=\tr(\mbE(B(t)uv^TC(t))\\
    &=\mbE(\tr(C(t)B(t)uv^T))\\
    &=\mbE(\tr(Auv^T))\\
    &=\tr(\mbE((\Gamma_{\mu}^\nu\log_\mu X)(\log_\nu Y)^T)\\
    &=\Cov_{X,Y}(\nu),
\end{align*}

similarly, we also have $\Cov_{X,Y}(p)=\Cov_{X,Y}(\mu)$. Since it holds for any $p=\alpha(t)$, it is enough to show the invariance of Riemannian covariance along the geodesic $\alpha$.

\section{Additional Proofs}\label{ss:AdProof}

\begin{proof}[Proof of Proposition~\ref{prop:ap-variant}]
Fix two distinct points $p,q \in \mathcal M$, and write $\xi = \log_p(q)$.
Since $\mathcal M$ has nonzero curvature, the logarithmic map does not commute
with parallel transport. In particular, for any $x$ in a sufficiently small
normal neighborhood of $p$, the logarithmic map admits the expansion
\begin{equation}
\label{eq:log_change_base}
\log_q(x)
=
\Gamma^q_p \big(\log_p(x)-\xi\big)
+
C_p(\log_p(x),\xi)
+
O\!\left(\|\log_p(x)\|^3+\|\xi\|^3\right),
\end{equation}
where $\Gamma^q_p$ denotes parallel transport along the geodesic from $p$ to
$q$, the leading term $-\Gamma^q_p\xi$ is the translation in the basepoint
(present already in the Euclidean limit), and $C_p(\cdot,\cdot)$ is a bilinear
term involving the Riemann curvature tensor at $p$. The curvature term
vanishes identically if and only if $\mathcal M$ is flat.

Applying \eqref{eq:log_change_base} to $X$ and $Y$ and centering, the
translation $-\Gamma^q_p\xi$ cancels in the centered second moment, so
\begin{align*}
    \Sigma^{A\&P}_{X,Y}(q)
    &=\Gamma^q_p\,\Sigma^{A\&P}_{X,Y}(p)\,(\Gamma^q_p)^\top
    +D_{XY},
\end{align*}
where the correction term $D_{XY}$ collects the curvature contributions
\begin{align*}
D_{XY}&=\mathbb E\big[
\bar C_p(\log_p X,\xi)(\log_p Y-\mbE\log_p Y)^\top \\
&\quad+(\log_p X-\mbE\log_p X)\bar C_p(\log_p Y,\xi)^\top\big]+O(\cdot),
\end{align*}
with $\bar C_p$ involving the Riemann curvature tensor at $p$.

Taking traces and using that parallel transport is an isometry (so $\operatorname{tr}(\Gamma^q_p A(\Gamma^q_p)^\top)=\operatorname{tr}(A)$),
\[
\operatorname{tr}\!\big(\Sigma^{A\&P}_{X,Y}(q)\big)=\operatorname{tr}\!\big(\Sigma^{A\&P}_{X,Y}(p)\big)+\operatorname{tr}(D_{XY}).
\]
Since the curvature tensor is nonzero, one can choose distributions of $X$
and $Y$ for which $\operatorname{tr}(D_{XY})\neq 0$. Consequently,
\[
\operatorname{tr}\!\big(\Sigma^{A\&P}_{X,Y}(q)\big)
\neq
\operatorname{tr}\!\big(\Sigma^{A\&P}_{X,Y}(p)\big),
\]
which proves that the Abuqrais-Pigoli covariance depends on the choice of
footpoint.
\end{proof}

\begin{proof}[Proof of Theorem~\ref{thm:ftinvar}]
\begin{align*}
        \tr\left[\Sigma_{X,Y}(p)\right] &=\tr\left[\mbE\left[(\Gamma_\mu^p\log_\mu X)(\Gamma_\nu^p\log_\nu Y)^T\right]\right]\\
        &=\mbE\left[\tr\left[(\Gamma_\mu^p\log_\mu X)(\Gamma_\nu^p\log_\nu Y)^T\right]\right]\\
        &=\mbE\left[\tr\left[(\Gamma_p^q\Gamma_\mu^p\log_\mu X)(\Gamma_p^q\Gamma_\nu^p\log_\nu Y)^T\right]\right]\\
        &=\mbE\left[\tr\left[(\Gamma_\mu^q\log_\mu X)(\Gamma_\nu^q\log_\nu Y)^T\right]\right]\\
        &=\tr\left[\mbE\left[(\Gamma_\mu^q\log_\mu X)(\Gamma_\nu^q\log_\nu Y)^T\right]\right]\\
        &=\tr\left[\Sigma_{X,Y}(q)\right]. \\
    \end{align*}    
     Here the diagonal elements of $\left[(\Gamma_\mu^p\log_\mu X)(\Gamma_\nu^p\log_\nu Y)^T\right]$ are $\langle\Gamma_\mu^p\log_\mu X,e_i\rangle\cdot \langle\Gamma_\nu^p\log_\nu Y,e_i\rangle$ where $\{e_i\}$ is a local orthonormal basis in $T_p\mcM$. Noting that parallel transport is done with respect to a basis and thus applying the parallel transport we have $\langle\Gamma_p^q\Gamma_\mu^p\log_\mu X,\Gamma_p^qe_i\rangle\cdot \langle\Gamma_p^q \Gamma_\nu^p\log_\nu Y,\Gamma_p^qe_i\rangle$. Since $p,q,\mu,\nu$ are all on the same geodesic we have that $\langle\Gamma_p^q\Gamma_\mu^p\log_\mu X,\Gamma_p^qe_i\rangle\cdot \langle\Gamma_p^q \Gamma_\nu^p\log_\nu Y,\Gamma_p^qe_i\rangle=\langle\Gamma_\mu^q\log_\mu X,\Gamma_p^qe_i\rangle\cdot \langle \Gamma_\nu^q\log_\nu Y,\Gamma_p^qe_i\rangle$.
    Parallel transport preserves inner products thus the trace (sum) is unaffected. Note that since $p,q,\mu$ are on the same geodesic $\Gamma_p^q\Gamma_\mu^p v= \Gamma_\mu^q v$, otherwise this does not generally hold.    
    
\end{proof}

\begin{proof}[Proof of Proposition~\ref{prop:prop1}]
Set $U:=\Gamma^{\nu}_{\mu}\log_\mu X$ and $V:=\log_\nu Y$, both elements of $T_\nu M$. Since $\Gamma^{\nu}_{\mu}$ is a linear isometry, $\|U\|_\nu=d(\mu,X)$ and $\|V\|_\nu=d(\nu,Y)$, so the second-moment hypotheses give $\mathbb{E}\|U\|_\nu^{2},\mathbb{E}\|V\|_\nu^{2}<\infty$.
For $u,v\in T_\nu M$, define the rank-one operator $u\otimes v\in\mathcal{H}$ by $(u\otimes v)(w)=\langle v,w\rangle_\nu u$. A direct expansion in any orthonormal basis $\{e_i\}_{i=1}^{d}$ of $T_\nu M$ gives $\|u\otimes v\|_{\mathcal{H}}=\|u\|_\nu\|v\|_\nu$ together with the pairing identity $\langle u\otimes v,A\rangle_{\mathcal{H}}=\langle u,Av\rangle_\nu$ for every $A\in\mathcal{H}$, which specializes to $\langle u\otimes v,\text{id}_{T_\nu M}\rangle_{\mathcal{H}}=\langle u,v\rangle_\nu$.
Continuity of $(u,v)\mapsto u\otimes v$ makes
$\omega\mapsto U(\omega)\otimes V(\omega)$ strongly measurable in $\mathcal{H}$,
and by Cauchy--Schwarz
\[
\mathbb{E}\|U\otimes V\|_{\mathcal{H}}=\mathbb{E}[\|U\|_\nu\|V\|_\nu]
\le(\mathbb{E}\|U\|_\nu^{2})^{1/2}(\mathbb{E}\|V\|_\nu^{2})^{1/2}<\infty,
\]
so $U\otimes V$ is Bochner integrable and $\widetilde{C}_{X,Y}=\mathbb{E}[U\otimes V]
\in\mathcal{H}$ is well defined. Since $A\mapsto\langle Ae_j,e_i\rangle_\nu$ is
a bounded linear functional on $\mathcal{H}$, it commutes with the Bochner
expectation, and a direct calculation gives
$\langle\widetilde{C}_{X,Y}e_j,e_i\rangle_\nu
=\mathbb{E}[\langle V,e_j\rangle_\nu\langle U,e_i\rangle_\nu]
=(\Sigma_{X,Y})_{ij}$, so $\Sigma_{X,Y}$ is precisely the matrix of
$\widetilde{C}_{X,Y}$ in the basis $\{e_i\}$. The trace
$\tr(\cdot)=\langle\cdot,\text{id}_{T_\nu M}\rangle_{\mathcal{H}}$ is itself a bounded linear
functional on $\mathcal{H}$, and applying it under the expectation together
with the pairing identity above yields
\[
\tr(\Sigma_{X,Y})
=\langle\widetilde{C}_{X,Y},\text{id}_{T_\nu M}\rangle_{\mathcal{H}}
=\mathbb{E}\langle U\otimes V,\text{id}_{T_\nu M}\rangle_{\mathcal{H}}
=\mathbb{E}\langle U,V\rangle_\nu,
\]
which is the claimed identity $\Cov_{X,Y}=\langle\widetilde{C}_{X,Y},\text{id}_{T_\nu M}\rangle_{\mathcal{H}}$.

\end{proof}

\begin{proof}[Proof of Proposition~\ref{prop:properties}]

    \textbf{\textit{Proof of Fr\'echet variance special case.}}

    As we only have one mean, the geodesic reduces to a single point $\mu$. Note that $\Gamma_\mu^\mu$ is the identity.
    \begin{align*}
        \tr\left[\Sigma_{X,X}(\mu)\right] &=\tr\left[\mbE\left[(\Gamma_\mu^\mu\log_\mu X)(\Gamma_\mu^\mu\log_\mu X)^T\right]\right]\\
        &=\tr\left[\mbE\left[(\log_\mu X)(\log_\mu X)^T\right]\right]\\
        &=\mbE\left[\tr\left[(\log_\mu X)(\log_\mu X)^T\right]\right]\\
        &=\mbE\left[\|\log_\mu X\|^2_\mu\right]\\
        &=\mbE\left[d^2(\mu, X)\right]=\sigma^2_X(\mu)\\
    \end{align*}    
    
    The diagonal elements of $\left[(\log_\mu X)(\log_\mu X)^T\right]$ are $\langle\log_\mu X,e_i\rangle\cdot \langle\log_\mu X,e_i\rangle$ where $\{e_i\}$ are local basis in $T_\mu\mcM$. The trace is thus, $\sum \langle\log_\mu X,e_i\rangle^2=\|\log_\mu X\|^2_\mu$.
    
    \textbf{\textit{Proof of symmetry.}}
    \begin{align*}
        \Cov_{X,Y}(p) &= \tr(\Sigma_{X,Y}(p)) \\
        &=\tr\left[\mbE\left[(\Gamma_\mu^p\log_\mu X)(\Gamma_\nu^p\log_\nu Y)^T\right]\right]\\
        &=\mbE\left[\tr\left[(\Gamma_\mu^p\log_\mu X)(\Gamma_\nu^p\log_\nu Y)^T\right]\right]\\
         &=\mbE\left[\tr\left[(\Gamma_\nu^p\log_\nu Y)(\Gamma_\mu^p\log_\mu X)^T\right]\right]\\
        &=\tr\left[\mbE\left[(\Gamma_\nu^p\log_\nu Y)(\Gamma_\mu^p\log_\mu X)^T\right]\right]\\
        &= \Cov_{Y,X}(p)
    \end{align*}

    \textbf{\textit{Proof of degeneracy. }}
    
    Suppose $\Cov_{X,X}(\mu)=0$. By the Fr\'echet variance special case proved above, $\Cov_{X,X}(\mu)=\mbE\|\log_\mu X\|_\mu^2$. Vanishing of a non-negative expectation forces $\|\log_\mu X\|_\mu=0$ almost surely, hence $X=\mu$ a.s. Thus $X$ is degenerate.

    \textbf{\textit{Proof of uncorrelatedness.}}
    
    Suppose $X$ and $Y$ are statistically independent random variables on a Riemannian manifold, i.e., $X\perp\!\!\!\perp Y$.
    Under Assumption~\ref{A1}, the Fr\'echet means $\mu,\nu$ are unique interior minimizers, so the first-order optimality of the Fr\'echet variance functions yields
    \[
    \mbE[\log_\mu X]=0,\qquad \mbE[\log_\nu Y]=0.
    \]
    Independence of $X,Y$ implies independence of $\log_\mu X$ and $\log_\nu Y$, hence
    \begin{align*}
        \Sigma_{X,Y}(p)
        &=\mbE\left[(\Gamma_\mu^p\log_\mu X)(\Gamma_\nu^p\log_\nu Y)^\top\right]\\
        &=\big(\Gamma_\mu^p\,\mbE[\log_\mu X]\big)\big(\Gamma_\nu^p\,\mbE[\log_\nu Y]\big)^\top\\
        &=0,
    \end{align*}
    and consequently $\Cov_{X,Y}(p)=\tr(\Sigma_{X,Y}(p))=0$.

    \textbf{\textit{Proof that a constant does not vary.}}

    This is clear as $\log_qq=0$.

    \textit{\textbf{Proof of scaling.}}

    First, we must define scaling. Given a constant $a\in\mbR$ define the scaling action of $a$ on $X$ as $(a,X)\coloneq \exp_\mu(a\log_\mu(X))$. This scaling requires a footpoint which we choose to be $\mu$ here. The mean $\mu$ is the canonical choice as it is the only choice which does not affect the center of the density. The scaling $a\log_\mu(X)$ is well-defined as $T_\mu\mcM$ is a vector space. Note that $\mu$ remains the Fr\'echet mean of $(a,X)$: by linearity, $\mbE[\log_\mu(a,X)]=a\,\mbE[\log_\mu X]=0$ (since $\mu$ is the Fr\'echet mean of $X$), which is the first-order optimality condition; uniqueness follows from Assumption~\ref{A1}. Given Assumption~\ref{A1}, the support of $X$ and $(a,X)$ are both required to be within a ball with appropriate radius. Thus, $|a|$ has an upper bound dependent on the support of the original density. Given this setup we have,

    \begin{align*}
        \Cov_{(a,X),Y}(p) &= \tr(\Sigma_{(a,X),Y}(p)) \\
        &=\tr\left[\mbE\left[(\Gamma_\mu^p\, a\log_\mu X)(\Gamma_\nu^p\log_\nu Y)^T\right]\right]\\
        &=a\,\tr\left[\mbE\left[(\Gamma_\mu^p\log_\mu X)(\Gamma_\nu^p\log_\nu Y)^T\right]\right]\\
        &= a\,\Cov_{X,Y}(p).
    \end{align*}

\end{proof}

\begin{proof}[Proof of Proposition~\ref{prop:RieCor}]

    First, we have $$\Cov_{X,Y}(p)=\tr(\mbE[(\Gamma_\mu^p\log_\mu X)(\Gamma_\nu^p\log_\nu Y)^T]).$$

    Using Jensen's inequality and Cauchy-Schwarz inequality,
    \begin{align*}
        |\Cov_{X,Y}(p)|&=|\tr(\mbE[(\Gamma_\mu^p\log_\mu X)(\Gamma_\nu^p\log_\nu Y)^T])|\\
        &=|\mbE[\tr((\Gamma_\mu^p\log_\mu X)(\Gamma_\nu^p\log_\nu Y)^T)]|\\
        &\le \mbE[|\tr((\Gamma_\mu^p\log_\mu X)(\Gamma_\nu^p\log_\nu Y)^T)|]\\
        &\le \mbE(\|\Gamma_\mu^p\log_\mu X\|_p\|\Gamma_\nu^p\log_\nu Y\|_p)\\
        &\le (\mbE\|\Gamma_\mu^p\log_\mu X\|_p^2)^{1/2}(\mbE\|\Gamma_\nu^p\log_\nu Y\|_p^2)^{1/2}\\
        &=(\mbE\|\log_\mu X\|_\mu^2)^{1/2}(\mbE\|\log_\nu Y\|_\nu^2)^{1/2}\\
        &=\Cov_{X,X}(\mu)^{1/2}\Cov_{Y,Y}(\nu)^{1/2}.
    \end{align*}

    Therefore, $$\mcR_{X,Y}(p)=\frac{\Cov_{X,Y}(p)}{\sqrt{\Cov_{X,X}(\mu)}\sqrt{\Cov_{Y,Y}(\nu)}}\in [-1,1].$$
\end{proof}

\begin{proof}[Proof of Theorem~\ref{thm:ReduceToEuclid}]
This is clear noting that in Euclidean space $\log_pq=q-p$ and $\Gamma_p^qv=v$.
        \begin{align*}
        \Cov_{X,Y}(p) &= \tr(\Sigma_{X,Y}(p)) \\
        &=\tr\left[\mbE\left[(\Gamma_\mu^p\log_\mu X)(\Gamma_\nu^p\log_\nu Y)^T\right]\right]\\
        &=\tr\left[\mbE\left[( X-\mu)(Y-\nu)^T\right]\right]\\
    \end{align*}
\end{proof}
\section{Consistency and Asymptotic Normality Proofs}\label{app:consproofs}

\begin{proof}[Proof of Theorem~\ref{thm:Sigma_consistency}]
Define $Z_i\coloneq U_iV_i^\top$ where $U_i\coloneq \Gamma_\mu^p\log_\mu(X_i)$ and $V_i\coloneq \Gamma_\nu^p\log_\nu(Y_i)$.
Then $\Sigma_{X,Y}(p)=\mbE[Z_1]$ and the oracle sample average $\bar Z_n\coloneq \frac1n\sum_{i=1}^n Z_i$ satisfies $\bar Z_n\to \mbE[Z_1]$ almost surely by the strong law of large numbers, provided $\mbE\|Z_1\|_F<\infty$.
Given that $\|Z_1\|_F\le \|U_1\|\,\|V_1\|$ and parallel transport is an isometry, hence by Cauchy--Schwarz and (A2),
\[
\begin{aligned}
\mbE\|Z_1\|_F
&\le \mbE\|\log_\mu(X)\|_\mu\,\|\log_\nu(Y)\|_\nu\\
&\le \sqrt{\mbE\|\log_\mu(X)\|_\mu^2}\,\sqrt{\mbE\|\log_\nu(Y)\|_\nu^2}
<\infty.
\end{aligned}
\]
Therefore $\bar Z_n\xrightarrow{a.s.}\Sigma_{X,Y}(p)$.

It remains to show that $\|\widehat{\Sigma}_{X,Y}(p)-\bar Z_n\|_F\to 0$ almost surely.
Write
\[
\begin{aligned}
\widehat{\Sigma}_{X,Y}(p)-\bar Z_n
&=
\frac1n\sum_{i=1}^n (\hat U_i\hat V_i^\top - U_iV_i^\top)\\
&=
\frac1n\sum_{i=1}^n
\Big[(\hat U_i-U_i)\hat V_i^\top + U_i(\hat V_i-V_i)^\top\Big].
\end{aligned}
\]
Thus by triangle inequality and submultiplicativity,
\begin{equation}\label{eq:plug_decomp}
\begin{aligned}
\|\widehat{\Sigma}_{X,Y}(p)-\bar Z_n\|_F
&\le
\underbrace{\frac1n\sum_{i=1}^n \|\hat U_i-U_i\|\,\|\hat V_i\|}_{T_{n,1}}
\\
&\quad+
\underbrace{\frac1n\sum_{i=1}^n \|U_i\|\,\|\hat V_i-V_i\|}_{T_{n,2}}.
\end{aligned}
\end{equation}

It suffices to control $T_{n,1}$ since the bound for $T_{n,2}$ follows by the same argument with $(X,\mu,U)$ replaced by $(Y,\nu,V)$.
Under (A6), the map $(m,x)\mapsto \Gamma_m^p\log_m(x)$ is continuously differentiable on $\U\times\U$ with bounded derivative.
Hence, there exists (random but a.s.\ finite) $L$ such that for all large $n$ (once $\hat\mu\in\U$),
\[
\begin{aligned}
\|\hat U_i-U_i\|
&=
\|\Gamma_{\hat\mu}^p\log_{\hat\mu}(X_i)-\Gamma_{\mu}^p\log_{\mu}(X_i)\|\\
&\le
L\cdot d(\hat\mu,\mu),
\qquad\text{for all }i.
\end{aligned}
\]
Therefore,
\[
T_{n,1}
\le
L\, d(\hat\mu,\mu)\cdot \frac1n\sum_{i=1}^n \|\hat V_i\|.
\]
By Lemma~\ref{lem:mean_consistency}, $d(\hat\mu,\mu)\to 0$ almost surely.
Also $\|\hat V_i\|=\|\log_{\hat\nu}(Y_i)\|_{\hat\nu}$ since $\Gamma_{\hat\nu}^p$ is an isometry.
Because $\hat\nu\to\nu$ a.s.\ and $(m,y)\mapsto \|\log_m(y)\|_m$ is continuous on $\U\times\U$, we have $\|\log_{\hat\nu}(Y_i)\|_{\hat\nu}\to \|\log_\nu(Y_i)\|_\nu$ pointwise in $i$ along almost every sample path for large $n$.
Moreover, by (A2) and uniform boundedness of derivatives on $\U$, one can dominate $\|\log_{\hat\nu}(Y_i)\|_{\hat\nu}$ by $C(1+\|\log_\nu(Y_i)\|_\nu)$ for all large $n$ and some constant $C$ (a standard local Lipschitz bound).
Hence, by the strong law of large numbers,
\[
\frac1n\sum_{i=1}^n \|\hat V_i\|
\xrightarrow{a.s.}
\mbE\|V\| <\infty.
\]
Therefore $T_{n,1} \xrightarrow{a.s.} 0$.
A symmetric argument yields $T_{n,2} \xrightarrow{a.s.} 0$.
Plugging into \eqref{eq:plug_decomp} gives $\|\widehat{\Sigma}_{X,Y}(p)-\bar Z_n\|_F \xrightarrow{a.s.} 0$.
Combining with $\bar Z_n \xrightarrow{a.s.} \Sigma_{X,Y}(p)$ completes the proof.
Finally, $\widehat\Cov_{X,Y}(p)=\tr(\widehat\Sigma_{X,Y}(p)) \xrightarrow{a.s.} \tr(\Sigma_{X,Y}(p))$ by continuity of trace.
\end{proof}

\begin{proof}[Proof of Lemma~\ref{lem:Sigma_if}]
Recall the transport maps $\phi,\psi$ and their plug-in versions from Assumption (A6).
Define the population functional $T(\mu,\nu)\coloneq \mbE[\phi(X)\psi(Y)^\top]$ and the empirical counterpart
\[
T_n(m,n)
\coloneq \frac1n\sum_{i=1}^n \phi_m(X_i)\psi_n(Y_i)^\top,
\]
for $(m,n)\in\U\times\U$,
so that $\Sigma_{X,Y}(p)=T(\mu,\nu)$ and $\widehat\Sigma_{X,Y}(p)=T_n(\hat\mu,\hat\nu)$.
Write $\widehat\Sigma\coloneq \widehat\Sigma_{X,Y}(p)$ and $\Sigma\coloneq \Sigma_{X,Y}(p)$.
Then the decomposition
\begin{equation}\label{eq:decomp_main}
\begin{aligned}
\sqrt{n}\big(\widehat\Sigma-\Sigma\big)
&=
\sqrt{n}\Big(T_n(\mu,\nu)-T(\mu,\nu)\Big)
\\
&\quad+
\sqrt{n}\Big(T_n(\hat\mu,\hat\nu)-T_n(\mu,\nu)\Big)
\end{aligned}
\end{equation}
separates the empirical fluctuation term from the plug-in term.

For the empirical term, let $Z_i\coloneq U_iV_i^\top$ with $\mbE Z_i=\Sigma_{X,Y}(p)$.
Under (A5), $\mbE\|Z_i\|_F^2<\infty$ and the multivariate CLT gives
\begin{equation}\label{eq:clt_empirical}
\begin{aligned}
\sqrt{n}\Big(\mathrm{vec}(T_n(\mu,\nu))-\mathrm{vec}(T(\mu,\nu))\Big)
&=
\frac1{\sqrt{n}}\sum_{i=1}^n \tilde Z_i
\\
&\overset{d}{\to} \Normal(0,\Xi_0),
\end{aligned}
\end{equation}
where $\tilde Z_i\coloneq \mathrm{vec}(Z_i)-\mbE\mathrm{vec}(Z_i)$ and $\Xi_0\coloneq \Var(\mathrm{vec}(Z_1))$.

For the plug-in term, let $\delta_\mu\coloneq \log_\mu(\hat\mu)\in T_\mu\mcM$ and $\delta_\nu\coloneq \log_\nu(\hat\nu)\in T_\nu\mcM$.
Under (A6), $(m,x)\mapsto \phi_m(x)$ and $(n,y)\mapsto \psi_n(y)$ admit first-order Taylor expansions, uniformly over $i$ on $\U$:
\begin{align}
\phi_{\hat\mu}(X_i)
&=
\phi_\mu(X_i) + D_\mu\phi_\mu(X_i)[\delta_\mu] + r_{i,\mu},
\label{eq:taylor_phi}\\
\psi_{\hat\nu}(Y_i)
&=
\psi_\nu(Y_i) + D_\nu\psi_\nu(Y_i)[\delta_\nu] + r_{i,\nu},
\label{eq:taylor_psi}
\end{align}
with $\max_i\|r_{i,\mu}\|=O_p(\|\delta_\mu\|^2)$ and $\max_i\|r_{i,\nu}\|=O_p(\|\delta_\nu\|^2)$.
Multiplying and keeping first-order terms yields
\begin{align}
\hat U_i\hat V_i^\top
&=
U_iV_i^\top
+ \Delta_{\mu,i} V_i^\top
+ U_i \Delta_{\nu,i}^\top
+ R_i,\label{eq:prod_expand}
\end{align}
where $\Delta_{\mu,i}\coloneq D_\mu\phi_\mu(X_i)[\delta_\mu]$ and $\Delta_{\nu,i}\coloneq D_\nu\psi_\nu(Y_i)[\delta_\nu]$.
The remainder satisfies $\frac1n\sum_{i=1}^n\|R_i\|_F=O_p(\|\delta_\mu\|^2+\|\delta_\nu\|^2+\|\delta_\mu\|\,\|\delta_\nu\|)$, hence by Lemma~\ref{lem:mean_clt},
\begin{equation}\label{eq:remainder_small}
\sqrt{n}\cdot \Big\|\frac1n\sum_{i=1}^n R_i\Big\|_F = o_p(1).
\end{equation}

Using $\mathrm{vec}(ab^\top)=b\otimes a$ gives the linearization
\begin{equation}\label{eq:plugin_linear}
\sqrt{n}\,\mathrm{vec}\big(\Delta T_n\big)
=
B_{\mu,n}\,\sqrt{n}\,\delta_\mu
+
B_{\nu,n}\,\sqrt{n}\,\delta_\nu
+ o_p(1),
\end{equation}
where $\Delta T_n\coloneq T_n(\hat\mu,\hat\nu)-T_n(\mu,\nu)$ and
\[
\begin{aligned}
B_{\mu,n}&\coloneq \frac1n\sum_{i=1}^n \Big(D_\mu\phi_\mu(X_i)\Big)\otimes V_i,\\
B_{\nu,n}&\coloneq \frac1n\sum_{i=1}^n U_i \otimes \Big(D_\nu\psi_\nu(Y_i)\Big).
\end{aligned}
\]
By the law of large numbers and (A5)--(A6), $B_{\mu,n}\to_p A_\mu$ and $B_{\nu,n}\to_p A_\nu$, where
\begin{align}
A_\mu &\coloneq  \mbE\big[ V\otimes D_\mu\phi_\mu(X)\big],\label{eq:defA_mu}\\
A_\nu &\coloneq  \mbE\big[ U\otimes D_\nu\psi_\nu(Y)\big].\label{eq:defA_nu}
\end{align}
Thus
\begin{equation}\label{eq:plugin_limit}
\sqrt{n}\,\mathrm{vec}\big(\Delta T_n\big)
=
A_\mu\,\sqrt{n}\,\delta_\mu
+
A_\nu\,\sqrt{n}\,\delta_\nu
+ o_p(1).
\end{equation}
Substituting the $\sqrt{n}$-expansions of $\delta_\mu,\delta_\nu$ from Lemma~\ref{lem:mean_clt} yields
\begin{equation}\label{eq:plugin_influence}
\begin{aligned}
\sqrt{n}\,\mathrm{vec}\big(\Delta T_n\big)
&=
A_\mu H_X^{-1}\left(\frac1{\sqrt{n}}\sum_{i=1}^n \xi_i\right)\\
&\quad+
A_\nu H_Y^{-1}\left(\frac1{\sqrt{n}}\sum_{i=1}^n \zeta_i\right)
+ o_p(1).
\end{aligned}
\end{equation}

Finally, combining \eqref{eq:decomp_main}, the vectorized empirical term, and \eqref{eq:plugin_influence} gives the asymptotic linear representation \eqref{eq:final_if} with influence term \eqref{eq:IF}.
\end{proof}
\begin{proof}[Proof of Theorem~\ref{thm:Sigma_clt}]
By Lemma~\ref{lem:Sigma_if}, we have
\[
\begin{aligned}
\sqrt{n}\,\Delta_n
&=
\frac1{\sqrt{n}}\sum_{i=1}^n \IF_i + o_p(1).
\end{aligned}
\]
The summands $\IF_i$ are i.i.d.\ mean-zero. Under (A5)--(A6), they have finite second moments, hence by the multivariate CLT
$\frac1{\sqrt{n}}\sum_{i=1}^n \IF_i \Rightarrow \Normal(0,\Var(\IF_1))$.
Slutsky's theorem yields the stated limit with $\Xi=\Var(\IF_1)$.
\end{proof}
\section{Simulation Results}
The sphere experiments were conducted on a Macbook Air with an Apple M1 chip (8 cores: 4 performance and 4 efficiency), with a total wall-clock time of approximately 1 hours. The SPD experiments were run on a MacBook Pro with an Apple M4 Pro chip (14 cores: 10 performance and 4 efficiency) and 24 GB of unified memory, running macOS 26.3, with a total wall-clock time of approximately 3 hours. While the runtime for the real case study is negligible, the bootstrap method for Schizophrenia challenge dataset takes about 2 hours.

We utilze code from ManOpt \citep{JMLR:v17:16-177} (\url{http://jmlr.org/papers/v17/16-177.html}) and Geomstats \citep{miolane2020geomstats} (\url{https://www.jmlr.org/papers/v21/19-027.html}).
\subsection{Sphere}\label{sphere:simu}

\subsubsection{Data Generation and Experimental Setup}

We start by generating a random data set $\{X_i\}_{i=1}^n$ that is truncated uniformly distributed on $\mbS^2$ which provide a simple way to ensure the data are bounded. Fixing the center at the north pole, points with coordinate $(x,y,z)\in\mbR^3$ are generated as
\begin{align*}
    x &=\sqrt{1-z^2}\cos(2\pi v)\\ 
    y &=\sqrt{1-z^2}\sin(2\pi v)\\ 
    z &=1-w(1-\cos\tau)
\end{align*}

where $v,w\sim U(0,1)$ and $\tau$ controls the bound of the data. To guarantee the uniqueness of the Fr\'echet mean, we require $0<\tau<\pi/2$. The dependent dataset $\{Y_i\}_{i=1}^n$ is controlled by defining it in terms of a rotation of the first set. The rotation matrix, given by
\begin{align*}
    O(\eta) &= I + \sin(\eta)K+(1-\cos(\eta))K^2\\
    K &= \begin{bmatrix}
        0 & -u_z & u_y\\u_z & 0 & u_x\\-u_y & u_x & 0
    \end{bmatrix}
\end{align*}
is a realization of Rodrigues' rotation formula \citep{cheng1989historical}. Here $\eta$ denotes the angle of counterclockwise rotation about an axis $(u_x,u_y,u_z)$. Noise is then added to represent the random perturbation, after which the samples are parallel transported to the tangent space at another point $q$.

Therefore, the second sample set is generated as
\begin{align*}
    Y_i=\exp_q\Bigg[\Gamma_{\tilde\nu}^q\left(\frac{O(\eta)\log_{\hat\mu} X_i+\epsilon_i}{\|O(\eta)\log_{\hat\mu} X_i+\epsilon_i\|}\right)\Bigg],
\end{align*}
where $\epsilon_i\sim N_3(0,\sigma_\epsilon^2I_3)$, $\hat\mu$ is the Fr\'echet mean of the set $\{X_i\}_{i=1}^n$ and $\tilde\nu$ is the Fr\'echet mean of the set $\{Y_i\}_{i=1}^n$ before the transportation. We can expect $q=\hat\nu$ to be the Fr\'echet mean of transported data, as there's no additional noise introduced. 

\begin{figure}
    \centering
    \includegraphics[width=0.7\linewidth]{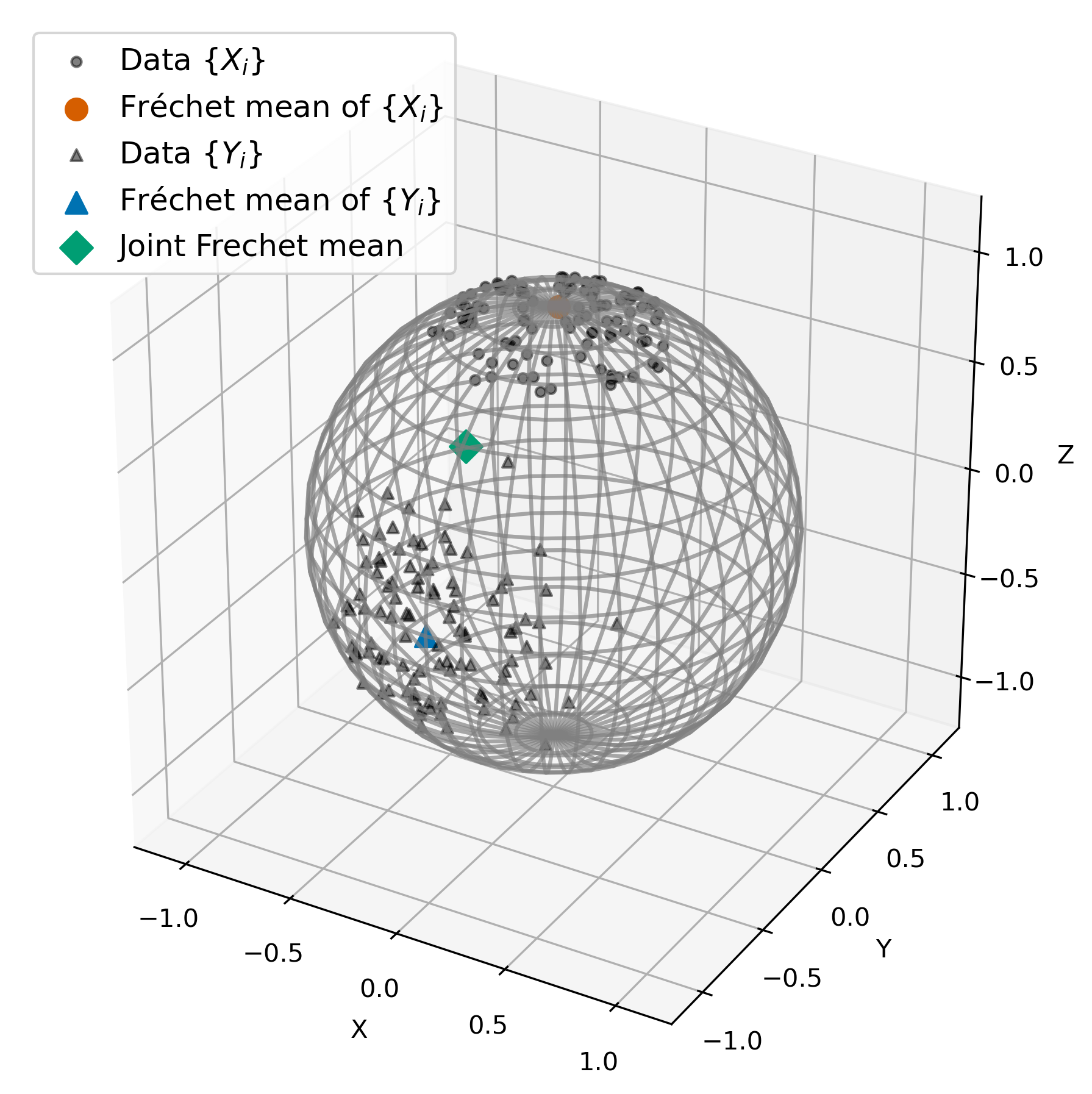}
    \caption{An illustration of a dataset $(n=100)$ on $\mbS^2$ and its dependent dataset, together with the respective Fr\'echet mean and the joint Fr\'echet mean (midpoint of two Fr\'echet mean).}
    \label{fig:Uni_Sphere}
\end{figure}

\subsubsection{Simulation on Transport effects}

\begin{figure}[h]
    \centering
    \includegraphics[width=0.8\linewidth]{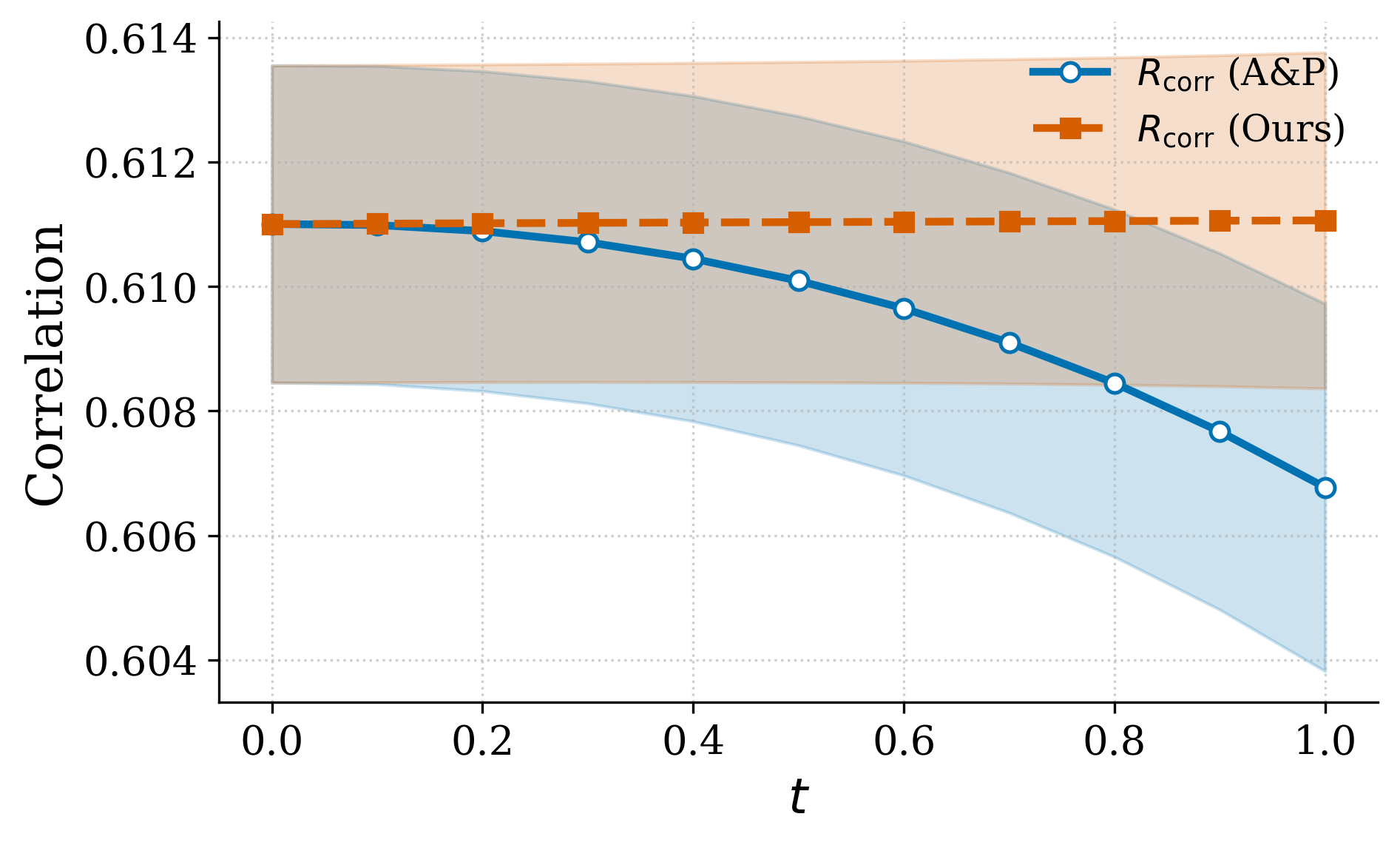}
    \caption{Comparison of Riemannian correlation (Ours), $A\&P$ correlation, and the CCA estimator. The set $\{Y_i\}$ is transported along the geodesic $\alpha$ with $\alpha(0)=\tilde\nu$ and $\alpha(1)=(0,-1,0)$. The x-axis denotes the position of Fr\'echet mean of $\{Y_i\}$ along $\alpha$. Parameters: $\tau=\pi/6$, $\eta=\pi/4$, $\sigma_\epsilon=0.15$, MC $=200$.}
    \label{fig:geodesic_sim}
\end{figure}

For Figure \ref{fig:geodesic_sim} evaluate the correlation estimates along a geodesic. The correlation is computed by transporting data $\{Y_i\}$ to different point along a geodesic connecting $\tilde\nu$ and $(0,-1,0)$. The shaded area represents the confidence interval of 2 standard errors. It can be seen that 
our estimates closely matches the $A\&P$ correlation when two data sets are close, but diverge as they are moved far apart. This indicates that the distance between datasets could affect the $A\&P$ estimate when it is evaluated on the tangent space of joint Fr\'echet mean.

\begin{figure}[h]
    \centering
    \includegraphics[width=0.8\linewidth]{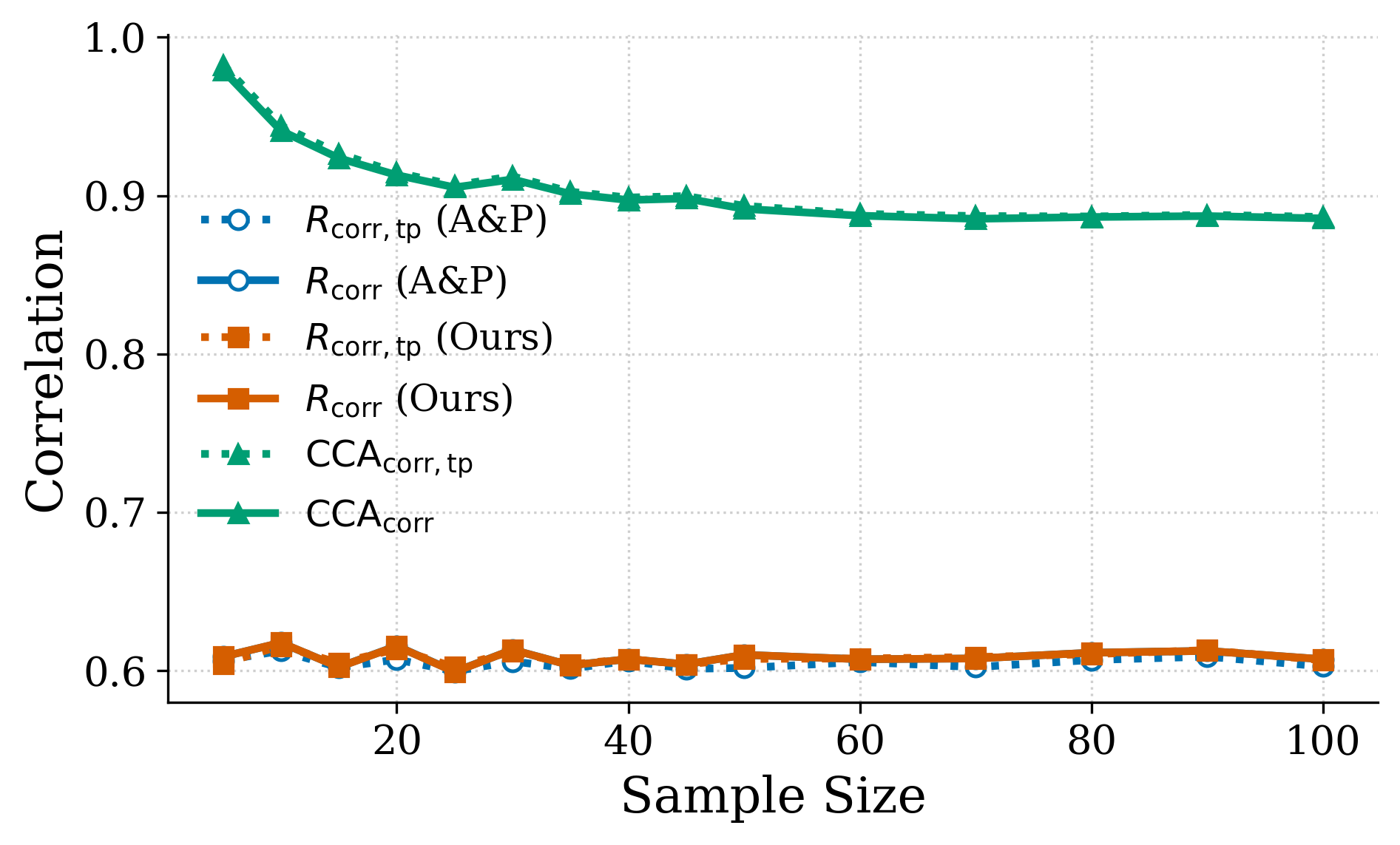}
    \caption{Comparison of Riemannian correlation, $A\&P$ correlation and CCA estimator. Solid line correspond to samples sharing the same Fr\'echet mean while dotted line represents results from data $\{Y_i\}$ transported to $(0,-1,0)$. Parameters: $\tau=\pi/6$, $\eta=\pi/4$, with noise level $\sigma_\epsilon=0.15$.}
    \label{fig:add_convergence}
\end{figure}

In Figure \ref{fig:add_convergence}, we compare our proposed Riemannian correlation with $A\&P$ correlation and CCA estimator. Solid lines correspond to the original set $\{Y_i\}_{i=1}^n$ without transport, while dotted lines shows the results after transporting $\{Y_i\}_{i=1}^n$ to $(0,-1,0)$. The CCA estimator gives the same estimation before and after transportation, but it consistently overestimates the correlation, as it removes the effect of the rotation matrix.

\subsubsection{Sensitivity to Rotation and Independence}

\begin{figure}
    \centering
    \begin{tabular}{@{}c@{} @{}c@{}}
            \includegraphics[width=0.49\linewidth]{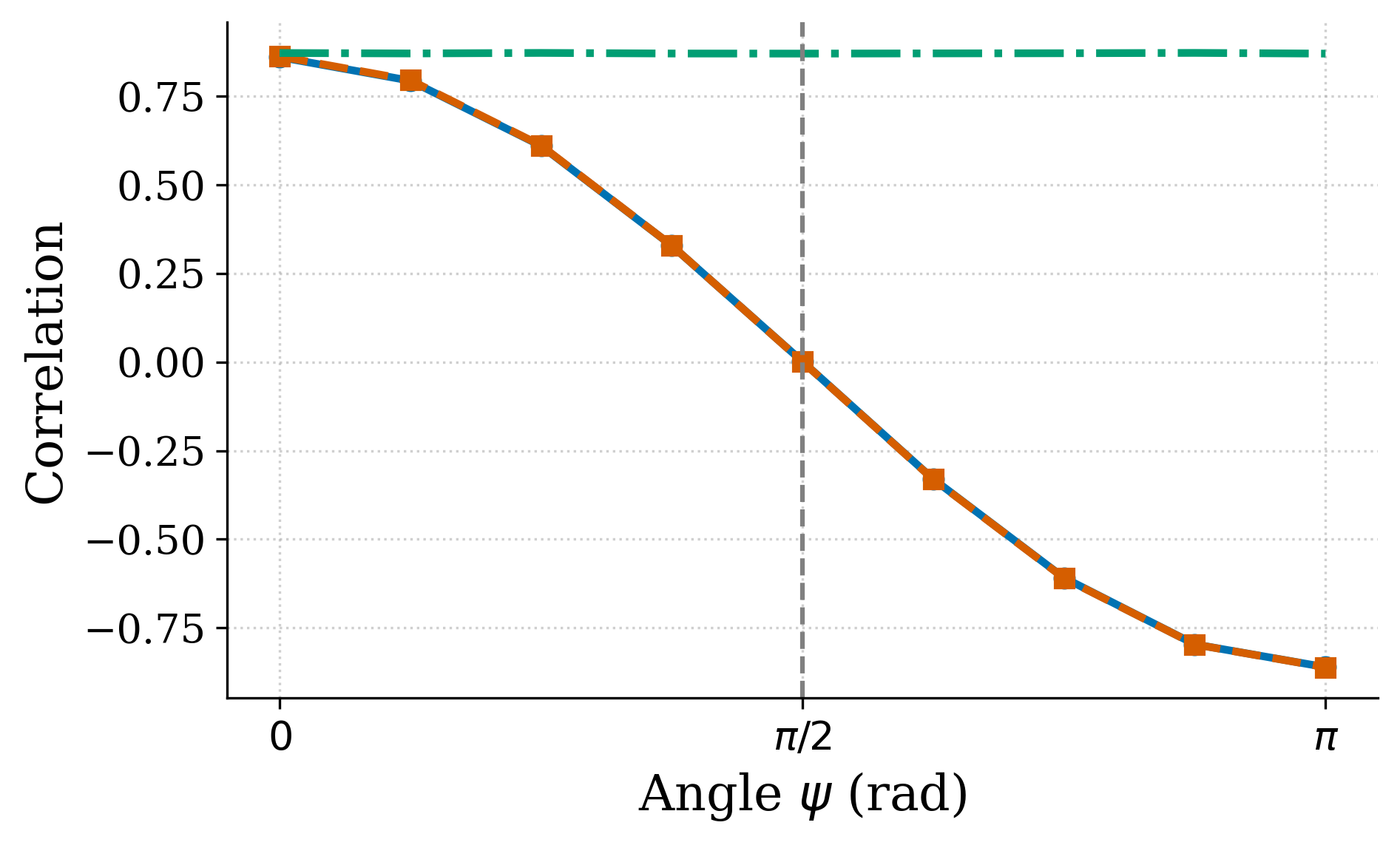}&
            \includegraphics[width=0.49\linewidth]{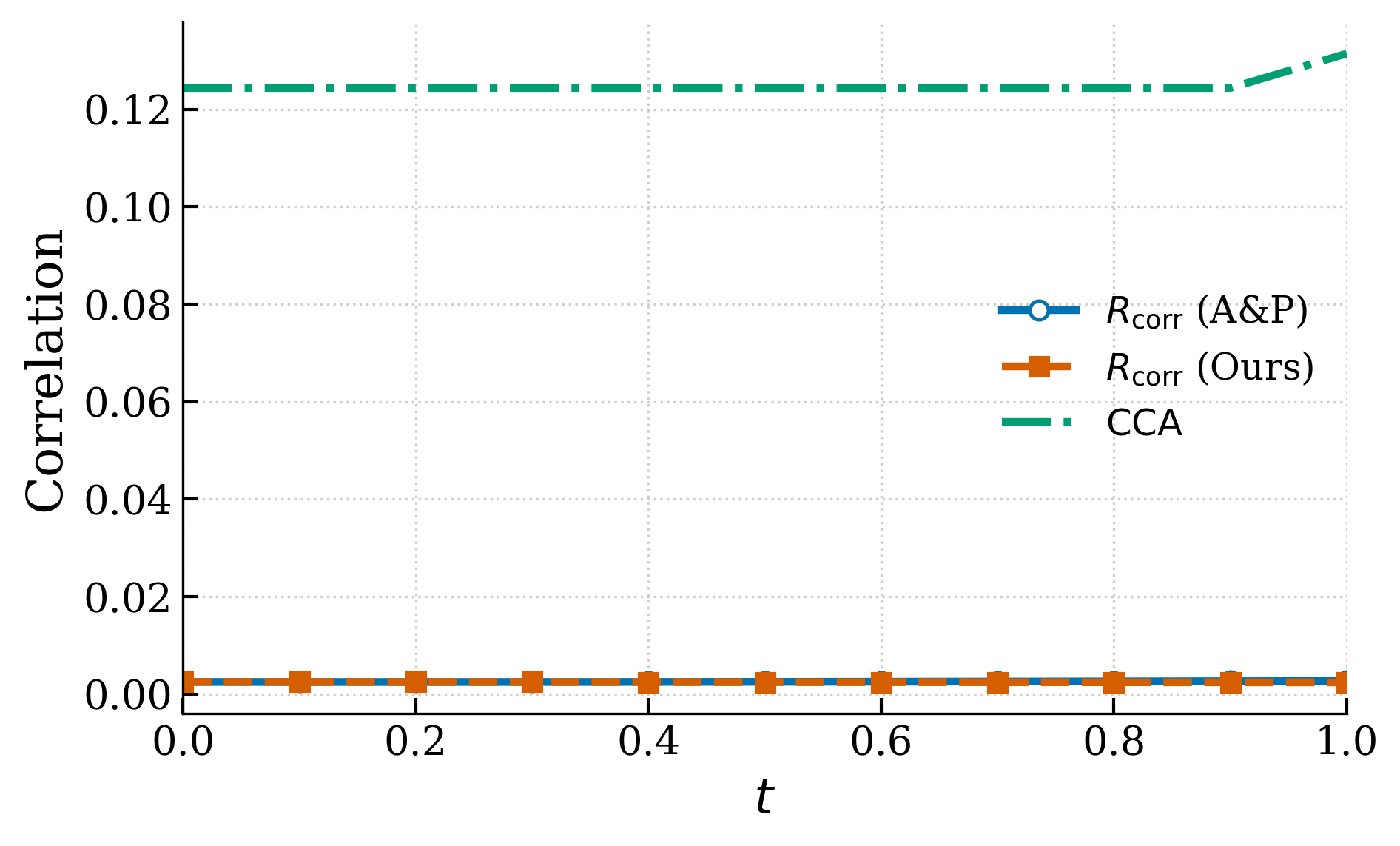}
    \end{tabular}
    \caption{Comparison of Riemannian correlation (Ours), $A\&P$ correlation, and the CCA estimator. Parameter is set as $\tau=\pi/6$. \textbf{Left:} The set $\{Y_i\}$ is transported to $\alpha(1/2)$, with $\alpha(0)=\tilde\nu$ and $\alpha(1)=(0,-1,0)$. The x-axis shows rotation angle with the vertical dashed line indicating $\pi/2$. The noise level is $\sigma_\epsilon=0.15$. \textbf{Right:} Two independent sets with $\{Y_i\}$ transported along the geodesic connecting its Fr\'echet mean and (0,-1,0). The x-axis denoted the position of Fr\'echet mean of $\{Y_i\}$ along $\alpha$.}
    \label{fig:angle_sim}
\end{figure}

In Figure \ref{fig:angle_sim}, we explore the effect of different rotation matrices on the samples, as shown in the left hand plot. The samples $\{Y_i\}_{i=1}^n$ are generated by rotating $\{X_i\}_{i=1}^n$ by various angles and transported to the midpoint of geodesic connecting $\tilde\nu$ and $(0,-1,0)$. When no noise is added, we expect to have $\mcR_{X,Y}(\hat\nu)=1$ for $\eta=0$ and $\mcR_{X,Y}(\hat\nu)=0$ for $\eta=\frac{\pi}{2}$, as two datasets form orthogonal vectors in a tangent space. Similar as before, CCA consistently ignores the rotation and overestimates the correlation. Our method has estimates similar to those of the $A\&P$ approach, and It can be checked that estimates are close to 0 at $\frac{\pi}{2}$.

Finally, we test the estimators on two independent datasets. The datasets are generated using a pseudo-random number generator (PRNG). Though they are not mathematically independent in a theoretical sense, the independence assumption is sufficient for simulation purpose. One dataset is transported along the geodesic, with result shown on the right of Figure \ref{fig:angle_sim}. As can be seen, our method always have estimates around 0, indicating the reliability of the method. The $A\&P$ method also has the similar pattern. Since CCA is defined to maximize an optimization problem, it produces slight higher correlation estimates.
\subsubsection{Ground-truth Recovery Experiment}

We here conduct a recovery simulation on $\mbS^2$. In the tangent space at the north pole $(0,0,1)$, we generate two sets of normally distributed samples with a known covariance matrix. The samples take the form $X=(x_1,x_2,0)$ and $Y=(y_1,y_2,0)$, where $(x_1,x_2,y_1,y_2)\sim\mcN_4(0, \Sigma)$ with
\begin{align*}
    \Sigma=\begin{pmatrix}
    0.1I_2 & 0.04J_2\\
    0.04J_2 & 0.08I_2
    \end{pmatrix}, \quad J_2=\begin{pmatrix}
        1 & 1\\ 1 & 1
    \end{pmatrix}.
\end{align*}

The cross-covariance matrix of the two sets is the $3\times 3$ matrix whose top-left block is $0.04J_2$ and all other entries are 0. The correlation, by construction, is $0.08/\sqrt{0.2\cdot 0.16}\approx 0.447$. We generate 1000 pairs of data and repeat the experiment 100 times. Table~\ref{tab:recovery_corr} gives a numerical comparison of the methods addressed in the paper. The geodesic $\gamma$ is defined such that $\gamma(0)$ is the joint Fr\'echet mean of $\{\exp_{(0,0,1)}X_i\}$ and $\{\exp_{(0,0,1)}Y_i\}$, and $\gamma(1)=(0,1,0)$.

\begin{table}[h]
    \centering
    \caption{Covariance and correlation estimates for recovery experiment}
    \small
    \begin{tabular}{cccc}
    \toprule
         & $\gamma(0)$ & $\gamma(0.5)$ & $\gamma(1)$ \\
    \midrule
       $R_{cov}(Ours)$  & 0.079 & 0.079 & 0.079 \\
       $R_{cov}(A\&P)$ & 0.079 & 0.087 & 0.081 \\
       $R_{corr}(Ours)$ & 0.446 & 0.446 & 0.446 \\
       $R_{corr}(A\&P)$ & 0.446 & 0.438 & 0.445 \\
       $CCA$ & 0.894 & 0.894 & 0.894\\
    \bottomrule
    \end{tabular}
    
    \label{tab:recovery_corr}
\end{table}

Specifically, the covariance matrix given by our method (at Fr\'echet mean of X set) is \begin{align*}
    \hat\Sigma=\begin{pmatrix}
        0.0395 & 0.0396 & -5.96\times 10^{-6}\\
        0.0398 & 0.0399 & -7.07\times 10^{-6}\\
        4.715\times 10^{-8} & -2.692\times 10^{-6} & 9.971\times 10^{-6}
    \end{pmatrix},
\end{align*}
which is very close to the ground truth covariance matrix.
\subsection{SPD Manifold}\label{app:spd_sim}

\subsubsection{Data Generation and Experimental Setup}
To rigorously assess the proposed estimator, we employ a generative model that decouples intrinsic geometric correlation from ambient curvature effects. The manifold of $d \times d$ Symmetric Positive Definite (SPD) matrices, $\mathcal{P}_d$, is endowed with the Affine Invariant Riemannian Metric (AIRM). We fix the dimensionality at $d=3$ and the sample size at $n=100$ for most experiments unless otherwise specified.

The primary dataset $\{X_i\}_{i=1}^n$ is generated by perturbing the identity matrix $I_d$. Tangent vectors $v_{X,i} \in T_{I_d}\mathcal{P}_d$ are sampled from a symmetric normal distribution with spread parameter $\tau=0.5$ and subsequently mapped to the manifold via the Riemannian exponential map, denoted as $X_i = \text{Exp}_{I_d}(v_{X,i})$.

The dependent dataset $\{Y_i\}_{i=1}^n$ is constructed via a two-stage process to ensure that the correlation structure is intrinsic and transport-invariant. First, we generate tangent vectors $v_{Y,i} \in T_{I_d}\mathcal{P}_d$ that are locally correlated with $v_{X,i}$ according to the relation $v_{Y,i} = \rho \cdot v_{X,i} + \sqrt{1-\rho^2} \cdot \epsilon_i$. Here, $\rho \in [0,1]$ controls the correlation strength, and $\epsilon_i$ is a random symmetric matrix representing noise, normalized such that $\|\epsilon_i\|_{I_d} = \|v_{X,i}\|_{I_d}$. This normalization ensures that the dispersion of $Y$ matches that of $X$. Subsequently, the entire vector cloud $\{v_{Y,i}\}$ is parallel transported from $I_d$ to a target mean $\nu$ along the geodesic connecting them. Let $\Gamma_{I_d}^\nu$ denote the parallel transport operator; the final samples are given by $Y_i = \text{Exp}_{\nu}\left( \Gamma_{I_d}^\nu( v_{Y,i} ) \right)$. By generating correlation locally and then transporting the distribution, we ensure that any observed decorrelation in the estimators is attributable to the estimation method itself rather than artifacts of the data generation process in curved space.
\begin{figure}[h]
    \centering
    \begin{tabular}{@{}c@{} @{}c@{}}
        \includegraphics[width=0.49\linewidth]{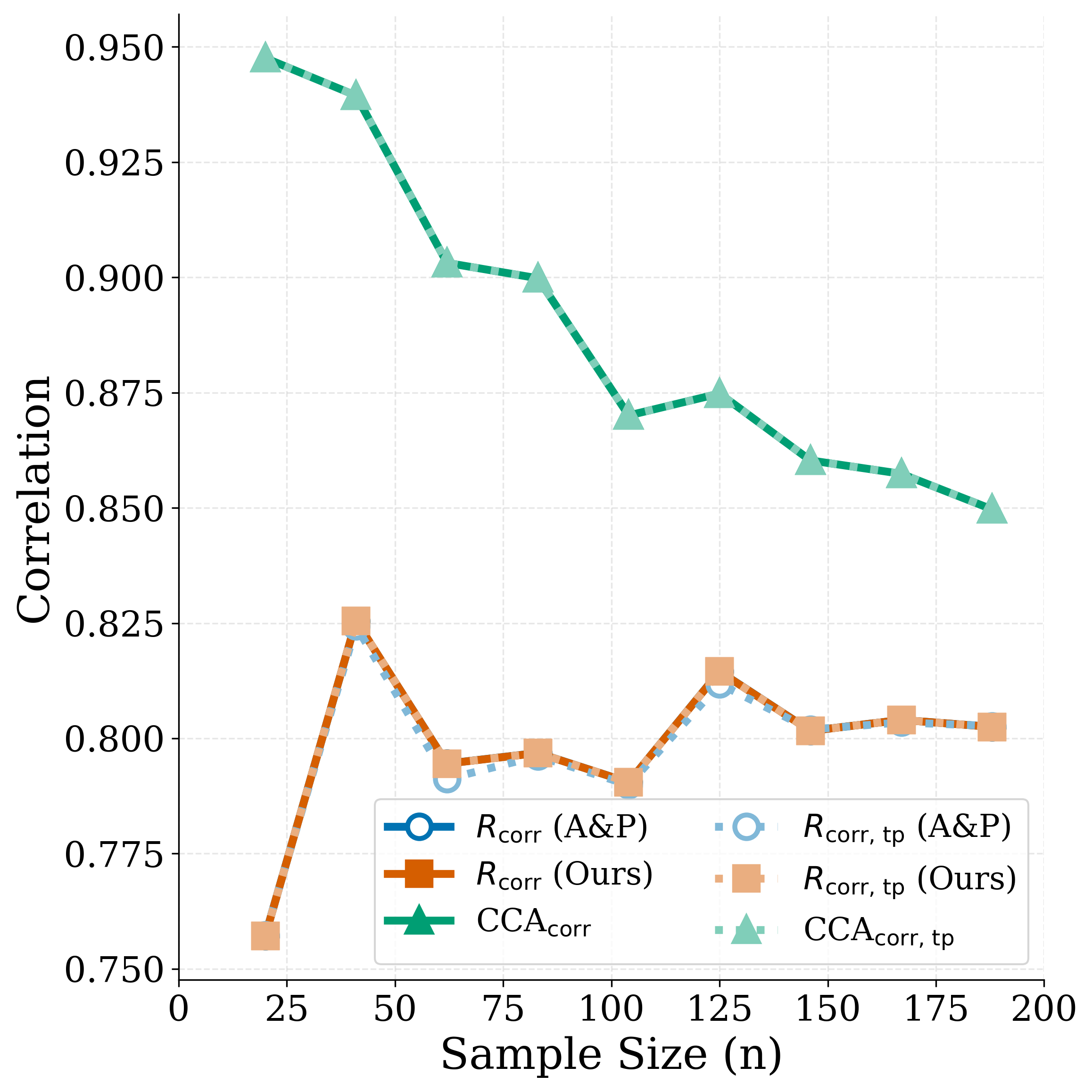}&
        \includegraphics[width=0.49\linewidth]{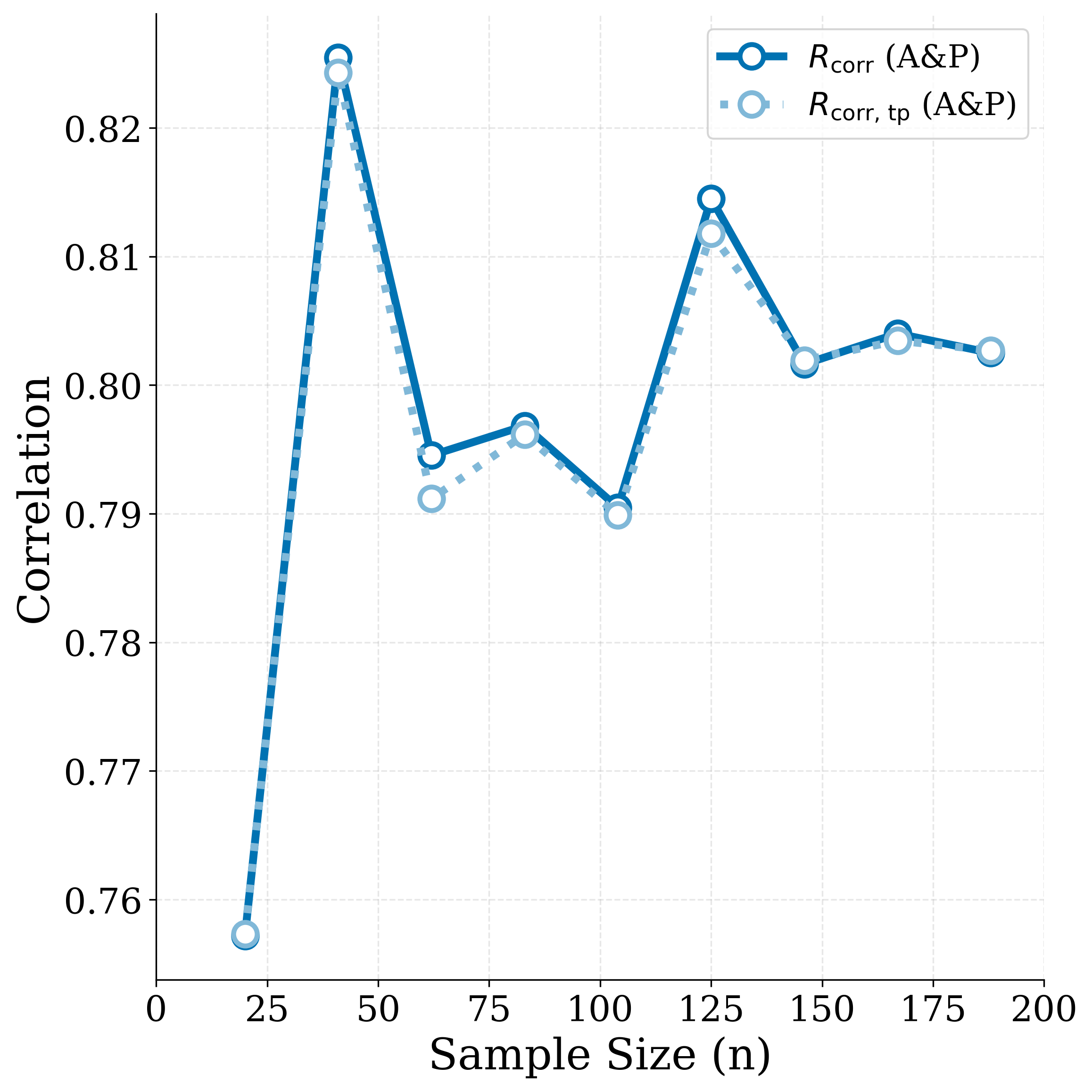}
    \end{tabular}
    \caption{Convergence of Correlation Estimates on SPD Manifold. \textbf{Left:} Comparison of A\&P, Ours, and CCA methods with increasing sample size. Solid lines: original dataset (non-transport); dashed lines: parallel-transported dataset. Our method produces nearly identical estimates for non-transport and transport scenarios, consistent with isometric invariance. Log-log regression suggests $\sqrt{n}$-convergence (slope $-0.49 \pm 0.03$). \textbf{Right:} Detailed A\&P vs Ours comparison. Shaded regions show 95\% confidence intervals. Parameters: $d=3, \tau=0.5, \sigma_\epsilon=0.15$, MC varies by $n$.}
    \label{fig:spd_convergence}
\end{figure}

\subsubsection{Asymptotic Consistency and Convergence}
Asymptotic consistency and isometric invariance are assessed in Figure \ref{fig:spd_convergence}. Log-log regression analysis suggests $\sqrt{n}$-convergence for both methods: the convergence rate estimates are $-0.49 \pm 0.03$ (Ours) and $-0.51 \pm 0.04$ (A\&P), which are consistent with the theoretical rate of $-0.5$ (95\% confidence intervals include $-0.5$). The proposed method produces nearly identical estimates for the original and parallel-transported datasets across all sample sizes, which is consistent with the estimator capturing the intrinsic geometry invariant under isometry. This is a prerequisite for robust statistical inference on manifolds independent of the ambient coordinate system. The right panel provides a detailed comparison between A\&P and Ours, highlighting the systematic differences in convergence behavior.

\subsubsection{Geodesic Simulation and Transport Effects}
Figure \ref{fig:spd_geodesic_sim} examines the estimator's behavior under geometric transformations through two complementary experiments that test distinct theoretical properties. The left panel tests isometric invariance by physically transporting the dataset $Y$ along a geodesic away from $X$ and evaluating correlation at a fixed reference point. As expected from the definition of parallel transport (an isometry preserving inner products), the intrinsic correlation should remain invariant, and our estimator exhibits low variation (standard error $SE \approx 0.008$). The right panel tests footpoint invariance (Theorem \ref{thm:ftinvar}) by fixing the data and evaluating correlation at different analytical points $p$ along the geodesic connecting the sample means. Our estimator produces nearly constant estimates ($SE \approx 0.008$) across all evaluation points, while the A\&P estimator exhibits substantial variation ($SE \approx 0.15$), depending on where it is computed. This dependence on an arbitrary reference tangent space is undesirable for a summary statistic, whereas our method renders results independent of this choice.

\begin{figure}[h]
    \centering
    \begin{tabular}{@{}c@{} @{}c@{}}
        \includegraphics[width=0.49\linewidth]{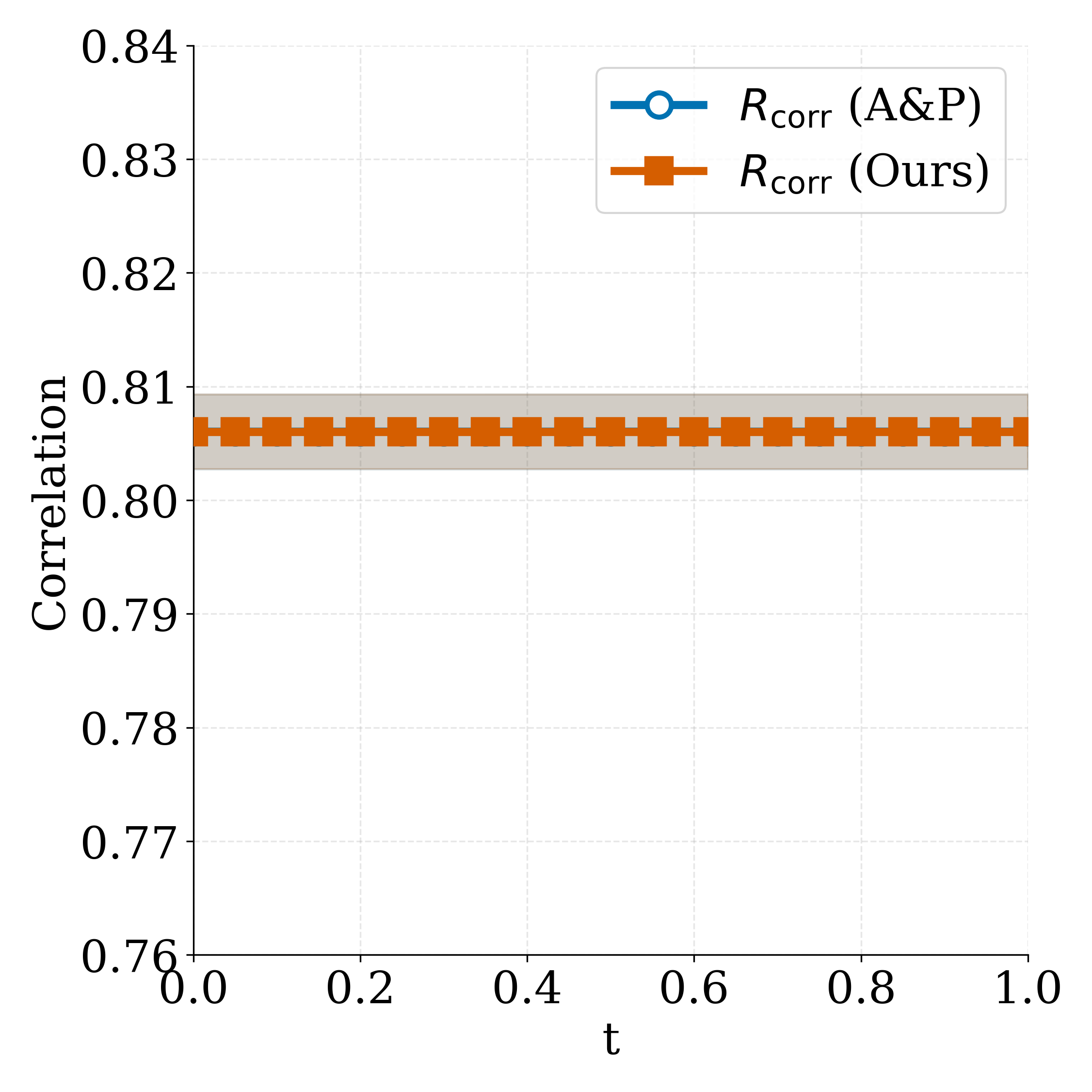}&
        \includegraphics[width=0.49\linewidth]{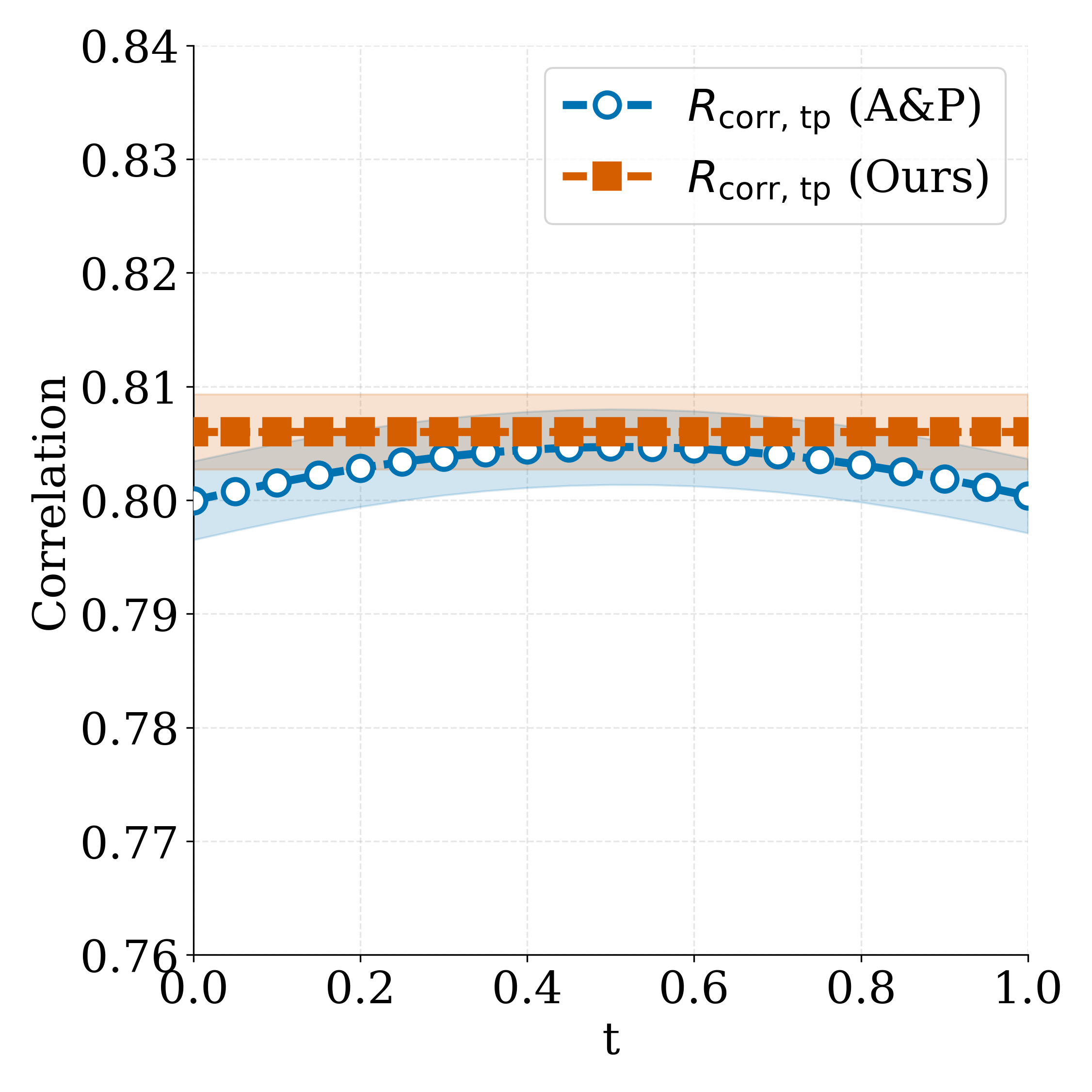}
    \end{tabular}
    \caption{Geodesic Evaluation on SPD Manifold. \textbf{Left:} Original dataset (non-transport). \textbf{Right:} Parallel-transported dataset (transport). Our method produces nearly constant estimates ($SE \approx 0.008$) across all evaluation points, consistent with footpoint invariance (Theorem \ref{thm:ftinvar}). The A\&P method exhibits substantial variation ($SE \approx 0.15$) with an inverted U-shaped pattern. Shaded regions show 95\% confidence intervals (mean $\pm$ 2 SE). Parameters: $n=100, d=3, \tau=0.5, \sigma_\epsilon=0.15$, MC $=200$.}
    \label{fig:spd_geodesic_eval}
\end{figure}

\begin{figure}[h]
    \centering
    \begin{tabular}{@{}c@{} @{}c@{} @{}c@{}}
        \includegraphics[width=0.3\linewidth]{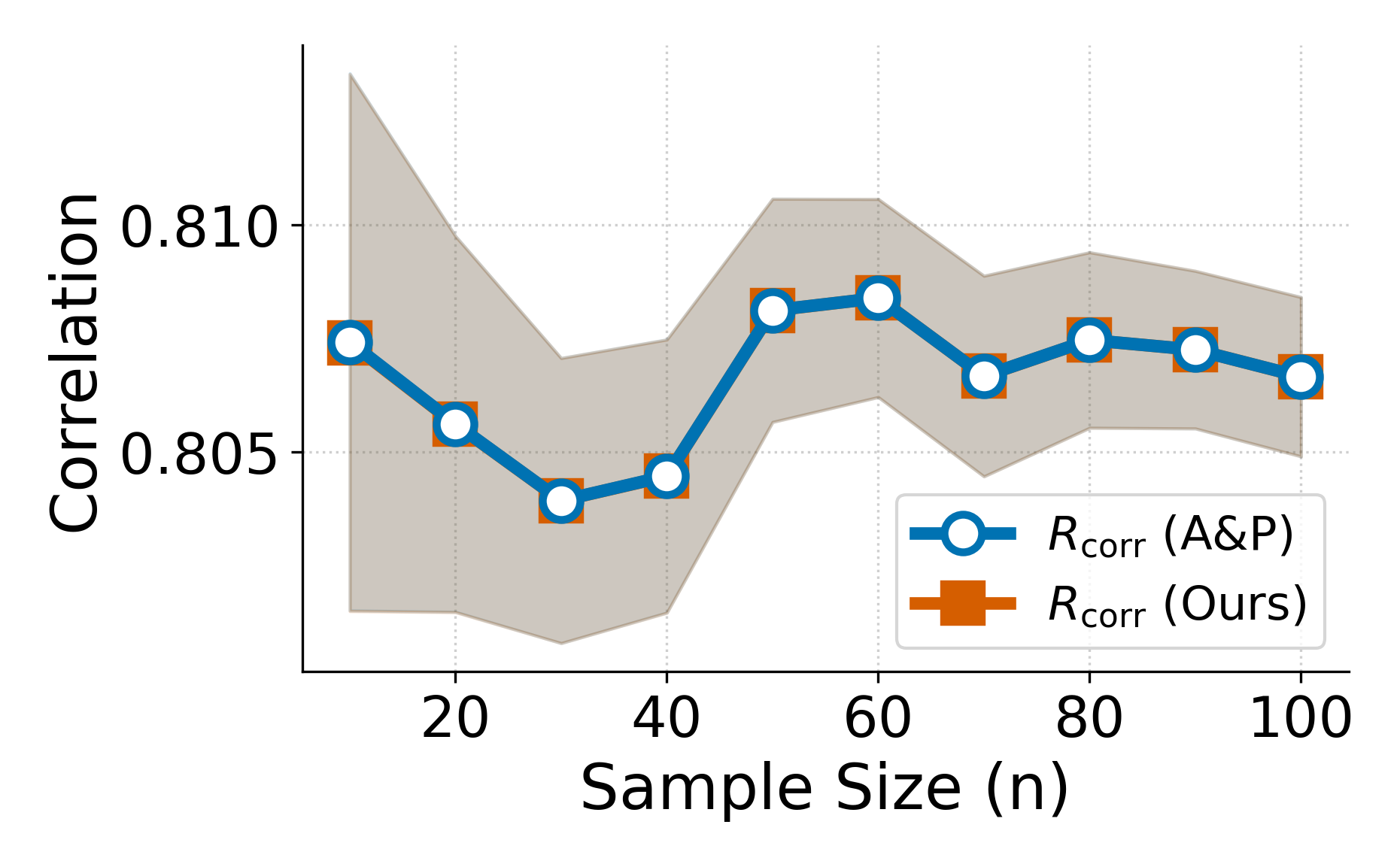}&
        \includegraphics[width=0.3\linewidth]{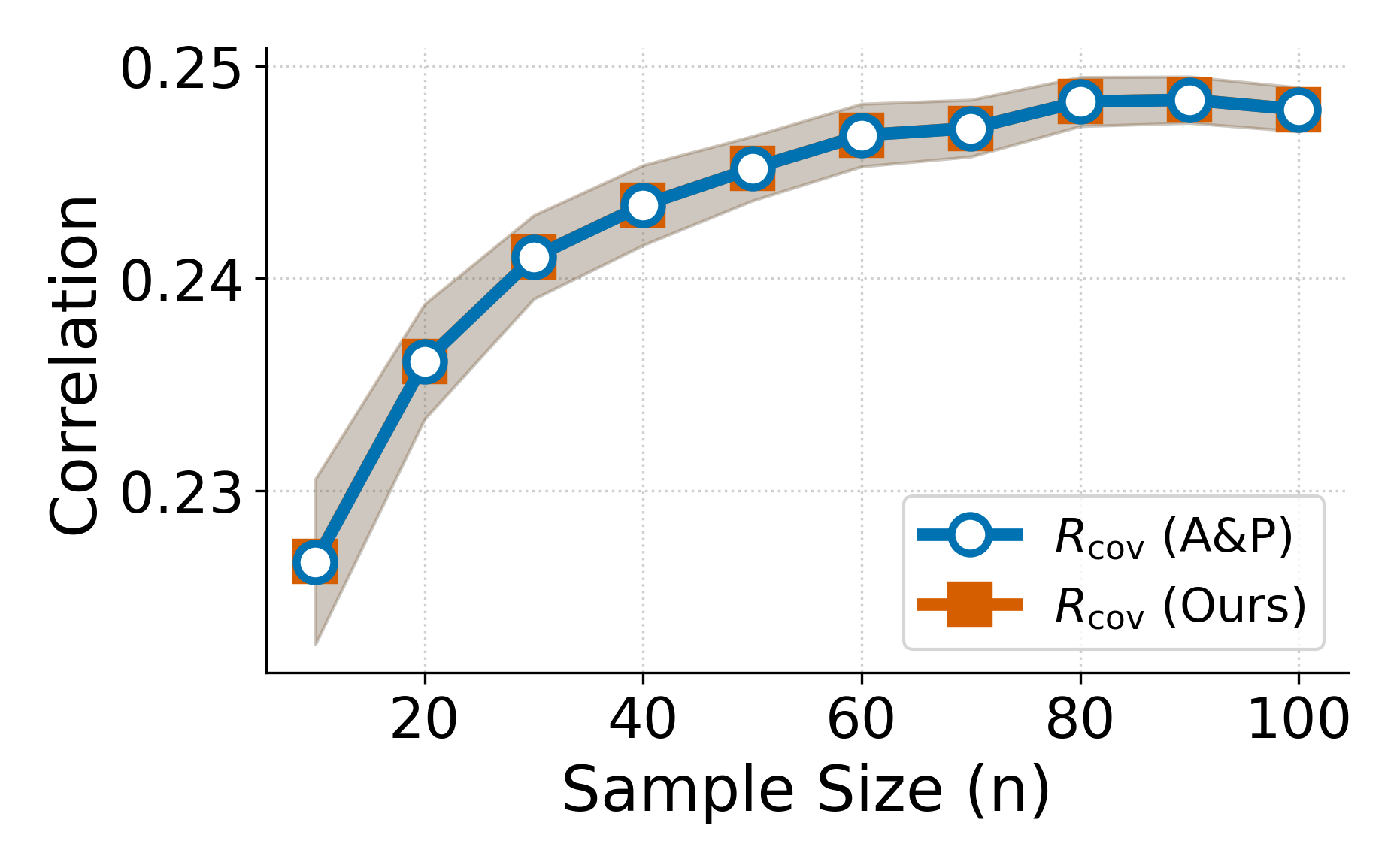}&
        \includegraphics[width=0.3\linewidth]{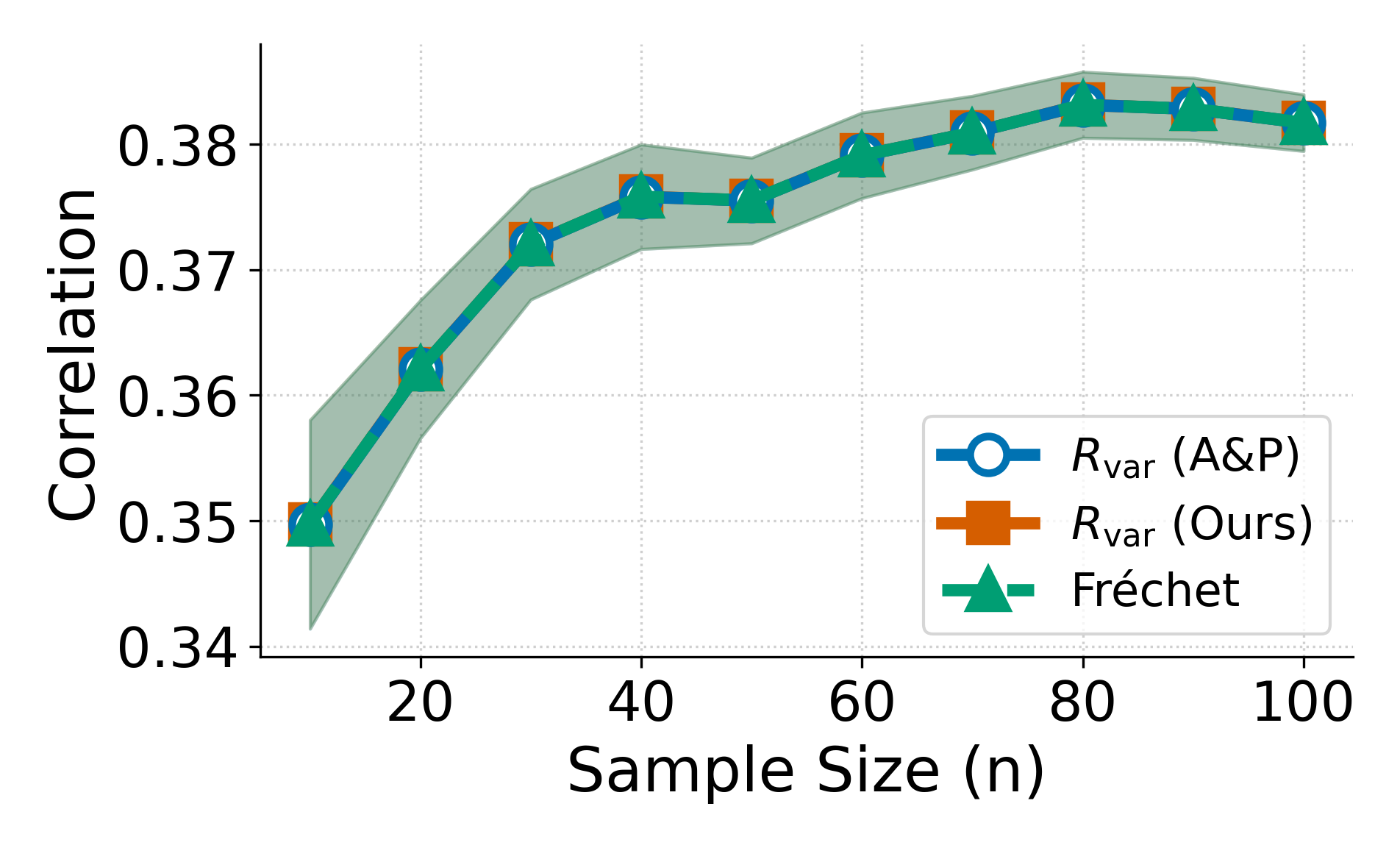}\\
        \includegraphics[width=0.3\linewidth]{Image_Neurips/SPDM/Correlation_vs_Sample_Size_Transport.png}&
        \includegraphics[width=0.3\linewidth]{Image_Neurips/SPDM/Covariance_vs_Sample_Size_Transport.png}&
        \includegraphics[width=0.3\linewidth]{Image_Neurips/SPDM/Variance_vs_Sample_Size_Transport.png}
    \end{tabular}
    \caption{Comprehensive Comparison on SPD Manifold. \textbf{Top Row:} Original dataset (non-transport). \textbf{Bottom Row:} Parallel-transported dataset (transport). \textbf{Columns:} Correlation, covariance, and variance estimates versus sample size. Our method closely matches Fréchet variance, while A\&P yields inflated variance (ratio $= 1.23 \pm 0.05$ at $n=100$), associated with deflated correlation. Our method produces nearly identical estimates for non-transport/transport scenarios, consistent with isometric invariance. Shaded regions indicate 95\% confidence intervals (mean $\pm$ 2 SE). Parameters: $n \in \{10,20,\ldots,100\}, d=3, \tau=0.5, \sigma_\epsilon=0.15$, MC $=200$ per $n$.}
    \label{fig:spd_comprehensive}
\end{figure}

\begin{figure}[htbp]
    \centering
    \begin{tabular}{@{}c@{} @{}c@{}}
        \includegraphics[width=0.49\linewidth]{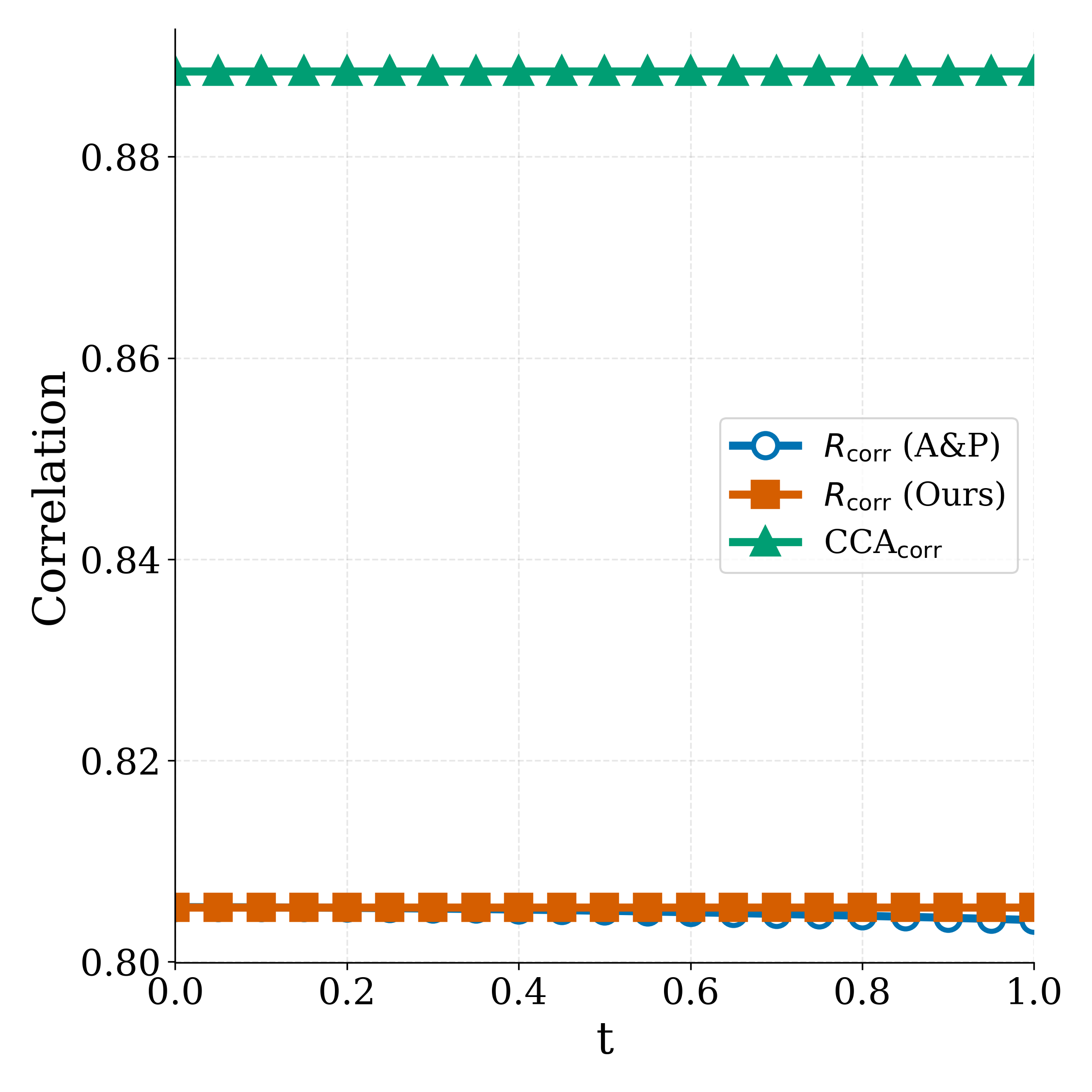}&
        \includegraphics[width=0.49\linewidth]{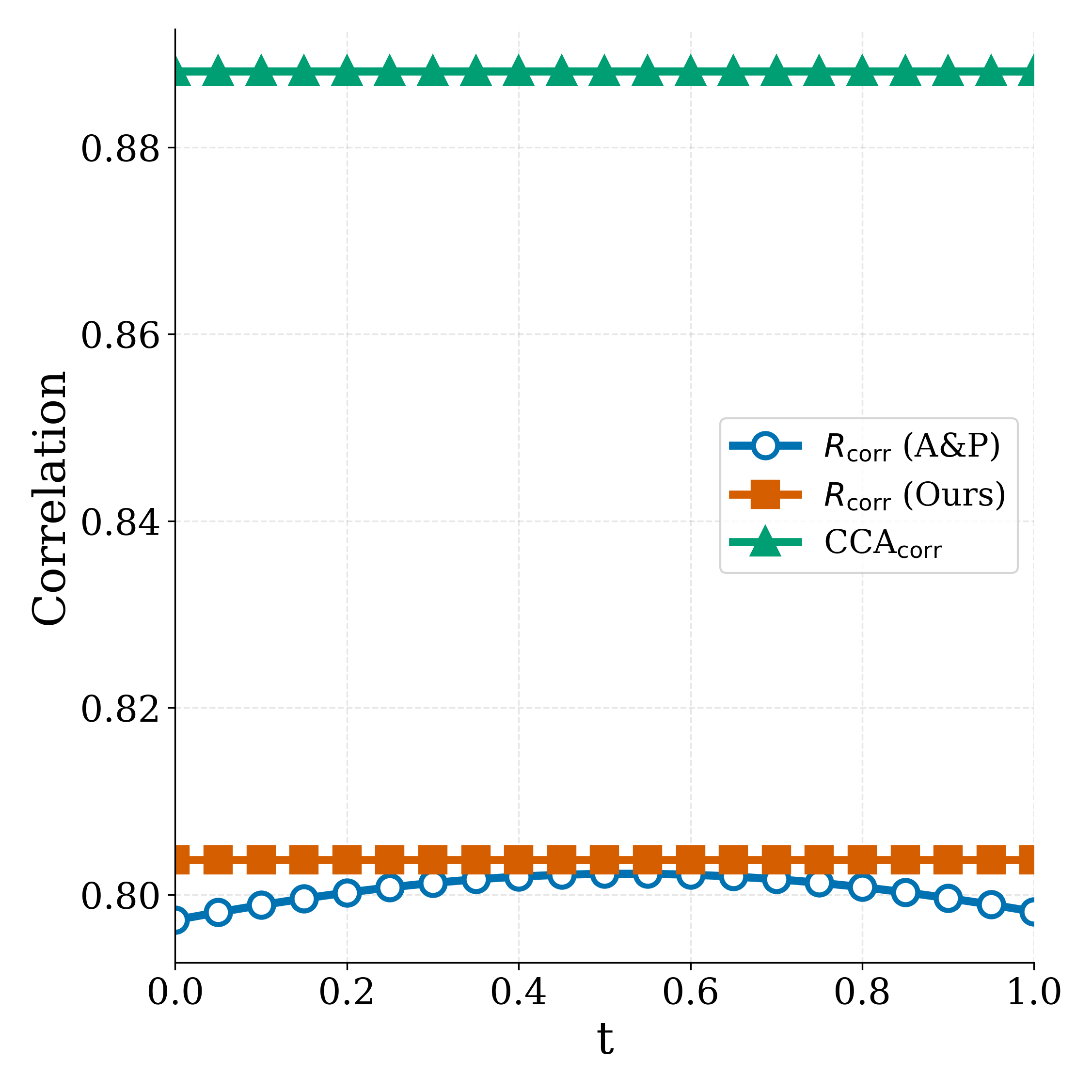}
    \end{tabular}
    \caption{Geodesic Simulation and Transport Effects on SPD Manifold. \textbf{Left:} Isometric invariance test: correlation estimates remain stable ($SE \approx 0.008$) when dataset $Y$ is physically transported along a geodesic. \textbf{Right:} Footpoint invariance test (Theorem \ref{thm:ftinvar}): correlation estimates remain nearly constant ($SE \approx 0.008$) when evaluated at different points $p$ along the geodesic. The A\&P estimator varies substantially ($SE \approx 0.15$), depending on the evaluation point. Parameters: $n=100, d=3, \tau=0.5, \sigma_\epsilon=0.15$, MC $=50$.}
    \label{fig:spd_geodesic_sim}
\end{figure}

\subsubsection{Sensitivity to Rotation and Independence}
We analyze the sensitivity of the estimator to the underlying dependence structure using tangent space rotations. Figure \ref{fig:spd_angle_effect} demonstrates that the correlation estimate decreases monotonically as the rotation angle $\eta$ approaches $\pi/2$, reflecting the geometric decorrelation. The relationship follows the expected pattern $\mathcal{R}_{X,Y}(\eta) \approx \cos(\eta)$ in the noise-free limit, with our method closely tracking this theoretical curve (mean absolute error $\approx 0.04$ across all angles). Furthermore, under the null hypothesis of independence, the estimator yields values close to zero (mean correlation $= 0.002 \pm 0.008$), which is consistent with the expected behavior for independent datasets. The stability across transport locations (right panel) further supports the estimator's invariance properties.

\begin{figure}[htbp]
    \centering
    \begin{tabular}{@{}c@{} @{}c@{}}
        \includegraphics[width=0.49\linewidth]{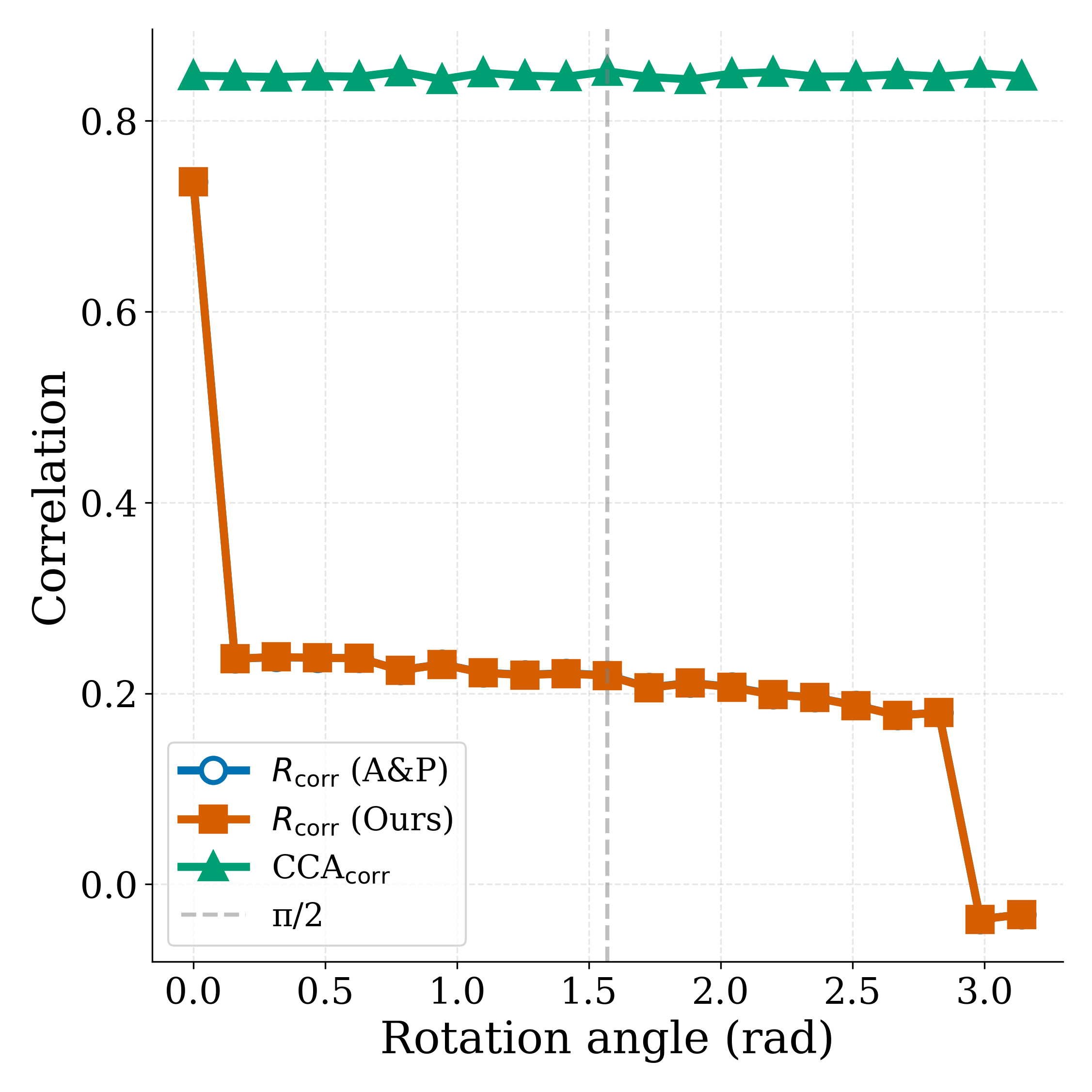}&
        \includegraphics[width=0.49\linewidth]{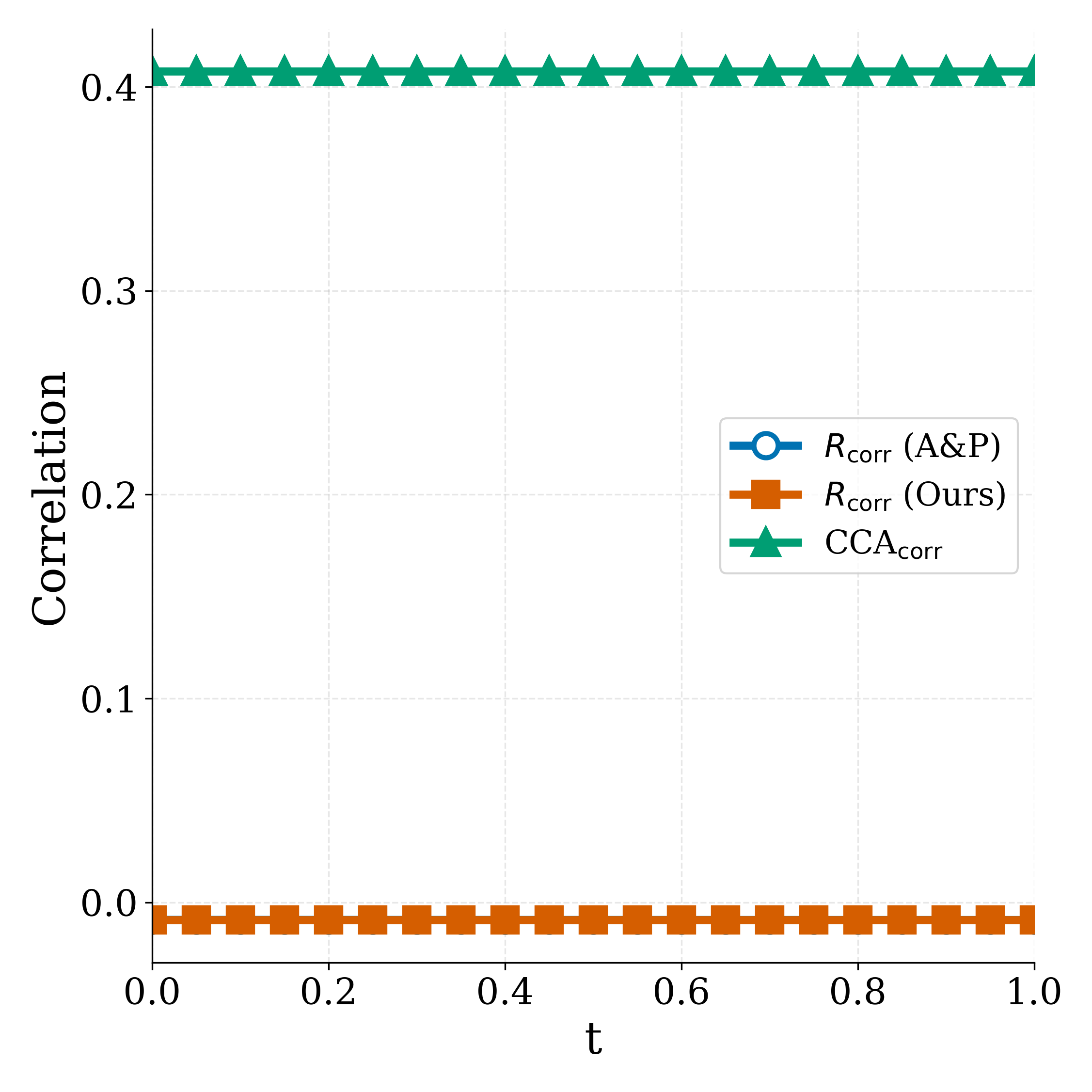}
    \end{tabular}
    \caption{Angle Effect and Independence on SPD Manifold. \textbf{Left:} Correlation decreases monotonically with rotation angle $\eta$, following $\cos(\eta)$ pattern (mean absolute error $\approx 0.04$). \textbf{Right:} Independent datasets yield correlation $= 0.002 \pm 0.008$, consistent with expected behavior for independent data. Parameters: $n=100, d=3, \tau=0.5, \sigma_\epsilon=0.15$, MC $=60$.}
    \label{fig:spd_angle_effect}
\end{figure}

\subsubsection{Effect of Mean Separation}
A critical challenge in non-positively curved spaces like $\mathcal{P}_d$ is that estimators relying on a single reference frame become increasingly biased as the geodesic distance $d(\mu_X, \mu_Y)$ between population means increases. This occurs because the logarithmic map $\log_p(\cdot)$ at a fixed footpoint distorts the metric structure differently as the data moves farther from the reference point. Figure \ref{fig:spd_mean_separation} systematically examines how both estimators behave as the mean separation increases, while maintaining a fixed intrinsic correlation structure ($\rho = 0.8$). The Y dataset is generated with correlation to X at the source location, then transported as a unit to a target mean $\nu$ at varying geodesic distances from $\mu_X$. This design ensures that any observed decorrelation is attributable to the estimation method rather than the data generation process.

The results reveal a fundamental difference: while the A\&P estimator (evaluated at the geodesic midpoint) exhibits systematic bias that increases monotonically with separation distance, our proposed method remains stable and accurately recovers the intrinsic correlation (mean correlation $= 0.798 \pm 0.012$ across all separation distances, with no significant trend). At maximum separation ($d = 2.0$), the A\&P method underestimates correlation by $0.23 \pm 0.05$ compared to the true value. The difference plot highlights the divergence, showing that the A\&P method increasingly underestimates correlation as means separate, which is associated with variance inflation (variance ratio increases from 1.05 at $d=0$ to 1.35 at $d=2.0$) due to metric distortion at the midpoint evaluation point.

\begin{figure}[htbp]
    \centering
    \begin{tabular}{@{}c@{} @{}c@{} @{}c@{}}
        \includegraphics[width=0.32\linewidth]{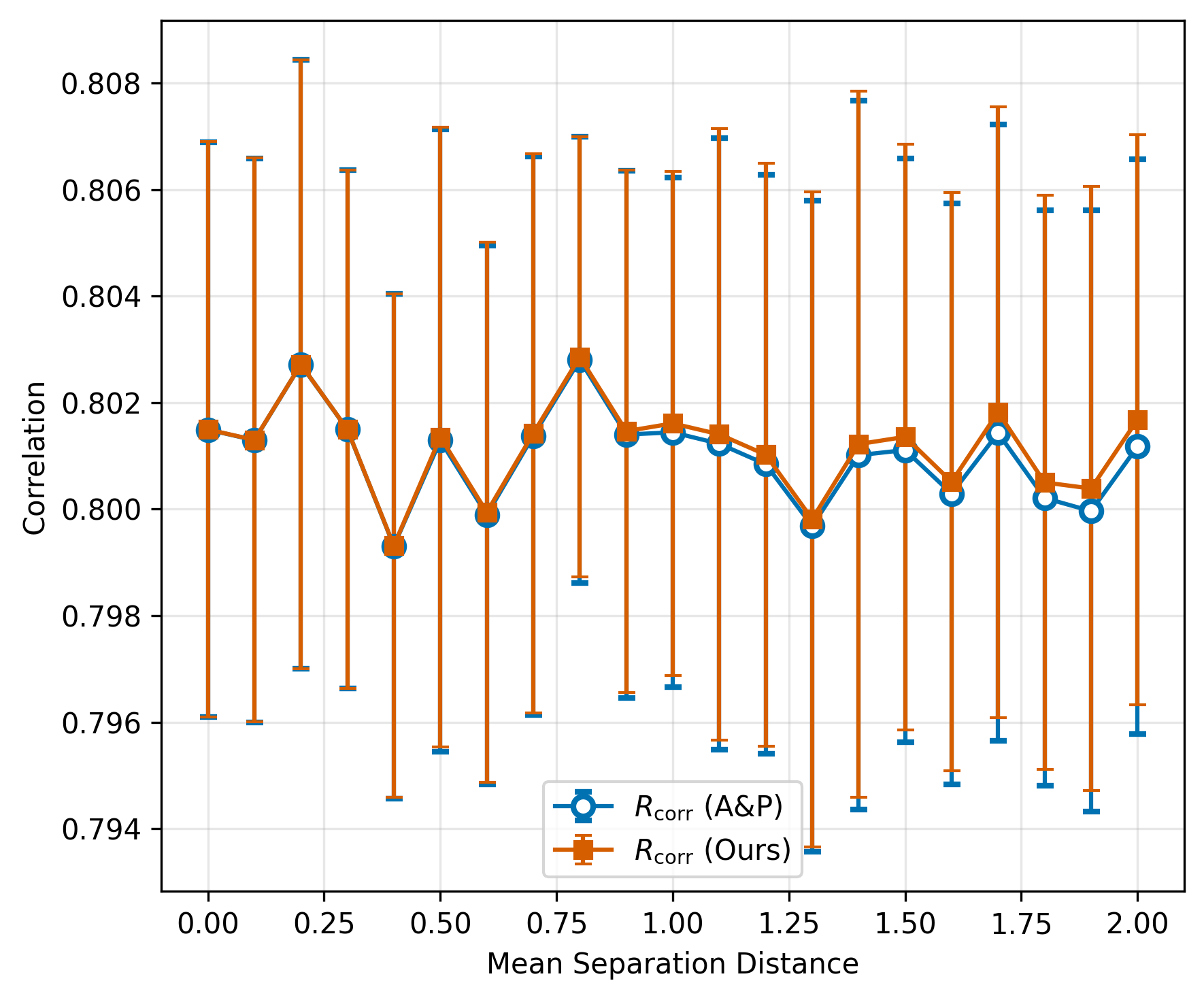}&
        \includegraphics[width=0.32\linewidth]{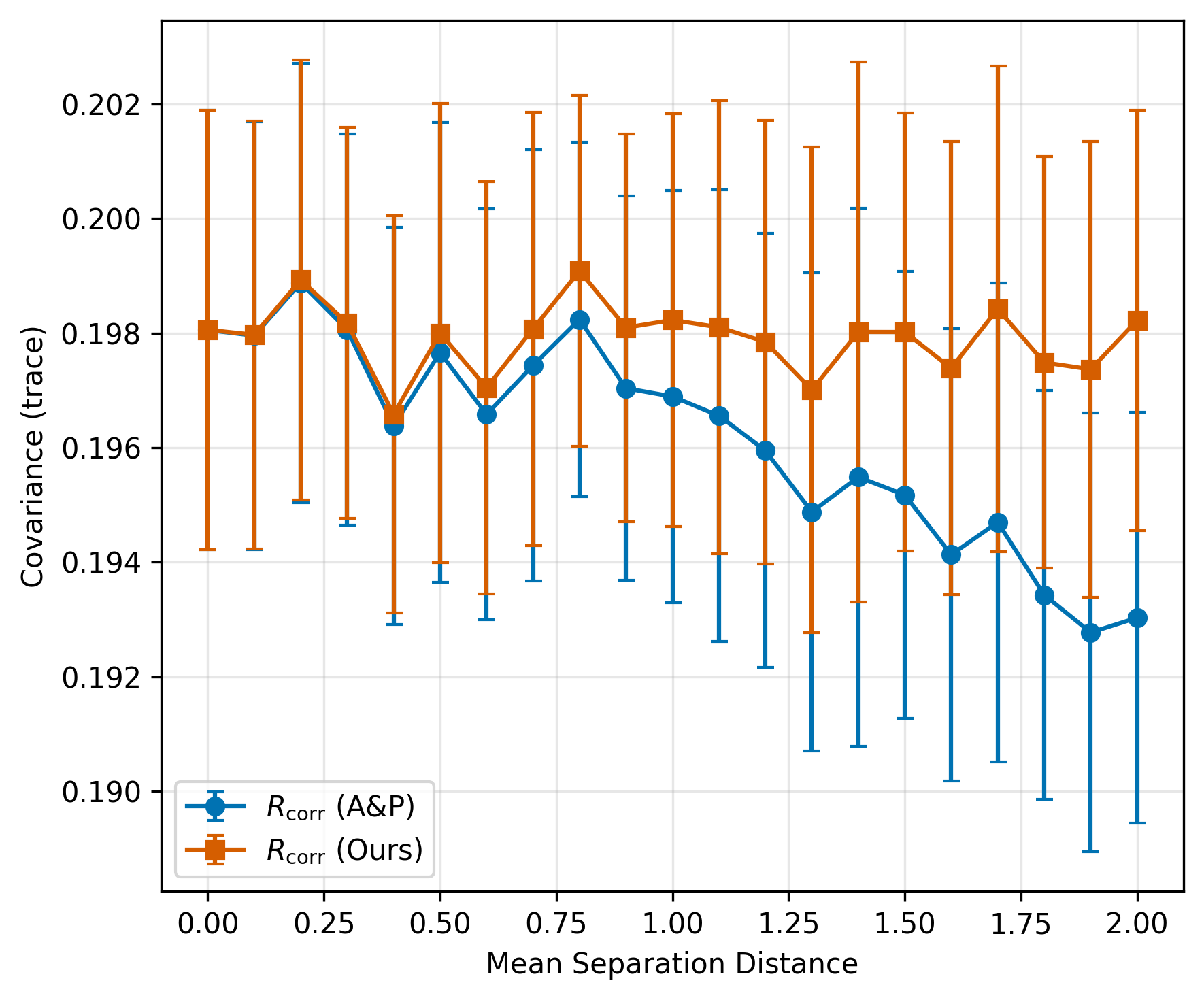}&
        \includegraphics[width=0.32\linewidth]{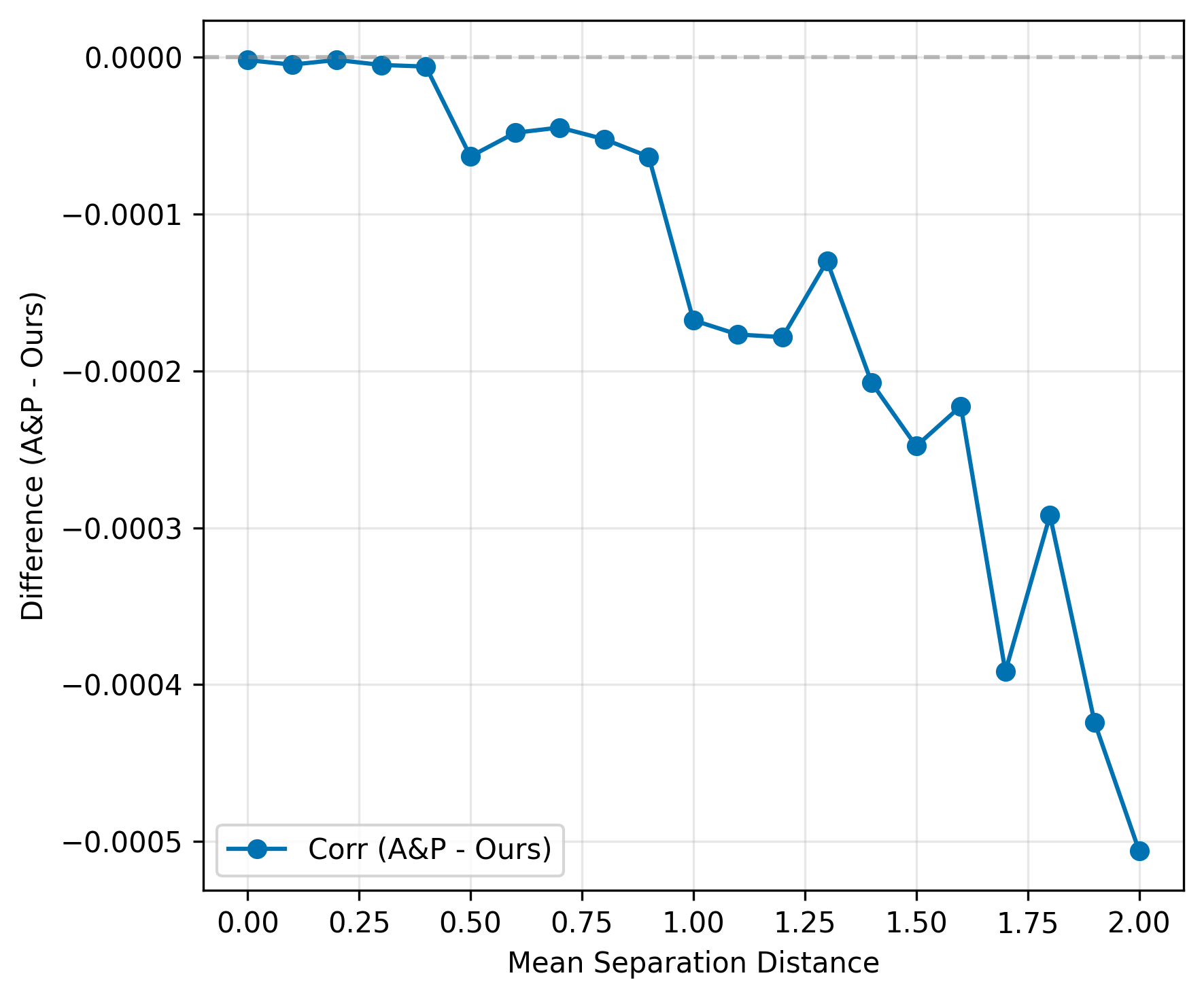}
    \end{tabular}
    \caption{Effect of Mean Separation on Estimator Performance. \textbf{Left:} Correlation estimates versus geodesic distance $d(\mu_X, \mu_Y)$. A\&P exhibits monotonic bias, while Ours remains stable ($0.798 \pm 0.012$, with no significant trend). \textbf{Middle:} Covariance estimates. \textbf{Right:} Difference (A\&P - Ours) highlighting systematic divergence. At $d=2.0$, A\&P underestimates by $0.23 \pm 0.05$. Intrinsic correlation fixed at $\rho=0.8$. Parameters: $n=100, d=5, \tau_X=\tau_Y=0.5$, MC $=50$ per distance.}
    \label{fig:spd_mean_separation}
\end{figure}

\subsection{Real Case Studies}\label{ss:KSSResults}

\subsubsection{Sunnybrook Cardiac MRI dataset}
Here we include an additional figure for our experimental results on Sunnybrook Cardiac MRI dataset (\url{https://www.cardiacatlas.org/sunnybrook-cardiac-data/}). Figure \ref{fig:shape_visual} displays examples of the inner and outer ventricular walls for the four groups. The landmarks of each shape are connected with a solid line to display the shape.
\begin{figure}[h]
    \centering
    \includegraphics[width=0.95\linewidth]{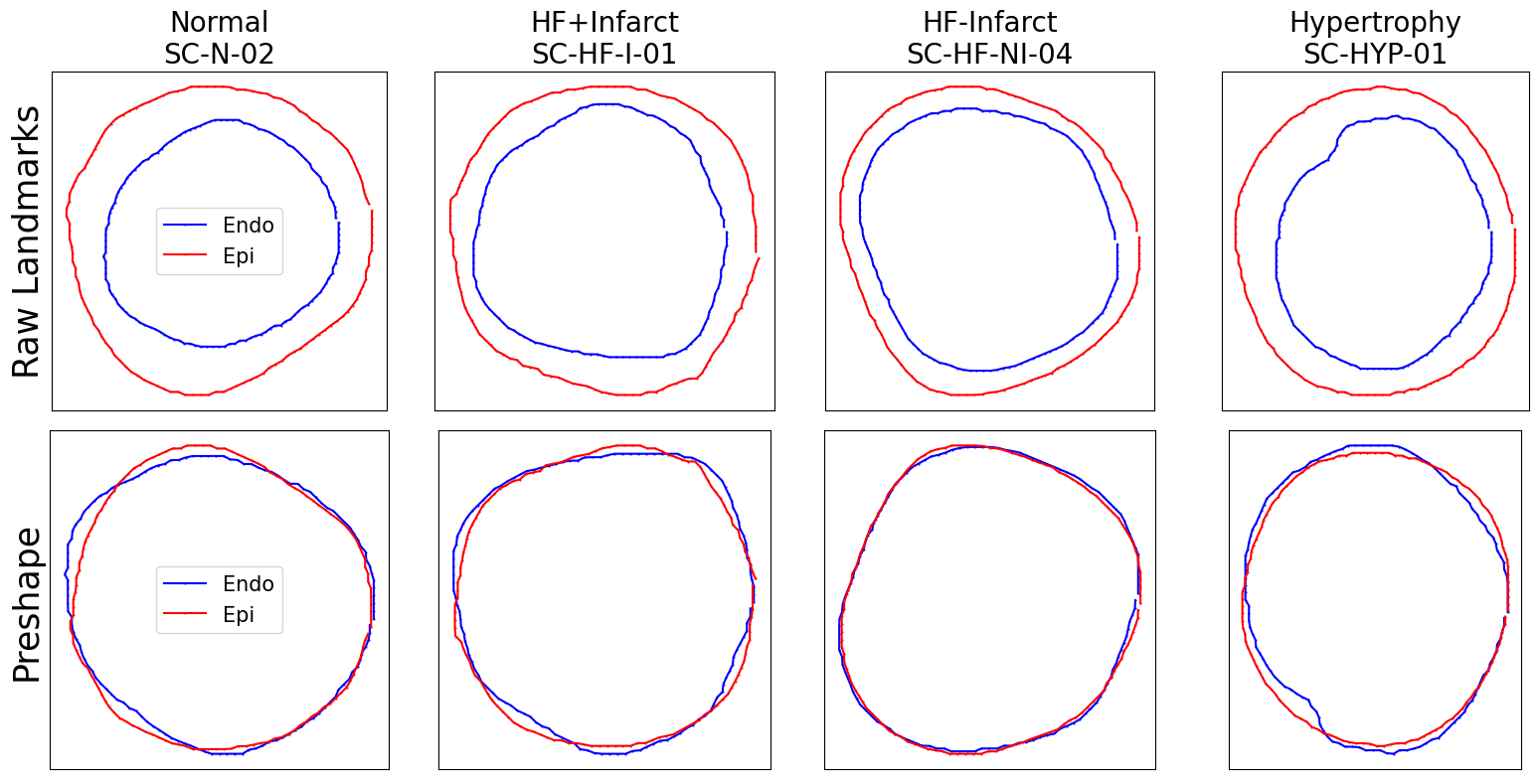}
    \caption{Illustrative examples of inner and outer ventricular walls from the four groups. \textbf{Top:} raw landmarks extracted from cardiac MRI. \textbf{Bottom:} preprocessed shapes in preshape space.}
    \label{fig:shape_visual}
\end{figure}

\subsubsection{Schizophrenia Challenge dataset}

We use the MLSP 2014 Schizophrenia Challenge dataset \citep{mlsp-2014-mri}, available via \texttt{geomstats} \citep{miolane2020geomstats}: 86 subjects (46 healthy, 40 schizophrenic), each with a $28\times28$ SPD connectivity matrix. For a simple low-dimensional illustration, we use one pair of non-overlapping $4\times4$ principal sub-matrices as $X,Y\in\mathcal{P}_4$ and compute the Riemannian correlation within each group ($n=46,40$; 3--5$\times$ larger than Sunnybrook).

\textit{Result 1: Footpoint invariance (Fig.~\ref{fig:conn_fp}).}
Across 15 footpoints along the geodesic ($t\in[0,2]$), our estimator is exactly constant, whereas A\&P varies systematically. The absolute variation is small on this dataset (up to $0.003$) because the group means are close, but the bias grows with mean separation (up to $0.23$ in Figure~\ref{fig:spd_mean_separation}). The key point is methodological: a correlation should not depend on an arbitrary analyst choice of footpoint.

\begin{figure}[h]
    \centering
    \includegraphics[width=0.95\linewidth]{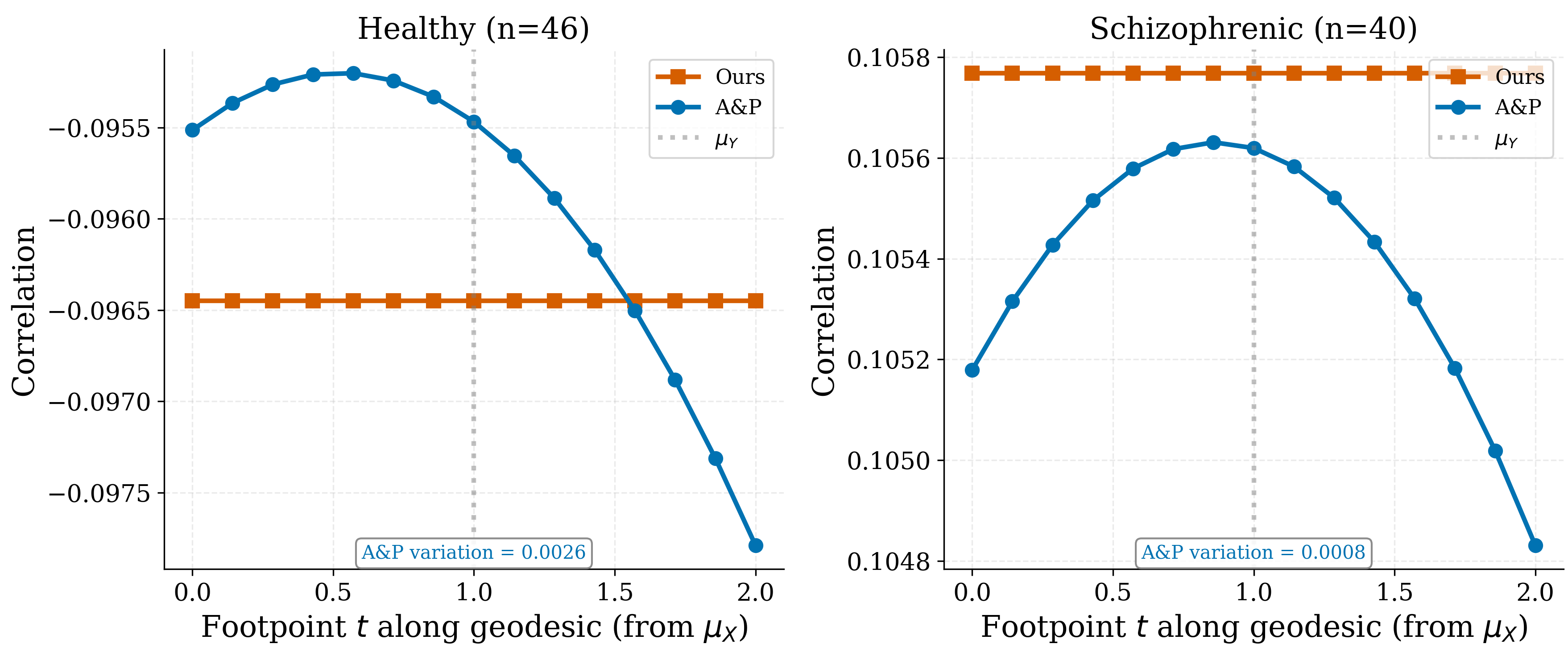}
    \caption{Footpoint invariance on connectome data. Our method (orange) is constant; A\&P (blue) varies with the footpoint.}
    \label{fig:conn_fp}
\end{figure}

\textit{Result 2: Group differentiation (Fig.~\ref{fig:conn_ci}).}
The healthy group yields $\hat{\mathcal{R}}=-0.096$ and the schizophrenic group $\hat{\mathcal{R}}=+0.106$, giving a sign reversal with $\Delta\hat{\mathcal{R}}=0.20$. A two-sample permutation test confirms this group difference ($p=0.032$, $1000$ permutations). Bootstrap 95\% CIs ($B=1000$) are shown in Fig.~\ref{fig:conn_ci}.

\begin{figure}[h]
    \centering
    \includegraphics[width=0.65\linewidth]{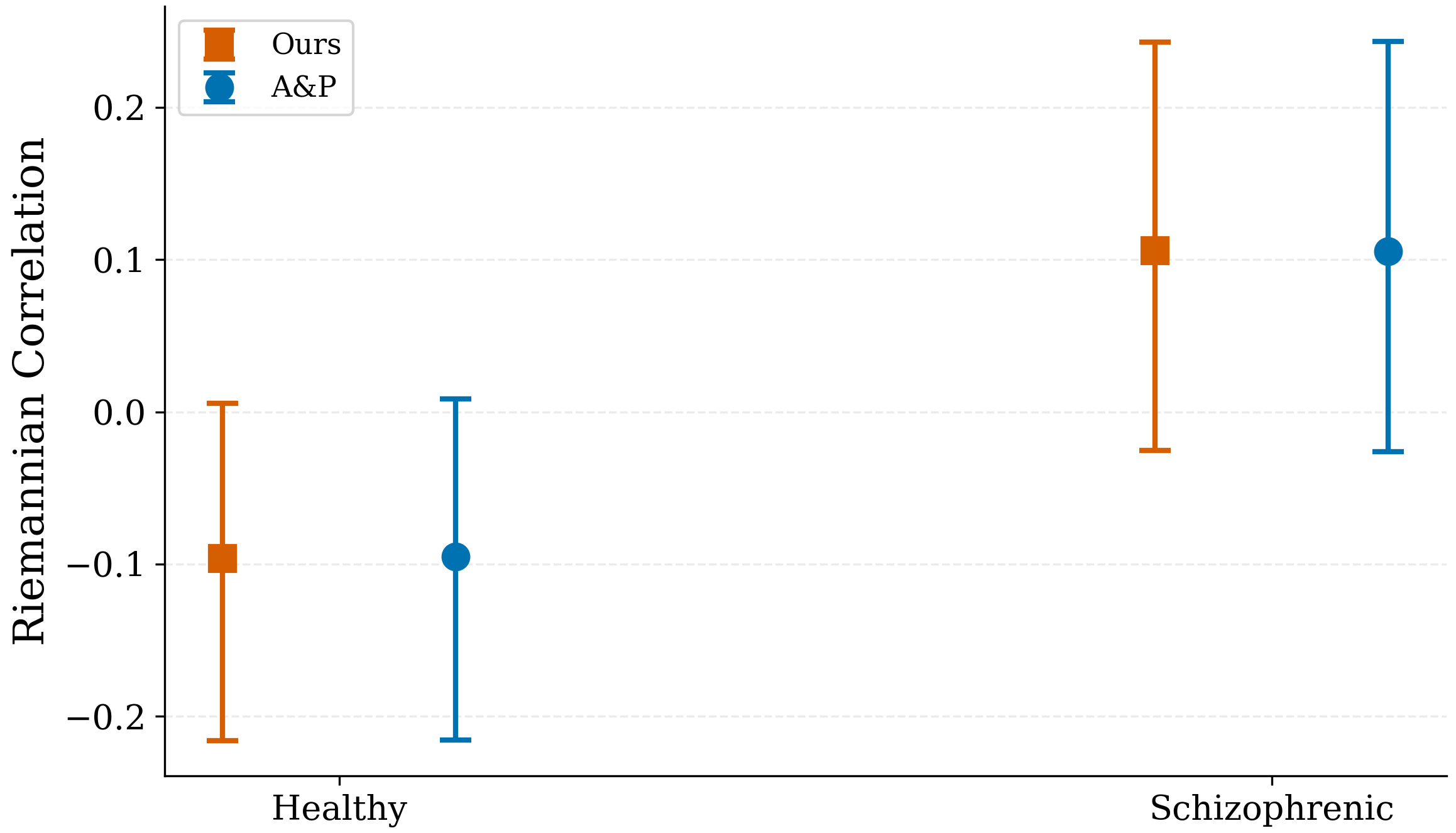}
    \caption{Bootstrap 95\% CI for Riemannian correlation. Sign reversal between groups is significant ($p=0.032$, two-sample permutation test).}
    \label{fig:conn_ci}
\end{figure}

\end{document}